\documentclass[11pt]{amsart}

\usepackage{amsmath,amssymb,amsthm,mathtools,mathrsfs}
\usepackage{booktabs,tabularx}
\usepackage{enumitem}
\usepackage{needspace}
\usepackage{xcolor}
\usepackage[colorlinks=true,hyperindex]{hyperref}
\usepackage{tikz-cd}
\newcommand\myshade{85}
\colorlet{mylinkcolor}{violet}
\colorlet{mycitecolor}{red}
\colorlet{myurlcolor}{cyan}
\hypersetup{
  linkcolor = mylinkcolor!\myshade!black,
  citecolor = mycitecolor!\myshade!black,
  urlcolor  = myurlcolor!\myshade!black,
  colorlinks = true,
  pdftitle={Derived Picard reconstruction for graded surfaces with stops},
  pdfauthor={Wen Chang and Dongjian Wu}
}

\numberwithin{equation}{section}
\newcommand{\kk}{\mathbf{k}}
\newcommand{\Z}{\mathbb Z}
\newcommand{\W}{\mathcal W}
\newcommand{\Perf}{\operatorname{per}}
\newcommand{\DPic}{\operatorname{DPic}}
\newcommand{\Mod}{\operatorname{Mod}}
\newcommand{\PMod}{\operatorname{PMod}}
\newcommand{\MCGgr}{\operatorname{MCG}_{\mathrm{gr}}}
\newcommand{\Stab}{\operatorname{Stab}}
\newcommand{\rk}{\operatorname{rk}}

\newcommand{\prk}{\operatorname{prk}}
\newcommand{\rvz}{\operatorname{rvz}}
\newcommand{\Div}{\operatorname{Div}}

\newtheorem{theorem}{Theorem}[section]
\newtheorem*{theoremA}{Theorem A}
\newtheorem*{theoremB}{Theorem B}
\newtheorem*{theoremC}{Theorem C}
\newtheorem*{theoremD}{Theorem D}
\newtheorem*{theoremE}{Theorem E}
\newtheorem{proposition}[theorem]{Proposition}
\newtheorem{lemma}[theorem]{Lemma}
\newtheorem{corollary}[theorem]{Corollary}
\theoremstyle{definition}
\newtheorem{definition}[theorem]{Definition}
\newtheorem{example}[theorem]{Example}
\newtheorem{remark}[theorem]{Remark}

\title[Derived Picard reconstruction]
{Derived Picard reconstruction for graded surfaces with stops}

\author{Wen Chang}
\address{School of Mathematics and Statistics,
Shaanxi Normal University,
Xi'an 710062, China}
\email{changwen161@163.com}

\author{Dongjian Wu}
\address{Shanghai Institute for Mathematics and
Interdisciplinary Sciences (SIMIS),
Shanghai 200433, China \& Research Institute of Intelligent Complex Systems, Fudan University, Shanghai 
200433, China}
\email{wdj@simis.cn}

\date{}
\subjclass[2020]{16E35, 18G80, 53D37, 57K20}
\keywords{Derived Picard group, gentle algebra, partially wrapped Fukaya
category, graded surface, Serre functor, Grothendieck group}

\begin{document}
\flushbottom

\begin{abstract}
We study reconstruction of graded surfaces with stops $(S,M,\eta)$
from their mapping class groups and the derived Picard groups of
their partially wrapped Fukaya categories.
In genus at least two, we prove that the ungraded mapping class
group determines the surface with stops $(S,M)$, while the derived
Picard group with the shift distinguished determines the underlying
graded surface $(S,\eta)$. Nevertheless, the derived Picard group with both the shift and the
Serre functor distinguished has countably infinite fibres.
We also construct arbitrarily large finite families of pairwise
non-equivalent categories with the same invariant and Grothendieck rank.

For line fields induced by nowhere-vanishing vector fields, we show
that the rational Serre representation determines the surface with
stops $(S,M)$ in every genus. For arbitrary line fields in genus at
least two, the full integral derived Picard representation, with
the shift distinguished, has fibres of cardinality at most four.
It completely determines the graded surface with stops $(S,M,\eta)$ when $S$
has one or at least five boundary components, while non-singleton
fibres occur for two, three, and four boundary components.
\end{abstract}

\maketitle

\tableofcontents
% --------------------------------------------------------------------------
\section{Introduction}

The partially wrapped Fukaya category of a graded
surface with stops was introduced in \cite{HKK}. It is proved that the perfect
derived category of any homologically smooth graded gentle algebra is
triangle equivalent to a partially wrapped Fukaya category of a graded
surface with stops \cite{HKK,LekiliPolishchuk}. 
In particular, the geometric models describe the derived categories of gentle algebras in
terms of curves and intersections on surfaces. In this model, graded
arcs represent homotopy string objects, closed curves with local systems
represent homotopy band objects, and rotation of the endpoints of arcs
describes Auslander--Reiten translation
\cite{OpperPlamondonSchroll}. Related surface models describe the module
categories of gentle algebras \cite{BC21} and the algebraic hearts in
their derived categories \cite{ChangHearts}. For a recent overview of these geometric models
and their role in representation theory, with particular emphasis on
graded gentle algebras and their connections with Fukaya categories of
surfaces, see \cite{Schroll2026}.

For the class of categories considered here, the graded surface with
stops is a complete geometric invariant of the corresponding partially
wrapped Fukaya category, or equivalently of the perfect derived category
of a graded gentle algebra. Lekili and Polishchuk
\cite{LekiliPolishchuk} classified mapping class group orbits of line
fields in terms of winding numbers and Arf invariants and, using the
surface model of \cite{HKK}, obtained a sufficient criterion for derived
equivalence of homologically smooth graded gentle algebras. For ordinary
gentle algebras, completeness of the geometric invariant was proved in
\cite{AmiotPlamondonSchroll} using dissections and silting theory, and in
\cite{Opper2019} through the geometric study of triangle equivalences.
Combined with the classification of mapping class group orbits of line
fields in \cite{LekiliPolishchuk}, these geometric classification
results yield complete numerical derived invariants; see
\cite[Theorem~7.4]{AmiotPlamondonSchroll} and the discussion following
Theorem~B in \cite{Opper2019}. For homologically
smooth graded gentle algebras, completeness was established in
\cite{JinSchrollWang}, confirming the conjecture in
\cite[Remark~3.19]{LekiliPolishchuk}. Thus two such perfect derived
categories are equivalent if and only if their associated graded
surfaces with stops are related by an orientation-preserving
diffeomorphism identifying their stop sets and the homotopy classes of
their line fields.

This leads to a natural inverse problem.

\textbf{How much of a graded surface with stops can be recovered
from its own symmetries or those of its partially wrapped
Fukaya category?}

On the geometric side, we consider both ungraded and graded mapping
class groups, while on the categorical side we consider derived
Picard groups. The geometrisation theorem
\cite[Corollary~10.6]{Opper2025}
provides a split surjection from the derived Picard group to the graded
mapping class group. However, an abstract isomorphism of derived Picard
groups does not itself identify either the objects of the categories or
the curves on the associated surfaces, and it need not preserve the
distinguished elements represented by the shift and the Serre functor.
We are therefore led to compare three levels of data: the derived Picard
group alone, the group with the shift distinguished, and the group with
both the shift and the Serre functor distinguished. We also study the
induced action on the Grothendieck group, either for the Serre functor
alone or for the entire derived Picard group.

Throughout the paper, \(\kk\) is an algebraically closed field of
characteristic zero, and all categories and equivalences are
\(\kk\)-linear. Our surfaces are compact, connected, and oriented, with
non-empty boundary. A finite set of stops \(M\subset\partial S\) meets
every boundary component of a surface $S$, and a line field \(\eta\) gives the grading. We
write
\[
\mathbb S=(S,M,\eta),\qquad
\partial S=B_1\sqcup\cdots\sqcup B_b,\qquad
m_i=|M\cap B_i|.
\]
The unordered tuple \(\{m_1,\ldots,m_b\}\), counted with multiplicity, is
called the \emph{stop distribution}. Forgetting the line field gives
the \emph{underlying surface with stops} \((S,M)\). Forgetting the stops
gives the \emph{underlying graded surface} \((S,\eta)\). An isomorphism
of graded surfaces with stops is an orientation-preserving
diffeomorphism \(f\colon(S,M)\to(S',M')\) satisfying
\(f^*\eta'\simeq\eta\). When counts are assigned to specified boundary
components, we retain this assignment in addition to the unordered
stop distribution.

We write \(\W(\mathbb S)\) for the partially wrapped Fukaya category of
\(\mathbb S\) and set
\[
D(\mathbb S)\coloneqq \DPic\bigl(\W(\mathbb S)\bigr),
\qquad
K(\mathbb S)\coloneqq K_0\bigl(\W(\mathbb S)\bigr).
\]
Here the derived Picard group is taken with respect to the standard dg
enhancement.
Let \(\MCGgr(\mathbb S)\) be the graded
mapping class group of \(\mathbb S\), and denote by
\[
\Psi_{\mathbb S}\colon
D(\mathbb S)\longrightarrow\MCGgr(\mathbb S)
\]
Opper's geometrisation homomorphism \cite[Corollary~10.6]{Opper2025}.

We first consider the reconstruction problem without a line field.
Let \(\Mod(S,M)\) denote the group of isotopy classes of
orientation-preserving diffeomorphisms of \(S\) preserving \(M\) setwise,
with isotopies subject to the same condition. Boundary components
may be permuted, and a mapping class preserving a boundary component
may cyclically permute its stops. Our first result recovers the
surface and its stop distribution from this abstract group.

\begin{theoremA}[Theorem~\ref{thm:ungraded-reconstruction}]
Let \((S,M)\) and \((S',M')\) be surfaces with stops of genus at least
two, with at least one stop on every boundary component. An abstract group
isomorphism
\[
 \Mod(S,M)\simeq\Mod(S',M')
\]
implies that \((S,M)\) and \((S',M')\) are isomorphic as surfaces with
stops.
\end{theoremA}

Rigidity of mapping class groups is closely related to automorphisms
of curve complexes \cite{Ivanov1997,Korkmaz,Luo}. For the punctured
surfaces obtained by capping our boundary components, we use the
finite-index rigidity theorem of \cite{BehrstockMargalit}.
The additional step is to recover the boundary stops. We then identify
the subgroup generated by one-step boundary rotations intrinsically.
After restricting to mapping classes preserving each boundary
component, capping gives a central extension of a pure mapping class
group by this rotation lattice. Finally, the stop numbers are recovered from
the extension class.

Example~\ref{ex:low-genus-ungraded-failure} gives pairs of surfaces
with stops in genera zero and one with different stop distributions
but isomorphic mapping class groups, which shows that the genus assumption cannot be omitted.

For a graded surface \(\mathbb S=(S,M,\eta)\), the group obtained from 
\(\MCGgr(\mathbb S)\) by quotienting out the grading shift is the
line-field stabilizer
\[
  H_\eta=\Stab_{\Mod(S,M)}(\eta),
\]
rather than the full group \(\Mod(S,M)\). A Dehn twist about a
nonseparating curve belongs to \(H_\eta\) precisely when the curve
has winding number zero. Thus Theorem~A cannot be applied directly
to the graded mapping class group supplied by geometrisation.
Nevertheless, the derived Picard group of the partially wrapped
Fukaya category recovers the following data.

\begin{theoremB}[Theorem~\ref{thm:derived-reconstruction}]
Let \(\mathbb S=(S,M,\eta), \mathbb S'=(S',M',\eta')\) be graded surfaces with stops of genus at least two whose partially wrapped Fukaya categories are smooth and proper.
\begin{enumerate}[label=\textup{(\roman*)}]
\item
If \(D(\mathbb S)\simeq D(\mathbb S')\), then \(g(S)=g(S')\), and \(|\pi_0(\partial S)|=|\pi_0(\partial S')|\).
\item A shift-preserving isomorphism
\(D(\mathbb S)\simeq D(\mathbb S')\) implies \((S,\eta)\simeq(S',\eta')\).
\end{enumerate}
\end{theoremB}

For part~\textup{(i)}, we first identify the geometrisation kernel
\(N(\mathbb S)=\ker\Psi_{\mathbb S}\) as the subgroup of
\(D(\mathbb S)\) generated by all divisible abelian subgroups.
In particular, this kernel is characteristic. The torsion of
its abelianization consists of finite-order rank-one local systems
on \(S\). Writing \(g=g(S)\) and \(b=|\pi_0(\partial S)|\), we show that
the \(\ell\)-Pr\"ufer rank of \(D(\mathbb S)\) is \(2g+b-1\) for every
prime \(\ell\), while the reduced virtual-centre rank of
\(D(\mathbb S)/N(\mathbb S)\) is \(b+1\). The latter rank accounts
for the boundary twists and the grading shift. These two
abstract-group invariants determine \(g\) and \(b\).

For part~\textup{(ii)}, conjugation on the \(\ell\)-primary torsion
of \(N(\mathbb S)^{\mathrm{ab}}\) detects the kernel of the integral
homological action of \(H_\eta\). Except for the one-holed genus-two
surface, this kernel allows us to recognize powers of separating
Dehn twists after quotienting out the boundary rotations. We then
apply the rigidity theorem for the separating-curve complex
\cite[Theorem~1.2]{Kida}. To recover the grading, we show that the stabilizer
of a line field in the pure mapping class group of the surface
without stops determines its homotopy class. Weighted
products of twists about genus-one bounding pairs determine the
winding numbers of their curves, and these winding numbers
determine the line field. The case \(\Sigma_{2,1}\) is treated separately
using the action on \(H^1(S;\Z/2)\). Related homological
representations for vector framings are studied in
\cite{CalderonSalterRelative}.

The conclusion of Theorem~B does not include the stops.
Since Auslander--Reiten translation is given geometrically by
simultaneous endpoint rotation along the boundary, one might ask
whether distinguishing the Serre functor as well recovers this
information. For a graded surface \(\mathbb S\) whose partially
wrapped Fukaya category is homologically smooth and proper, put
\[
  \mathfrak D(\mathbb S)
  =
  \bigl(D(\mathbb S),[1],\mathsf S_{\mathbb S}\bigr),
\]
where \(\mathsf S_{\mathbb S}\) denotes the class of the Serre functor
of \(\W(\mathbb S)\). An isomorphism of these triples is a group
isomorphism preserving both distinguished elements, and we take fibres
of \(\mathfrak D\) on isomorphism classes of graded surfaces with
stops.

\begin{theoremC}[Theorem~\ref{thm:countable-fibres}]
For every \(g\geq2\), the invariant \(\mathfrak D\) has a countably
infinite fibre on graded surfaces with stops of genus \(g\). The surfaces
in this fibre have the same underlying graded surface, while their smooth
and proper partially wrapped Fukaya categories are pairwise non-equivalent.
Moreover, for every \(R\geq1\), some fibre contains more than \(R\) such
categories with the same Grothendieck rank.
\end{theoremC}

The construction uses a vector-framed line field with boundary winding
numbers \((2,-2,-2(2g+1))\). The associated derived Picard groups have a
common central rank-two rotation factor. In suitable coordinates, the
Serre elements are represented by primitive vectors, so an integral
change of basis identifies them. Varying the stop counts produces the
infinite fibre, while keeping their total fixed gives the finite families
of common Grothendieck rank.

Theorem~C shows that distinguishing the Serre functor does not
extend the conclusion of Theorem~B to the stops. The genus
restriction in Theorem~B is also essential.
In Example~\ref{ex:low-genus-failure}, we construct, in each of
genera zero and one, two graded surfaces with stops having
isomorphic invariants \(\mathfrak D\) but non-isomorphic underlying
graded surfaces.

We now pass from group invariants to linear actions on the
Grothendieck group. Boundary stop numbers and winding numbers are
recorded by the Avella--Alaminos--Geiss (AG) invariant
\cite{AvellaAlaminosGeiss2008,LekiliPolishchuk}. For ordinary gentle
algebras, its relation to the Coxeter polynomial is described in
\cite{JimenezGonzalezMroz}. These boundary data also enter the
formulas for Serre entropy in \cite{ChangElaginSchroll}.
We then consider whether the rational Grothendieck group equipped with
the induced Serre action determines the underlying surface with
stops under the assumption that the line field is induced by a nowhere-vanishing
vector field.

\begin{theoremD}[Theorem~\ref{thm:serre-reconstruction}]
Let \(\mathbb S=(S,M,\eta)\) and
\(\mathbb S'=(S',M',\eta')\) be graded surfaces with stops whose
partially wrapped Fukaya categories are smooth and proper. Suppose that
\(\eta\) and \(\eta'\) are induced by nowhere-vanishing vector fields.
An isomorphism of their rational Grothendieck groups which intertwines
the Serre actions implies \((S,M)\simeq(S',M')\).
\end{theoremD}

Note that there is no genus restriction in this theorem. On
\(K(\mathbb S)\otimes_{\Z}\mathbb Q\), put
\[
  \mathsf C=(\mathsf S_{\mathbb S}[-1])_*
  =-(\mathsf S_{\mathbb S})_*.
\]
We prove that \(\mathsf C\) has a unipotent power. If
\(n_{\mathsf C}\) is the least positive integer for which
\(\mathsf C^{n_{\mathsf C}}\) is unipotent, then
\[
  b=1+\rk_{\mathbb Q}
  \bigl(\mathsf C^{n_{\mathsf C}}-I\bigr).
\]
Together with \(b\), the characteristic polynomial of \(\mathsf C\)
determines the stop distribution, and \(\rk K(\mathbb S)\) then
determines the genus. Theorem~D applies in particular to an integral isomorphism
intertwining the Serre actions. Its conclusion recovers \((S,M)\),
not the line field.

We next retain the action of the entire derived Picard group and
allow arbitrary line fields. Write
\[
  \varrho_{\mathbb S}\colon D(\mathbb S)\longrightarrow
  \operatorname{Aut}\bigl(K(\mathbb S)\bigr),
  \qquad
  \mathfrak R(\mathbb S)=
  \bigl(D(\mathbb S),[1],K(\mathbb S),\varrho_{\mathbb S}\bigr),
\]
where $\varrho_{\mathbb S}$ is the natural representation. An isomorphism of these invariants is a pair \((\Phi,\Theta)\),
where \(\Phi\colon D(\mathbb S)\to D(\mathbb S')\) is a group
isomorphism with \(\Phi([1])=[1]\), and
\(\Theta\colon K(\mathbb S)\to K(\mathbb S')\) is an integral
isomorphism satisfying
\[
  \Theta\,\varrho_{\mathbb S}(x)
  =\varrho_{\mathbb S'}\bigl(\Phi(x)\bigr)\,\Theta
  \qquad\text{for every }x\in D(\mathbb S).
\]
Unlike \(\mathfrak D\), the invariant \(\mathfrak R\) does not
distinguish the Serre element. Thus the above compatibility does
not require \(\Theta\) to intertwine the two Serre actions.

For the next theorem, fibres are taken within the class of graded
surfaces with stops of genus at least two whose partially wrapped
Fukaya categories are smooth and proper, modulo graded isomorphism.

\begin{theoremE}[Theorem~\ref{thm:general-boundary-reconstruction}]
For graded surfaces with stops of genus at least two whose partially
wrapped Fukaya categories are smooth and proper, the invariant
\(\mathfrak R\) determines the underlying graded surface and the total
number of stops. Let \(b=|\pi_0(\partial S)|\), and let \(t\) be the
number of boundary components with odd winding. Its fibre on isomorphism
classes of graded surfaces with stops satisfies
\[
 \#\mathfrak R^{-1}([\mathfrak R(\mathbb S)])\leq
 \begin{cases}
 4,&b=4,\ t=4,\\
 2,&b=2,\ \text{or }(b=3,\ t=2),\\
 1,&\text{otherwise}.
 \end{cases}
\]
In particular, \(\mathfrak R\) determines the complete graded surface
with stops when \(b=1\) or \(b\geq5\).
\end{theoremE}

The genus hypothesis is essential. Example~\ref{ex:represented-low-genus-failure} gives examples in genera zero and
one with isomorphic full represented invariants but non-isomorphic
underlying graded surfaces.

The main step is to match
the boundary blocks of relative homology with their geometric
boundaries. Separating twists control the absolute homological
intertwiner, and integral gluing of the endpoint blocks gives the
remaining parity constraints. Theorem~\ref{thm:general-boundary-reconstruction}
records the finer bounds for three and four boundaries.
The counterexamples in Examples~\ref{ex:two-boundary-represented-failure},
\ref{ex:three-boundary-represented-failure}, and
\ref{ex:four-boundary-represented-failure} concern the allocation of
stops to boundaries with different winding numbers. They preserve the underlying graded
surface and the unordered stop distribution, and do not assert
preservation of the actual Serre element or the Euler pairing.

Table~\ref{tab:symmetry-reconstruction} summarizes the information
recovered from the different symmetry data and the limitations of
reconstruction.

\begin{table}[htbp]
\centering
\caption{Reconstruction from symmetry data}
\label{tab:symmetry-reconstruction}
\small
\renewcommand{\arraystretch}{1.12}
\setlength{\tabcolsep}{5pt}
\begin{tabularx}{\textwidth}{@{}
  >{\raggedright\arraybackslash}p{0.39\textwidth}
  >{\raggedright\arraybackslash}X
  >{\raggedright\arraybackslash}p{0.17\textwidth}
@{}}
\toprule
\textbf{Symmetry data}
& \textbf{Reconstruction and limitations}
& \textbf{Result} \\
\midrule

Ungraded mapping class group\newline
$\Mod(S,M)$
& Surface with stops $(S,M)$
& \hyperref[thm:ungraded-reconstruction]{Theorem~A} \\
\addlinespace[5pt]

Derived Picard group\newline
$D$
& Underlying surface $S$:
  genus $g$ and number $b$ of boundary components
& \hyperref[thm:derived-reconstruction]{Theorem~B(i)} \\
\addlinespace[5pt]

Shift-pointed derived Picard group\newline
$(D,[1])$
& Underlying graded surface $(S,\eta)$
& \hyperref[thm:derived-reconstruction]{Theorem~B(ii)} \\
\addlinespace[5pt]

Shift- and Serre-pointed derived Picard group\newline
$\mathfrak D(\mathbb S)=(D,[1],\mathsf S_{\mathbb S})$
& Stops are not determined in general;
  countably infinite fibres occur
& \hyperref[thm:countable-fibres]{Theorem~C} \\
\addlinespace[5pt]

Rational Serre representation\newline
$(K\otimes_{\Z}\mathbb Q,(\mathsf S_{\mathbb S})_*)$
& Surface with stops $(S,M)$
& \hyperref[thm:serre-reconstruction]{Theorem~D} \\
\addlinespace[5pt]

Full integral derived Picard representation\newline
$\mathfrak R(\mathbb S)=(D,[1],K,\varrho_{\mathbb S})$
& $(S,\eta)$ and $|M|$; at most four possibilities for
  $\mathbb S$. Complete reconstruction if $b=1$ or $b\geq5$
& \hyperref[thm:general-boundary-reconstruction]{Theorem~E} \\

\bottomrule
\end{tabularx}

\smallskip
\begin{minipage}{\textwidth}
\footnotesize
Here $\mathbb S=(S,M,\eta)$, $D=\DPic\bigl(\W(\mathbb S)\bigr)$, $K=K(\mathbb S)$. All rows are subject to the standing
hypotheses, including smoothness and properness for the categorical
statements. Except in Theorem~D, the genus is at least two.
Theorem~D holds in every genus for line fields induced by
nowhere-vanishing vector fields.
\end{minipage}
\end{table}

The paper is organized as follows. Section~\ref{sec:preliminaries}
recalls line fields, graded mapping class groups, and the surface model
for derived Picard groups. Section~\ref{sec:ungraded} reconstructs the
marked surface from its mapping class group. Section~\ref{sec:derived-reconstruction}
recovers the underlying graded surface from the shift-pointed derived
Picard group. Section~\ref{sec:fibres} constructs the infinite fibres of
the shift- and Serre-pointed invariant. Section~\ref{sec:grothendieck-actions}
first reconstructs the marked surface from the Serre action for
vector-framed line fields. It then studies the full integral derived
Picard representation, bounds its fibres, and constructs
non-singleton examples. Appendix~\ref{app:mapping-class-gradings}
supplies the mapping-class-group and relative-grading arguments.
Appendix~\ref{app:serre-family} contains the stabilizer and Serre
calculations for the three-boundary family, and
Appendix~\ref{app:twisted-homology} develops the twisted-homology
and integral boundary-matching arguments.

\section{Preliminaries}\label{sec:preliminaries}
In this section, we recall the basic facts about line fields, graded mapping
class groups, and surface categories used below.  We follow
\cite{HKK,LekiliPolishchuk,CalderonSalter} for line fields and surface models,
\cite{Opper2025} for derived Picard groups, and
\cite{OpperPlamondonSchroll} for the Serre functor.
% --------------------------------------------------------------------------

\subsection{Line fields and winding numbers}

Let \(S\) be a compact connected oriented surface with non-empty boundary,
and let \(M\subset\partial S\) be a finite set meeting every boundary
component. The elements of \(M\) are called \emph{stops}. A line field is
a section
\(
 \eta:S\to\mathbb P(TS)
\)
of the bundle of unoriented tangent lines, considered up to homotopy. A
\emph{graded surface with stops} is a triple
\(\mathbb S=(S,M,\eta)\).

The line field determines a \(\Z\)-cover of \(\mathbb P(TS)\), whose fibre
over a line \(\ell\subset T_xS\) consists of homotopy classes, relative to
the endpoints, of paths in \(\mathbb P(T_xS)\) from \(\eta(x)\) to
\(\ell\). Explicitly, it is given by
\begin{equation}\label{eq:grading-cover}
 \widetilde{\mathbb P(TS)}_\eta
 \coloneq\left\{(x,\ell,[\gamma])\,\middle|\,
   \gamma(0)=\eta(x),\ \gamma(1)=\ell\right\}.
\end{equation}
Concatenation with the positive generator of
\(\pi_1(\mathbb P(T_xS))\simeq\Z\) gives the deck transformation. A
grading of an immersed curve is a lift of its tangent-line map to
\eqref{eq:grading-cover}.

We now recall the definition of the winding number from
\cite[Definition~1.3]{LekiliPolishchuk}. By
Poincar\'e--Lefschetz duality, the image of a line field determines a
class
\[
 [\eta]\in H^1(\mathbb P(TS);\Z)
\]
whose restriction to every oriented fibre of
\(\mathbb P(TS)\to S\) is the positive generator. If
\(c:S^1\to S\) is an oriented immersed closed curve, its tangent-line
lift is the map
\[
 \widetilde c:S^1\longrightarrow\mathbb P(TS),
 \qquad
 \widetilde c(t)=\bigl(c(t),\mathbb R\dot c(t)\bigr).
\]
The \emph{winding number} of \(c\) with respect to \(\eta\) is given by
\[
 \omega_\eta(c)
 \coloneq\bigl\langle[\eta],[\widetilde c]\bigr\rangle\in\Z.
\]
The homotopy classes of line fields form an \(H^1(S;\Z)\)-torsor
\cite[Lemma~1.2]{LekiliPolishchuk}. If
\(\theta-\eta=\alpha\in H^1(S;\Z)\), then
\begin{equation}\label{eq:winding-difference}
 \omega_\theta(c)-\omega_\eta(c)=\langle\alpha,[c]\rangle,
 \qquad
 \omega_\eta(-c)=-\omega_\eta(c).
\end{equation}
Let \(c^\vee\in H^1(S;\Z)\) be the class defined by algebraic
intersection with \(c\). Then we have
\begin{equation}\label{eq:twist-line-field}
 T_c^*\eta-\eta=\omega_\eta(c)c^\vee.
\end{equation}
Consequently, for a nonseparating curve \(c\), the twist \(T_c\)
preserves \(\eta\) if and only if \(\omega_\eta(c)=0\). If \(c\) is
separating, then \(c^\vee=0\), so \(T_c\) preserves every line field.

Let \(a,b\) be disjoint, non-isotopic, nonseparating curves, oriented so
that \([a]=[b]\). Suppose that they bound a connected subsurface
\(R\subset S\) with
\(\partial R=a\sqcup(-b)\), containing no component of \(\partial S\).
Such a pair is called an \emph{oriented bounding pair}. For
\(u,v\in\Z\), formula \eqref{eq:twist-line-field} gives
\begin{equation}\label{eq:weighted-bounding-pair}
 (T_a^uT_b^{-v})^*\eta\simeq\eta
 \quad\Longleftrightarrow\quad
 u\omega_\eta(a)-v\omega_\eta(b)=0.
\end{equation}

A \emph{vector framing} is a nowhere-vanishing vector field \(\xi\).
Its projectivization \(\eta=\mathbb R\xi\) is a line field, and
\begin{equation}\label{eq:winding-doubling}
 \omega_\eta(c)=2\omega_\xi(c).
\end{equation}
We orient the boundary with \(S\) on its left. If
\(\partial S=B_1\sqcup\cdots\sqcup B_b\), the boundary winding numbers
of a framing satisfy
\begin{equation}\label{eq:framing-coherence}
 \sum_{i=1}^b\omega_\xi(B_i)=\chi(S).
\end{equation}
Conversely, every integral tuple satisfying
\eqref{eq:framing-coherence} is realized by a framing. Indeed, begin
with any framing of \(S\). The homotopy classes of framings form an
\(H^1(S;\Z)\)-torsor, and the image of the restriction map
\[
 H^1(S;\Z)\longrightarrow H^1(\partial S;\Z)\simeq\Z^b
\]
is the kernel of the summation map. This follows from the cohomology
sequence of the pair and shows that the initial boundary tuple can be
changed by any tuple of sum zero.

The relative displacement of a line field under a boundary-pointwise
mapping class will be used in Sections~\ref{sec:fibres} and
\ref{sec:grothendieck-actions}. Its construction, its values on
boundary twists, and the resulting change of normalization for graded
lifts are recorded in
Section~\ref{appsubsec:relative-gradings-charge}.

The Poincar\'e--Hopf formula for line fields
\cite[Equation~(1.3)]{LekiliPolishchuk} gives
\begin{equation}\label{eq:line-field-boundary-coherence}
 4-4g(S)=\sum_{i=1}^b\bigl(\omega_\eta(B_i)+2\bigr).
\end{equation}
More generally, let \(R\subset S\) be a connected oriented subsurface
and give each component of \(\partial R\) its boundary orientation.
Then
\[
 \sum_{C\subset\partial R}\bigl(\omega_\eta(C)+2\bigr)
 =4-4g(R).
\]
In particular, if an oriented bounding pair \(a,b\) bounds a
genus-\(h\) subsurface containing no component of \(\partial S\), then
\begin{equation}\label{eq:bounding-pair-winding-difference}
 \omega_\eta(a)-\omega_\eta(b)=-4h.
\end{equation}

Let \(\widehat S\) denote the closed surface obtained by capping every
boundary component of \(S\) with a disc. Suppose first that \(g(S)=1\).
If \(s,t\subset S\) are nonseparating curves whose images form a
symplectic basis of \(H_1(\widehat S;\Z)\), set
\begin{equation}\label{eq:genus-one-line-field-invariant}
 \widetilde{\mathcal A}(\eta)
 =\gcd\bigl(
 \omega_\eta(s),\omega_\eta(t),
 \omega_\eta(B_1)+2,\ldots,\omega_\eta(B_b)+2
 \bigr),
\end{equation}
where the greatest common divisor is taken to be non-negative.
The discussion preceding Theorem~1.8 in \cite{LekiliPolishchuk},
following the argument of \cite[Lemma~2.6]{Kawazumi}, gives
\[
 \widetilde{\mathcal A}(\eta)
 =
 \gcd\bigl\{
   \omega_\eta(\gamma)
   \,\big|\,
   \gamma\subset S
   \text{ is a nonseparating simple closed curve}
 \bigr\}.
\]
Consequently, \(\widetilde{\mathcal A}(\eta)\) is independent of the
chosen symplectic basis.

We next recall the invariants used in higher genus. Reduction of the
winding function modulo two defines a cohomology class
\([\omega_\eta]^{(2)}\in H^1(S;\Z/2)
\). We put
\[
 \sigma(\eta)=
 \begin{cases}
  0,&[\omega_\eta]^{(2)}=0,\\
  1,&[\omega_\eta]^{(2)}\neq0.
 \end{cases}
\]
Equivalently, \(\sigma(\eta)=0\) precisely when the winding number of
every closed curve is even.

By \cite[Proposition~1.6]{LekiliPolishchuk}, the winding function
determines a well-defined map
\begin{align*}
 q_\eta:H_1(S;\Z/4)&\longrightarrow\Z/4\\
 [c]&\mapsto\omega_\eta(c)+2\pmod4
\end{align*}
for every oriented simple closed curve \(c\), which satisfies
\[
 q_\eta(x+y)
 =q_\eta(x)+q_\eta(y)+2(x\mathbin{\cdot}y).
\]
Thus, if \(x=\sum_{j=1}^m[\alpha_j]\), then
\[
 q_\eta(x)=
 \sum_{j=1}^m\bigl(\omega_\eta(\alpha_j)+2\bigr)
 +2\sum_{i<j}[\alpha_i]\mathbin{\cdot}[\alpha_j]
 \pmod4.
\]
The last sum vanishes when the curves \(\alpha_j\) are pairwise
disjoint.

Suppose now that \(\sigma(\eta)=0\) and \(\omega_\eta(B_i)\in2+4\Z\) for every $i$. The function
\(q_\eta/2\) then vanishes on the boundary classes and descends to a
quadratic form
\[
 \overline q_\eta:
 H_1(\widehat S;\Z/2)\longrightarrow\Z/2
\]
satisfying
\[
 \overline q_\eta(x+y)
 =
 \overline q_\eta(x)+\overline q_\eta(y)
 +x\mathbin{\cdot}y.
\]
For any symplectic basis
\[
 e_1,f_1,\ldots,e_g,f_g
 \quad\text{of}\quad H_1(\widehat S;\Z/2),
\] following \cite{LekiliPolishchuk}, we define the Arf invariant of the
line field by
\[
 \operatorname{Arf}(\eta)
 \coloneq
 \sum_{j=1}^g
 \overline q_\eta(e_j)\overline q_\eta(f_j)
 \in\Z/2,
\]
which is independent of the chosen symplectic basis.

\begin{theorem}[{\cite[Lemma~1.4, Theorem~1.8, and
Corollary~1.10]{LekiliPolishchuk}}]
\label{thm:line-field-classification}
Let \(S\) be a compact connected oriented surface with non-empty
boundary. A line field \(\eta\) is induced by a nowhere-vanishing vector
field if and only if all its winding numbers are even. The orbits of
line fields under orientation-preserving mapping classes fixing the
boundary pointwise are classified as follows.
\begin{enumerate}[label=\textup{(\roman*)}]
\item In genus zero, two line fields are homotopic if and only if their
boundary winding numbers agree componentwise.
\item In genus one, two line fields lie in the same orbit if and only if
their boundary winding numbers and their invariants
\(\widetilde{\mathcal A}\) agree.
\item In genus at least two, two line fields lie in the same orbit if
and only if their boundary winding numbers and their
\(\sigma\)-invariants agree, and their Arf invariants agree whenever
these are defined.
\end{enumerate}
The corresponding statement for graded surfaces with stops also
requires the numbers of stops on matched boundary components to agree.
\end{theorem}

\subsection{Graded mapping class groups}

Let \(\Mod(S,M)\) be the group of isotopy classes of orientation-preserving
diffeomorphisms which preserve \(\partial S\) and \(M\) setwise.  The same
condition is imposed throughout an isotopy. The group \(\Mod(S)\) is defined similarly
after forgetting the stop set $M$. \(\PMod(S)\) is the subgroup of $\Mod(S)$ consisting of all elements which preserve every boundary component individually.
We write \(\Mod(S,\partial S)\) when diffeomorphisms and isotopies fix the
boundary pointwise. Thus boundary twists are non-trivial in \(\Mod(S,\partial S)\), while they
become trivial in \(\Mod(S)\). Finally, \(\Mod^*(S)\) is the corresponding extended
group of $\Mod(S)$ in which orientation-reversing diffeomorphisms are also allowed. 

If \(X\) is an oriented surface and \(Q\) is a finite set of interior
punctures, we write \(\Mod(X,Q)\) for the mapping class group preserving
\(Q\) and \(\partial X\) setwise.  Its pure subgroup \(\PMod(X,Q)\) fixes
every point of \(Q\) and preserves every boundary component individually.
Equivalently,
\[
 \PMod(X,Q)\coloneq
 \ker\!\left(
   \Mod(X,Q)\longrightarrow
   \operatorname{Sym}(Q)\times
   \operatorname{Sym}\bigl(\pi_0(\partial X)\bigr)
 \right).
\]
When \(Q=\varnothing\), this agrees with the notation \(\PMod(X)\) above. The action of \(\Mod(S,M)\) on line fields is given by pullback.  Put
\begin{equation*}
 H_\eta=\Stab_{\Mod(S,M)}(\eta)
 =\{[f]\in\Mod(S,M):f^*\eta\simeq\eta\}.
\end{equation*}
%In semidirect products, the corresponding left action on local systems is \(f\cdot L=(f^{-1})^*L\).

\begin{definition}[{\cite[Definition~6.6]{Opper2025}}]
A graded mapping class of \(\mathbb S=(S,M,\eta)\) is an isotopy class
of pairs \((f,[h])\), where \(f\) represents an element of
\(\Mod(S,M)\) and \([h]\) is the homotopy class of a path from
\(f^*\eta\) to \(\eta\) in the space of line fields.  The resulting group
is denoted by \(\MCGgr(\mathbb S)\).
\end{definition}

\begin{lemma}
The forgetful homomorphism from the graded mapping class group has image
\(H_\eta\) and fits into a central extension
\begin{equation}\label{eq:grading-extension}
 1\longrightarrow\Z\langle[1]\rangle
 \longrightarrow\MCGgr(\mathbb S)
 \longrightarrow H_\eta\longrightarrow1.
\end{equation}
Here \([1]\) is the positive deck translation of the grading cover and
corresponds to the categorical shift.
\end{lemma}

\begin{proof}
A path from \(f^*\eta\) to \(\eta\) determines a \(\Z\)-equivariant lift
of the map on \(\mathbb P(TS)\) induced by \(Df\).  Conversely, such a
lift determines the homotopy class of this path.  Hence a graded lift
exists precisely when \(f\in H_\eta\).  Any two lifts of the same \(f\)
differ by a unique deck translation. Therefore, the kernel is 
\(\Z\langle[1]\rangle\). Since deck translations commute with every
\(\Z\)-equivariant lift, the extension is central.
\end{proof}

Throughout, the positive boundary direction is the one with the
surface on the left. We write \(\rho_i\) for the collar-supported
one-step rotation of the stops on \(B_i\) in this direction and
\(T_i=T_{B_i}\) for the positive boundary twist. Thus
\(\rho_i^{m_i}=T_i\), where \(m_i=|M\cap B_i|\).

\begin{lemma}\label{lem:stopped-stabilizer-root-presentation}
Let \(H_\eta^0\) be the subgroup of \(H_\eta\) preserving every
boundary component, and put
\(F=\Stab_{\Mod(S,\partial S)}(\eta)\). If
\(m_i=|M\cap B_i|\), then
\[
 H_\eta^0\cong
 \left\langle F,\rho_1,\ldots,\rho_b\ \middle|\
 [\rho_i,F]=1,\ [\rho_i,\rho_j]=1,\
 \rho_i^{m_i}=T_i\right\rangle .
\]
Consequently,
\[
 H_\eta^0/F\cong\prod_{i=1}^b\Z/m_i\Z.
\]
\end{lemma}

\begin{proof}
Let \(h\in H_\eta^0\). Since \(h\) preserves each boundary component \(B_i\), it induces a cyclic permutation of the stops \(M\cap B_i\). Hence there exist integers \(a_i\) such that
\(
h'= h\,\rho_1^{a_1}\cdots \rho_b^{a_b}\)
fixes every stop.

On each \(B_i\), the boundary restriction of \(h'\) is an orientation-preserving diffeomorphism of the circle fixing all points of \(M\cap B_i\). It can therefore be isotoped to the identity relative to these stops. After an isotopy supported in a collar of the boundary, we may assume that \(h'\) is the identity on \(\partial S\). Collar isotopies do not change the absolute homotopy class of the line field, so the adjusted \(h'\) still lies in \(H_\eta\). Thus \(h'\in F\), which implies that \(H_\eta^0\) is generated by \(F\) together with \(\rho_1,\ldots,\rho_b\).

The rotations \(\rho_i\) are central in the boundary-preserving subgroup, and \(\rho_i^{m_i}=T_i\) is the full boundary twist on \(B_i\). Therefore
\[
[\rho_i,F]=1,\qquad [\rho_i,\rho_j]=1,\qquad \rho_i^{m_i}=T_i.
\]
Conversely, every element of the group on the right can be written as
\[
f\,\rho_1^{a_1}\cdots \rho_b^{a_b},\qquad f\in F,
\]
and the exponents \(a_i\) are well defined only modulo \(m_i\), because \(\rho_i^{m_i}=T_i\in F\). Hence these relations give the required presentation.

Finally, the standard inclusion of the boundary-pointwise mapping class group into the stopped mapping class group is injective. Therefore \(F\) is exactly the subgroup of \(H_\eta^0\) consisting of elements that act trivially on the stops on every boundary component. The quotient \(H_\eta^0/F\) therefore records only the cyclic permutations of the \(m_i\) stops on each \(B_i\), so
\[
H_\eta^0/F\cong \prod_{i=1}^b \Z/m_i\Z.
\]
This proves the lemma.
\end{proof}

\subsection{Perfect derived categories and derived Picard groups}

A graded gentle algebra is a gentle algebra \(\kk Q/I\) with
integer degrees assigned to the arrows. Here \(Q\) is finite,
\(I\) is generated by paths of length two, and the gentle conditions
require at most two arrows entering or leaving each vertex and at most
one permitted and one forbidden continuation on either side of an
arrow. We regard a graded algebra as a dg algebra with zero
differential, and write \(\Perf(A)\) for its perfect derived category
of right dg modules. Let \(\DPic(A)\) denote the derived Picard group,
i.e. the group of isomorphism classes of standard enhanced
autoequivalences, with multiplication \([F][G]=[F\circ G]\). An invertible dg \(A\)-bimodule \(X\)
represents the functor \(F_X=-\otimes_A^{\mathbf L}X\). Thus the multiplication is given by
\[
 [X]\star[Y]=[Y\otimes_A^{\mathbf L}X],
 \qquad F_{X\star Y}\simeq F_X\circ F_Y.
\]
If \(A\) is smooth and proper, its dual bimodule
\(\operatorname{RHom}_{\kk}(A,\kk)\) represents the Serre functor.

We use the terminology of graded surfaces with stops as in
\cite{ChangElaginSchroll}. Choose an auxiliary set \(M^\circ\subset\partial S\setminus M\)
with one point in each component of \(\partial S\setminus M\).
Thus the points of \(M\) and \(M^\circ\) alternate on every boundary
component. In the two-colour convention, \(M\) is the set of
\(\bullet\)-points, while the arcs representing objects have endpoints
in \(M^\circ\). In particular,
\[
 |M^\circ\cap B_i|=|M\cap B_i|=m_i.
\]
The position of the auxiliary points is unique up to isotopy relative
to \(M\). A diffeomorphism preserving \(M\) may therefore be
isotoped to preserve \(M^\circ\) as well, uniquely at the level of
mapping classes. We use this identification when applying a mapping
class to arcs or to relative homology. Replacing the auxiliary points
by small distinguished boundary intervals gives the convention of
\cite{HKK,Opper2025}.

A collection of simple arcs with endpoints in \(M^\circ\) is an
\emph{admissible \(\circ\)-dissection} if its arcs are pairwise
disjoint in the interior and cut \(S\) into polygons, each containing
exactly one point of \(M\). Grading these arcs gives a full formal
system of graded arcs. If \(E\) is their direct sum, its total graded
endomorphism algebra
\(A=\bigoplus_{n\in\Z}\operatorname{Hom}(E,E[n])\), with its induced
multiplication, is finite-dimensional and graded gentle,
and there is a Morita equivalence
\(
 \Perf(A)\simeq\W(\mathbb S)\). Conversely, every homologically smooth and proper graded gentle
algebra admits such a surface model. Here, we follow the topological construction and its smoothness and
properness criteria from \cite[Section~3]{HKK},
\cite[Theorem~3.11]{LekiliPolishchuk}, and
\cite[Proposition~2.5]{Opper2025}.

\begin{theorem}[{\cite[Theorem~1.4]{JinSchrollWang}}]
\label{thm:complete-surface-invariant}
Let \(A\) and \(A'\) be homologically smooth and proper graded gentle
algebras, with surface models \((S,M,\eta)\) and
\((S',M',\eta')\).  Then \(\Perf(A)\) and \(\Perf(A')\) are triangle
equivalent if and only if there is an orientation-preserving homeomorphism
\[
 f\colon(S,M)\longrightarrow(S',M')
\]
such that \(f^*\eta'\simeq\eta\).
\end{theorem}
 
According to  \cite{AvellaAlaminosGeiss2008,LekiliPolishchuk} and
\cite[Section~9]{Opper2025}, the \emph{graded AG-invariant} associated with $(S,M,\eta)$ is defined as 
\[
 \phi_A(n,d)=
 \#\{B_i\subset\partial S:
       |M\cap B_i|=n,\ n-\omega_\eta(B_i)=d\}.
\]
 In particular,
\(\phi_A(1,1)\) counts the boundary components with one
stop and winding number zero. Set \(D(\mathbb S)\coloneqq\DPic(\W(\mathbb S))\) and recall the following description of this group.

\begin{theorem}[{\cite[Corollary~10.6]{Opper2025}}]
\label{thm:opper-geometrization}
Let \(A\) be a connected homologically smooth and proper graded
gentle algebra over \(\kk\), with surface model
\(\mathbb S=(S,M,\eta)\). Assume that every boundary component
contains a stop and that \(\chi(S)<0\).
Under the Morita identification
\(D(\mathbb S)\simeq\DPic(A)\), there is a split exact sequence
\[
 1\longrightarrow N(\mathbb S)
 \longrightarrow D(\mathbb S)
 \xrightarrow{\,\Psi_{\mathbb S}\,}\MCGgr(\mathbb S)
 \longrightarrow1,
\]
where
\[
 N(\mathbb S)\simeq
 \mathbb G_a^{\phi_A(1,1)}\rtimes H^1(S;\kk^\times).
\]
\end{theorem}

The restrictions in this statement explain the specialization of
\cite[Corollary~10.6]{Opper2025}. The exceptional Kronecker surface is an
annulus and is excluded by \(\chi(S)<0\). Moreover, there is no boundary
component without stops, so the factors indexed by \(\phi_A(0,0)\) in
the general formula do not occur. The additive factors that remain are
exactly those indexed by \(\phi_A(1,1)\).

Under the hypotheses of Theorem~\ref{thm:opper-geometrization}, the kernel is the group of
\(\kk\)-points of a connected solvable algebraic group.  Its maximal torus
has rank
\begin{equation*}
 \rk H^1(S;\kk^\times)=\rk H_1(S;\Z)=2g+b-1.
\end{equation*}

For surfaces occurring in Section~\ref{sec:fibres}, every boundary
winding is non-zero, so \(\phi_A(1,1)=0\).  Theorem~\ref{thm:opper-geometrization} then gives
\begin{equation}\label{eq:derived-picard-formula}
 D(\mathbb S)\simeq
 H^1(S;\kk^\times)\rtimes\MCGgr(\mathbb S).
\end{equation}

\subsection{The Serre functor and the Grothendieck group}

Having described the derived Picard group, we now turn to its
distinguished Serre element and the induced action on the
Grothendieck group. We first recall the geometric description of
Auslander--Reiten translation by boundary rotation, and then
identify the Grothendieck group with the charge lattice of the
graded surface. Together, these descriptions allow us to express
the Serre action in terms of boundary rotation on relative homology.

\begin{proposition}[{\cite[Corollaries~5.2 and~5.4]
{OpperPlamondonSchroll}; \cite[Proposition~1.18]{ChangElaginSchroll}}]
\label{prop:geometric-serre-functor}
Suppose that the surface category is smooth and proper. Under the
surface model, the Auslander--Reiten translation \(\mathsf S[-1]\)
of a perfect arc object is obtained by moving both endpoints to the
adjacent points of \(M^\circ\) in the negative boundary direction
(with the surface on the right) and using the induced grading.
The Serre functor is obtained by following this graded endpoint
move with the shift \([1]\).
\end{proposition}

\begin{theorem}[{\cite[Theorem~5.1]{HKK}}]
\label{thm:charge-lattice}
Let \(\mathbb S=(S,M,\eta)\) be a compact graded surface with stops, and let
\(\Z_\eta\) be the orientation local system of the line field.
There is a natural isomorphism
\[
 K(\mathbb S)
 \simeq H_1(S,M^\circ;\Z_\eta).
\]
If \(\eta=\mathbb R\xi\) is induced by a vector framing, then
\(\Z_\eta\) is trivial and the right-hand side is
\(H_1(S,M^\circ;\Z)\).
\end{theorem}

The charge lattice is free abelian of finite rank. Indeed, the pair
\((S,M^\circ)\) admits a finite relative graph model. Its first homology with
coefficients in the rank-one local system \(\Z_\eta\) is therefore a
subgroup of a free chain group, while its relative homology vanishes in
degrees zero and at least two. The Euler characteristic gives
\begin{equation}\label{eq:charge-rank}
 \rk K(\mathbb S)=|M|-\chi(S)=2g+b+|M|-2.
\end{equation}
This calculation does not require the line field to be orientable.

Under the hypotheses of Theorem~\ref{thm:opper-geometrization}, the
action of \(D(\mathbb S)\) on the charge lattice factors through the
graded mapping class group. Under Theorem~\ref{thm:charge-lattice},
it is the natural action on twisted relative homology, and the shift
acts by \(-I\). The homology sequence of the pair gives
\[
 0\longrightarrow H_1(S;\Z_\eta)
 \longrightarrow K(\mathbb S)
 \longrightarrow
 \ker\!\left(H_0(M^\circ;\Z_\eta)\longrightarrow
                       H_0(S;\Z_\eta)\right)
 \longrightarrow0.
\]
The quotient is a subgroup of the free endpoint group
\(H_0(M^\circ;\Z_\eta)\). Thus the absolute homology is a saturated
integral sublattice of the charge lattice. This exact sequence is
valid without the Euler-characteristic restriction imposed in
geometrisation.
The equivariance assertions and the triviality of the geometrisation
kernel action are proved in Lemma~\ref{lem:geometrization-charge-action}
in Section~\ref{appsubsec:relative-gradings-charge}.

We finish the preliminaries by relating these homological actions to
the Serre functor.
\begin{lemma}
Assume that \(\eta=\mathbb R\xi\) is induced by a vector framing.
Let \(R\) be simultaneous positive one-step boundary rotation,
represented in boundary collars so that it moves both \(M\) and
\(M^\circ\) to the next points of the same type in the positive
boundary direction. Under the identification
\(K(\mathbb S)\simeq H_1(S,M^\circ;\Z)\),
\begin{equation}\label{eq:serre-rotation-action}
 (\mathsf S[-1])_*=\varepsilon R_*^{-1},
 \qquad \varepsilon\in\{1,-1\}.
\end{equation}
Here \(\varepsilon\) is the sign induced by the graded lift on the
orientation local system after trivializing it using \(\xi\).
\end{lemma}

\begin{proof}
By Proposition~\ref{prop:geometric-serre-functor}, the functor
\(\mathsf S[-1]\) acts on the generating graded arcs by a single
graded lift of \(R^{-1}\). The charge-lattice identification is
natural for graded diffeomorphisms, including the induced
isomorphism of orientation local systems. The vector field \(\xi\)
trivializes \(\Z_\eta\), and in this trivialization the induced
isomorphism of rank-one integral local systems is multiplication
by a single sign \(\varepsilon\), since \(S\) is connected. Thus its
action on relative homology is \(\varepsilon R_*^{-1}\). The classes
of the graded arcs generate \(K(\mathbb S)\), which proves
\eqref{eq:serre-rotation-action}.
\end{proof}

% --------------------------------------------------------------------------
\section{Ungraded reconstruction}\label{sec:ungraded}
In this section, we identify the boundary-rotation subgroup and the capped
mapping class group from the ungraded mapping class group.
The resulting central extension recovers the number of stops on each
boundary component.
% --------------------------------------------------------------------------

\begin{theorem}\label{thm:ungraded-reconstruction}
Let \((S,M)\) and \((S',M')\) be compact connected oriented
surfaces with stops and non-empty boundary, at least one stop on every boundary
component, and genera at least two.  If
\[
 \Mod(S,M)\simeq\Mod(S',M')
\]
as abstract groups, then there is an orientation-preserving diffeomorphism
\((S,M)\simeq(S',M')\).  In particular, the group determines the genus,
the number of boundary components, and the distribution of stops among
them.
\end{theorem}

For a punctured surface
\((X,Q)\), write \(\Mod^*(X,Q)\) for the extended mapping class group,
which allows orientation-reversing mapping classes and permutations of
the points of \(Q\).

\begin{theorem}[{\cite[Theorem~3]{BehrstockMargalit}}]
\label{thm:mapping-class-rigidity}
Let \(X\) be a closed oriented surface of genus at least two, and let
\(Q\neq\varnothing\) be finite.  If \(\Gamma\) is a finite-index subgroup
of \(\Mod^*(X,Q)\), then every injective homomorphism
\[
 \Gamma\longrightarrow\Mod^*(X,Q)
\]
is the restriction of conjugation by an element of
\(\Mod^*(X,Q)\).
\end{theorem}

Let \(S=\Sigma_{g,b}\) with \(g\geq2\), let
\[
 \partial S=B_1\sqcup\cdots\sqcup B_b,
 \qquad
 M_i=M\cap B_i,
 \qquad
 m_i=|M_i|,
\]
and write \(G=\Mod(S,M)\). For each \(i\), glue a disc \(D_i\) to
\(B_i\), choose a point \(q_i\in\operatorname{int}(D_i)\), and denote the
resulting closed surface by \(\widehat S\). Set
\(
 Q=\{q_1,\ldots,q_b\}\). The points of \(Q\) record the capped boundary components. For each
boundary component \(B_i\), let \(\rho_i\in G\) be the positive one-step
rotation of the stops on \(B_i\), supported in a boundary collar.
Set
\[
 A_M=\langle\rho_1,\ldots,\rho_b\rangle
 \qquad\text{and}\qquad
 G^0=\{f\in G\mid f(B_i)=B_i\text{ for every }i\}.
\]

\begin{lemma}\label{lem:intrinsic-capping-extension}
The subgroup \(A_M\) is isomorphic to \(\Z^b\). It is the unique maximal
normal abelian subgroup of \(G\), and hence is characteristic. Moreover,
\[
 G^0=C_G(A_M).
\]
Capping the boundary components induces a central extension
\[
 1\longrightarrow A_M\longrightarrow G^0
 \longrightarrow\PMod(\widehat S,Q)\longrightarrow1.
\]
Its kernel is precisely the subgroup generated by the one-step boundary
rotations. In particular,
\(
 A_M=Z(G^0)\).
\end{lemma}
\begin{proof}

The rotations commute and satisfy
\begin{equation}\label{eq:boundary-root-relation}
 \rho_i^{m_i}=T_{B_i},
\end{equation}
where \(T_{B_i}\) is the positive boundary twist. For
\(g\geq2\), the boundary twists \(T_{B_i}\) are independent in
\(\Mod(S,\partial S)\), and \(\Mod(S,\partial S)\) embeds in
\(\Mod(S,M)\). Suppose that \(\prod_i\rho_i^{a_i}=1\), and choose a
positive common multiple \(m\) of \(m_1,\ldots,m_b\). Then
\[
  1=\left(\prod_i\rho_i^{a_i}\right)^m
   =\prod_i(\rho_i^{m_i})^{a_i m/m_i}
   =\prod_i T_{B_i}^{a_i m/m_i}.
\]
Independence gives \(a_i m/m_i=0\) for every \(i\), hence \(a_i=0\).
Thus \(A_M\cong\Z^b\).

We next show that \(A_M\) can be recognized from the abstract group
\(G\). The capping construction defines a homomorphism
\begin{equation}\label{eq:coloured-capping-map}
 \kappa:G\longrightarrow \Mod(\widehat S,Q).
\end{equation}
Its image consists of those mapping classes whose induced permutation of
\(Q\) preserves the function \(q_i\mapsto m_i\), and its kernel is
\(A_M\).  Indeed, when all boundary points are required to be fixed, this is the usual capping exact sequence, whose kernel is generated by the
boundary twists. Allowing cyclic permutations of the \(m_i\) stops on
\(B_i\) replaces the generator \(T_{B_i}\) by its distinguished
\(m_i\)-th root \(\rho_i\), and introduces no further element in the
kernel.

We claim that the image of \(\kappa\) has no nontrivial normal abelian
subgroup. Let \(L\) be such a subgroup and put
\(P=\PMod(\widehat S,Q)\). The group \(P\) is normal in
\(\operatorname{im}\kappa\) and has finite index there.
Lemma~\ref{lem:no-normal-abelian-subgroup} gives \(L\cap P=1\), so
\(L\) is finite. The kernel of the conjugation action of \(P\) on \(L\)
has finite index in \(P\). Since this kernel is centralized by \(L\),
the second assertion of
Lemma~\ref{lem:no-normal-abelian-subgroup} implies \(L=1\).

It follows that every normal abelian subgroup of \(G\) is contained in
\(\ker\kappa=A_M\). Hence \(A_M\) is the unique maximal normal abelian
subgroup of \(G\), and therefore is characteristic. Note that this
argument also covers the case \((g,b)=(2,1)\). In that case the puncture
removes the hyperelliptic centre that occurs for the closed genus-two
mapping class group.

The conjugation action of \(G\) on \(A_M\) agrees with its permutation
action on the boundary components. Indeed, an orientation-preserving mapping
class sends \(\rho_i\) to \(\rho_j\) precisely when it sends \(B_i\) to
\(B_j\). The boundary orientation rules out \(\rho_j^{-1}\). It follows that
\(
 G^0=C_G(A_M)
\)
is characteristic. Its elements are precisely the mapping classes that
preserve every boundary component. Restricting
\eqref{eq:coloured-capping-map} therefore gives the central extension
\begin{equation}\label{eq:marked-capping-extension}
 1\longrightarrow A_M
 \longrightarrow G^0
 \longrightarrow\PMod(\widehat S,Q)
 \longrightarrow1.
\end{equation}
The last arrow is onto.  Indeed, a pure mapping class of \((\widehat S,Q)\) may be
represented by a diffeomorphism which is the identity on small discs about
the points of \(Q\), and whose restriction to the complement fixes all
stops.  The quotient
\(\PMod(\widehat S,Q)\) is centre-free for \(g\geq2\) and \(b\geq1\)
\cite{ParisRolfsen}. Thus \(A_M=Z(G^0)\), which gives a second
description of both the kernel and the quotient in
\eqref{eq:marked-capping-extension}.
\end{proof}

\begin{proposition}[{\cite{Harer,BirmanLubotzkyMcCarthy}}]\label{prop:pure-mapping-class-ranks}
Let \(X\) be a closed surface of genus \(g\geq2\), and let \(Q\) be a
non-empty set of \(b\) punctures.  Then the virtual cohomological dimension of $\operatorname{PMod}(X,Q)$ is
\[
\operatorname{vcd}\PMod(X,Q)=4g-4+b,
\]
and the maximal rank of a free abelian
subgroup of $\operatorname{PMod}(X,Q)$ is
\[\max\{n:\Z^n\hookrightarrow\PMod(X,Q)\}=3g-3+b.
\]
\end{proposition}

\begin{lemma}\label{lem:pure-mapping-class-rigidity}
Let \(\widehat S\) and \(\widehat S'\) be closed oriented surfaces of
genera \(g,g'\geq2\), and let \(Q,Q'\) be non-empty finite sets of
cardinalities \(b,b'\).  Every abstract isomorphism
\[
 \PMod(\widehat S,Q)\xrightarrow{\sim}
 \PMod(\widehat S',Q')
\]
is induced by an extended diffeomorphism
\((\widehat S,Q)\to(\widehat S',Q')\).  In particular,
\(g=g'\) and \(b=b'\).
\end{lemma}

\begin{proof}
The two numbers in Proposition~\ref{prop:pure-mapping-class-ranks}
are invariants of the abstract group. Their difference is \(g-1\),
and either formula then recovers \(b\). Hence an isomorphism forces
\(g=g'\) and \(b=b'\).

Choose an extended diffeomorphism between the two punctured surfaces
and use it to identify them. The given isomorphism becomes an injective
homomorphism from the finite-index subgroup \(\PMod(\widehat S,Q)\)
of \(\Mod^*(\widehat S,Q)\) into \(\Mod^*(\widehat S,Q)\).
Since \(b>0\), the exceptional closed surface of genus two does not
occur. Theorem~\ref{thm:mapping-class-rigidity} shows that this
homomorphism is induced by conjugation by an extended mapping class.
Composing with the preliminary identification proves the assertion.

\end{proof}

\begin{proposition}\label{prop:extension-recovers-stops}
Let \(G=\Mod(S,M)\) and \(G'=\Mod(S',M')\), where both surfaces have genus
at least two.  Every abstract isomorphism \(G\simeq G'\) identifies the
genera and the numbers of boundary components, and carries the stop
distribution of \((S,M)\) to that of \((S',M')\).
\end{proposition}

\begin{proof}
By Lemma~\ref{lem:intrinsic-capping-extension}, an isomorphism
\(\Phi:G\to G'\) maps \(A_M\) onto \(A_{M'}\), maps \(G^0\) onto
\((G')^0\), and induces
\begin{equation}\label{eq:pure-capped-isomorphism}
 \overline\Phi:
 \PMod(\widehat S,Q)\xrightarrow{\sim}
 \PMod(\widehat S',Q').
\end{equation}
Lemma~\ref{lem:pure-mapping-class-rigidity} shows directly that
\eqref{eq:pure-capped-isomorphism} is induced by an extended
diffeomorphism \((\widehat S,Q)\to(\widehat S',Q')\).  In particular, \(g(S)=g(S'), b=b'\), and, after reindexing, the induced map sends \(q_i\) to \(q_i'\).

It remains to recover the integers \(m_i\). Let
\(\psi_i\in H^2(\PMod(\widehat S,Q);\Z)\) be the cotangent-line
class at \(q_i\). We use the convention in which the central
extension obtained by replacing the punctures with parametrized
boundary circles has \(\psi_i\) as its \(i\)-th component. Put
\(E_\partial=\Mod(S,\partial S)\).  The pointwise-boundary capping sequence
has class
\(
 e_\partial=\sum_{i=1}^b\psi_i\otimes T_{B_i}
\)
and fits into the upper row of the pushout diagram
\[
\begin{tikzcd}
1 \ar[r] &
\displaystyle \bigoplus_{i=1}^{b} \mathbb Z\langle T_{B_i}\rangle
  \ar[r] \ar[d, hook] &
E_{\partial} \ar[r] \ar[d, hook] &
\PMod(\widehat S,Q) \ar[r] \ar[d, equal] &
1 \\[2pt]
1 \ar[r] &
A_M \ar[r] &
G^0 \ar[r] &
\PMod(\widehat S,Q) \ar[r] &
1
\end{tikzcd}
\]
The left vertical map is
\(
 T_{B_i}\longmapsto m_i\rho_i\)
by \eqref{eq:boundary-root-relation}.  The lower row is the pushout of the
upper row along this map.  Naturality of extension classes therefore gives
\begin{equation*}
 e_M=\sum_{i=1}^b m_i\psi_i\otimes\rho_i
 \in H^2\bigl(\PMod(\widehat S,Q);A_M\bigr).
\end{equation*}
Changing the convention for positive twists changes every sign
simultaneously and does not affect the positive integers recovered below.

By Lemma~\ref{lem:psi-classes} in
Section~\ref{appsubsec:capping-rigidity}, the classes
\(\psi_1,\ldots,\psi_b\) are linearly independent over \(\mathbb Q\).
We may therefore compare extension classes after tensoring the
coefficient lattice with \(\mathbb Q\). Any resulting equality of
lattice vectors holds integrally because \(A_M\) is torsion free.
Let \(\alpha:A_M\to A_{M'}\) be the restriction of \(\Phi\).  Naturality
of extension classes gives
\begin{equation}\label{eq:extension-naturality}
 (\overline\Phi^*\otimes\alpha^{-1})(e_{M'})=e_M.
\end{equation}
The geometric description of \(\overline\Phi\) gives a permutation
\(\pi\in\mathfrak S_b\) and one sign
\(\varepsilon\in\{1,-1\}\) such that
\(
 \overline\Phi^*\psi_{\pi(i)}'=\varepsilon\psi_i\). The sign is common to all \(i\), since it records whether the inducing
diffeomorphism preserves or reverses the orientation of the surface.
Comparing the independent \(\psi_i\)-coefficients in
\eqref{eq:extension-naturality} yields
\begin{equation}\label{eq:stop-number-coefficients}
 m_i\rho_i
 =\varepsilon m_{\pi(i)}'\alpha^{-1}(\rho_{\pi(i)}')
 \qquad\text{ for }1\leq i\leq b.
\end{equation}
Each \(\rho_i\) and each \(\alpha^{-1}(\rho_i')\) is primitive in the
free abelian group \(A_M\).  Equation~\eqref{eq:stop-number-coefficients} therefore implies
\(m_i=m_{\pi(i)}'\) for every \(i\).  This recovers the number of stops on each boundary component.
\end{proof}

\begin{proof}[Proof of Theorem~\ref{thm:ungraded-reconstruction}]
By Proposition~\ref{prop:extension-recovers-stops}, the surfaces have the same
genus, the same number of boundary components, and the same stop
distribution.  The classification of compact connected oriented surfaces
with stops therefore gives an orientation-preserving
diffeomorphism \((S,M)\simeq(S',M')\).
\end{proof}

\begin{example}\label{ex:low-genus-ungraded-failure}
The genus assumption in Theorem~\ref{thm:ungraded-reconstruction} cannot be
removed.  For genus zero, let \(P=\Sigma_{0,3}\) be a pair of pants and
choose stop distributions \((1,2,3)\) and \((1,2,4)\).  In either case the
three boundary components are distinguished by their stop numbers.  The
mapping class group is therefore freely generated by the three one-step
boundary rotations and is isomorphic to \(\Z^3\).  The two surfaces with stops
are not diffeomorphic because their stop distributions are different.

For genus one, let \(M_m\) be a set of \(m\) stops on the boundary of
\(\Sigma_{1,1}\), and write \(G_m=\Mod(\Sigma_{1,1},M_m)\).  Let \(X_m\) and \(Y_m\)
be the twists about two nonseparating curves meeting once, and let
\(\rho_m\) be the positive one-step boundary rotation.  The braid relation and the
two-chain relation give the presentation
\[
 G_m=
 \left\langle X_m,Y_m,\rho_m\ \middle|\
 \begin{aligned}
 X_mY_mX_m&=Y_mX_mY_m,\\
 [X_m,\rho_m]&=[Y_m,\rho_m]=1,\\
 (X_mY_m)^6&=\rho_m^m
 \end{aligned}
 \right\rangle
\]
(see \cite{FarbMargalit}).  There is an isomorphism 
\(G_m\longrightarrow G_{m+12}\) defined by
\[
 \rho_m\longmapsto\rho_{m+12},
 \qquad
 X_m\longmapsto X_{m+12}\rho_{m+12}^{-1},
 \qquad
 Y_m\longmapsto Y_{m+12}\rho_{m+12}^{-1}.
\]
Indeed, the central factors contribute \(\rho_{m+12}^{-12}\) to the
sixth power of \(X_mY_m\).  Taking \(m=1\) gives one-holed tori
with one and thirteen stops whose mapping class groups are isomorphic.
The surfaces with stops themselves are not diffeomorphic.
\end{example}

\begin{remark}
Theorem~\ref{thm:ungraded-reconstruction} concerns the full ungraded
mapping class group.  The quotient supplied by the geometrisation sequence
in Theorem~\ref{thm:opper-geometrization} is instead
\(\MCGgr(\mathbb S)\), and after
dividing by the shift it is the line-field stabilizer \(H_\eta\), not
\(\Mod(S,M)\).  Thus the theorem cannot be applied directly to an abstract
derived Picard group.
\end{remark}

% --------------------------------------------------------------------------
\section{Reconstruction from the derived Picard group}
\label{sec:derived-reconstruction}
In this section, we first identify the geometrisation kernel as a characteristic
subgroup of the derived Picard group. Two group-theoretic ranks
then recover the genus and the number of boundary components.
When the shift is distinguished, we use the action on the torsion
of the abelianized kernel to recover the line field up to an
orientation-preserving diffeomorphism.
% --------------------------------------------------------------------------

\begin{theorem}\label{thm:derived-reconstruction}
Let \(A\) and \(A'\) be connected homologically smooth and proper graded
gentle algebras over \(\kk\).  Write their surface models as
\[
 \mathbb S_A=(S_A,M_A,\eta_A),
 \qquad
 \mathbb S_{A'}=(S_{A'},M_{A'},\eta_{A'}).
\]
Assume that every boundary component contains a stop.  We further assume
that both \(S_A\) and \(S_{A'}\) have genus at least two.
\begin{enumerate}[label=\textup{(\roman*)}]
\item An abstract isomorphism \(\DPic(A)\simeq\DPic(A')\) implies
\[
 g(S_A)=g(S_{A'}),\qquad
 |\pi_0(\partial S_A)|=|\pi_0(\partial S_{A'})|.
\]
\item If the isomorphism \(\Phi\colon\DPic(A)\to\DPic(A')\) satisfies
\(\Phi([1])=[1]\), then there is an orientation-preserving diffeomorphism
\[
 f\colon S_A\longrightarrow S_{A'}
 \qquad\text{with}\qquad f^*\eta_{A'}\simeq\eta_A.
\]
\end{enumerate}
Thus the derived Picard group, together with the shift, determines
the underlying graded surface after the stops are forgotten.
\end{theorem}

\subsection{The kernel and intrinsic ranks}

Let \(C_{\ell^\infty}=\varinjlim_n\Z/\ell^n\Z\) be the Pr\"ufer
\(\ell\)-group.  For an abstract group \(G\), define
\begin{equation*}
 \prk_\ell(G)
 \coloneq\sup\{r\in\Z_{\geq0}:(C_{\ell^\infty})^r\hookrightarrow G\}.
\end{equation*}
For a group \(G\), let
\[
 \mathcal F(G)=
 \{\Gamma\leq G\mid [G:\Gamma]<\infty\}
\]
denote the family of all finite-index subgroups of \(G\).  For
\(\Gamma\in\mathcal F(G)\), write \(Z(\Gamma)\) for the centre of
\(\Gamma\), and let \(\Div(Z(\Gamma))\) be its maximal divisible subgroup.
The \emph{reduced virtual-centre rank} of \(G\) is given by
\begin{equation*}
 \rvz(G)\coloneq
 \sup_{\Gamma\in\mathcal F(G)}
 \dim_{\mathbb Q}\left(
   \bigl(Z(\Gamma)/\Div(Z(\Gamma))\bigr)
   \otimes_{\Z}\mathbb Q
 \right).
\end{equation*}

For an abstract group \(G\), let
\begin{equation*}
 \operatorname{DA}(G)
 \coloneq\left\langle A\leq G\,\middle|\,
 A\text{ is divisible and abelian}\right\rangle .
\end{equation*}
Automorphisms of \(G\) permute the subgroups occurring in this definition,
so \(\operatorname{DA}(G)\) is characteristic.  For a prime \(\ell\), write
\(\operatorname{Tor}_{\ell^\infty}(B)\) for the subgroup of an abelian
group \(B\) consisting of elements of \(\ell\)-power order.

\begin{lemma}\label{lem:characteristic-geometrization-kernel}
Assume that \(\mathbb S=(S,M,\eta)\) has genus at least two.  Then
\[
 \operatorname{DA}\bigl(D(\mathbb S)\bigr)=N(\mathbb S).
\]
Moreover, for every prime \(\ell\), there is a natural identification
\begin{equation}\label{eq:ell-local-systems}
 \operatorname{Tor}_{\ell^\infty}
 \bigl(N(\mathbb S)^{\mathrm{ab}}\bigr)
 \simeq H^1(S;\mu_{\ell^\infty}),
\end{equation}
and conjugation by \(D(\mathbb S)\) induces the action on rank-one local systems given by
\(f\cdot L=(f^{-1})^*L\).
\end{lemma}

\begin{proof}
Set \(N=N(\mathbb S)\). By Theorem~\ref{thm:opper-geometrization},
\[
 N=U\rtimes T,
 \qquad
 U\simeq\mathbb G_a^{\phi_A(1,1)},
 \qquad
 T=H^1(S;\kk^\times).
\]
Both \(U\) and \(T\) are divisible abelian groups.  For \(T\), this uses
the fact that \(\kk\) is algebraically closed.  Since they generate \(N\),
we have \(N\subseteq\operatorname{DA}(D(\mathbb S))\).

Conversely, let \(A\le D(\mathbb S)\) be a divisible abelian subgroup.
Its image in \(\MCGgr(\mathbb S)\) is divisible and abelian. It is known that abelian subgroups of a
mapping class group are finitely generated \cite{BirmanLubotzkyMcCarthy}.
The same holds in the graded central extension because its kernel is
cyclic.  Thus the image of \(A\) is a
finitely generated divisible abelian group, hence trivial.  Therefore
\(
 A\subseteq N(\mathbb S)\). Therefore, every divisible abelian subgroup of \(D(\mathbb S)\) lies in \(N\),
and hence
\(
 \operatorname{DA}\bigl(D(\mathbb S)\bigr)\subseteq N(\mathbb S)\), which proves the first assertion.

Since \(N/U\simeq T\) is abelian, we have \([N,N]\subseteq U\) and
\(T\cap[N,N]=1\).  Consequently,
\[
 N^{\mathrm{ab}}\simeq U/[T,U]\times T.
\]
The subspace \([T,U]\) is the sum of the non-trivial weight spaces
for the torus action on \(U\). Hence \(U/[T,U]\) is a
\(\kk\)-vector space with torsion-free additive group. Therefore
\[
 \operatorname{Tor}_{\ell^\infty}(N^{\mathrm{ab}})
 =T[\ell^\infty]
 =H^1(S;\mu_{\ell^\infty}).
\]
Inner automorphisms of \(N\) become trivial after abelianization.  The induced action
on the remaining torus is the local-system action described in \cite[Sections~7.1 and~10.1]{Opper2025}.
\end{proof}

\begin{lemma}\label{lem:pruefer-rank-DPic}
Let \(\mathbb S=(S,M,\eta)\) be a graded surface with stops whose
partially wrapped Fukaya category is smooth and proper.
Suppose that \(S\) has genus \(g\geq2\) and \(b\) boundary
components. For every prime \(\ell\),
\begin{equation}\label{eq:pruefer-rank-DPic}
 \prk_\ell D(\mathbb S)=2g+b-1.
\end{equation}
\end{lemma}

\begin{proof}
Set \(D=D(\mathbb S)\) and \(N=N(\mathbb S)\). We first note that
\(D/N\simeq\MCGgr(\mathbb S)\) is virtually torsion free.
Indeed, the
ordinary mapping class group is virtually torsion free, and the kernel of
\eqref{eq:grading-extension} is infinite cyclic.  After replacing a
torsion-free finite-index subgroup downstairs by its full inverse image,
the same is true for the graded extension.

Let \(P\leq D\) be isomorphic to \((C_{\ell^\infty})^r\).  The image of
\(P\) in $D/N$ is torsion and divisible. Since \(D/N\) is virtually torsion
free, every torsion subgroup of \(D/N\) is finite. More precisely,  after passing to
a normal torsion-free subgroup of finite index, this image injects into the finite
quotient. Note that a finite divisible group is trivial. Therefore the image of \(P\) in \(D/N\) is trivial, so
\(P\le N\).

Write \(N=U\rtimes T\), where
\(U\simeq\mathbb G_a^{\phi_A(1,1)}\) and
\(T=H^1(S;\kk^\times)\). Since \(U\) is torsion free,
\(P\cap U=1\), so projection to \(T\) is injective on
\(P\). In particular,
\[
 P[\ell]\simeq(\Z/\ell)^r
 \hookrightarrow T[\ell]
 \simeq(\Z/\ell)^{2g+b-1}.
\]
Thus \(r\leq2g+b-1\).
Conversely, the local-system torus
\(
 H^1(S;\kk^\times)\simeq(\kk^\times)^{2g+b-1}
\)
contains \((C_{\ell^\infty})^{2g+b-1}\), since \(\kk\) is algebraically
closed of characteristic zero. Hence
\(\prk_\ell D\geq2g+b-1\). Together with the upper bound,
this proves \eqref{eq:pruefer-rank-DPic}.
\end{proof}

\begin{lemma}\label{lem:virtual-centre-DPic}
Under the assumptions of Lemma~\ref{lem:pruefer-rank-DPic}, put
\[
 \overline D(\mathbb S)
 =D(\mathbb S)/\operatorname{DA}\bigl(D(\mathbb S)\bigr).
\]
Then
\begin{equation}\label{eq:virtual-centre-DPic}
 \rvz\overline D(\mathbb S)=b+1.
\end{equation}
\end{lemma}

\begin{proof}
By Lemma~\ref{lem:characteristic-geometrization-kernel} and
Theorem~\ref{thm:opper-geometrization}, we may identify
\(\overline D(\mathbb S)\) with \(\MCGgr(\mathbb S)\).  Let \(L\le\MCGgr(\mathbb S)\) be a
finite-index subgroup and let \(z\in Z(L)\). 
Abelian subgroups of \(H_\eta\) are finitely generated
\cite{BirmanLubotzkyMcCarthy}.  Since the kernel of the graded extension is
cyclic, the same holds for abelian subgroups of \(\MCGgr(\mathbb S)\).
In particular, \(Z(L)\) is finitely generated and has no non-trivial
divisible subgroup.

Let \(\bar z\in H_\eta\) denote the ungraded image of $z$. Every separating
essential curve \(c\subset S\) satisfies \(c^\vee=0\), so
\eqref{eq:twist-line-field} gives \(T_c\in H_\eta\).
Choose any graded
lift \(\widetilde T_c\) of \(T_c\). Since \(L\) has finite index in
\(\MCGgr(\mathbb S)\), there exists \(n(c)>0\) such that
\(\widetilde T_c^{\,n(c)}\in L\). Its ungraded image is
\(T_c^{n(c)}\). By centrality of \(z\), we obtain
\(
 T_{\bar z(c)}^{n(c)}=T_c^{n(c)}\).
Nonzero powers of twists determine their underlying curves, so
\(\bar z\) fixes every separating curve.

Lemma~\ref{lem:separating-fixer} then implies that the image of \(Z(L)\) in
\(H_\eta\) is contained in a finite extension of \(A_M\simeq\Z^b\). On the other hand, 
the kernel of \(\MCGgr(\mathbb S)\to H_\eta\) is generated by the grading
shift. Hence,
\[
 \dim_{\mathbb Q}\bigl(Z(L)\otimes_{\Z}\mathbb Q\bigr)\leq b+1.
\]
This gives the upper bound for the reduced centre.

For the reverse inequality, let \(H_\eta^p\leq H_\eta\)
be the finite-index subgroup fixing every stop individually.
Choose a stop \(p\) and a lift of the tangent line
\(T_p\partial S\) to the grading cover. Each mapping class
in \(H_\eta^p\) has a unique graded lift fixing this chosen
point of the cover. This normalization is preserved under
composition and isotopy, and therefore splits the graded
extension over \(H_\eta^p\).

Boundary twists are central in \(H_\eta^p\) and generate
a free abelian subgroup of rank \(b\). In the normalized
splitting, their lifts, together with the deck translation
\([1]\), generate a central subgroup isomorphic to
\(\Z^{b+1}\). Since the centre is finitely generated, this
lattice survives in the reduced centre. This proves
\eqref{eq:virtual-centre-DPic}.

\end{proof}

\begin{proof}[Proof of Theorem~\ref{thm:derived-reconstruction} \textup{(i)}]
Let
\[
 p=\prk_\ell\DPic(A),
 \qquad
 z=\rvz\!\left(
   \DPic(A)/\operatorname{DA}(\DPic(A))
 \right),
\]
where \(\ell\) is any prime.  Lemmas~\ref{lem:pruefer-rank-DPic} and
\ref{lem:virtual-centre-DPic} give
\begin{equation*}
 b=z-1,
 \qquad
 g=\frac{p-b+1}{2}=\frac{p-z+2}{2}.
\end{equation*}
The same formulas hold for \(A'\).  An abstract group isomorphism preserves
both \(p\) and \(z\), so \(g(S_A)=g(S_{A'})\) and
\(|\pi_0(\partial S_A)|=|\pi_0(\partial S_{A'})|\).  The classification of
compact connected oriented surfaces thus gives an orientation-preserving
diffeomorphism of the underlying surfaces.
\end{proof}

\subsection{Rigidity of line-field stabilizers}

For part~\textup{(ii)} of Theorem \ref{thm:derived-reconstruction}, we use the action on the torsion of the
abelianized geometrisation kernel. We first recall the required
rigidity theorem and establish the properties of line-field
stabilizers.

For a compact oriented surface \(S\), let \(\mathcal C_s(S)\) be the
simplicial complex whose vertices are isotopy classes of essential
separating simple closed curves and whose simplices are represented by
pairwise disjoint curves. Here essential means neither null-homotopic
nor isotopic to a boundary component. Let \(\Mod^*(S)\) denote the
extended mapping class group.

\begin{theorem}[{\cite[Theorem~1.2]{Kida}}]
\label{thm:kida-rigidity}
Suppose that either \(g(S)\geq3\), or that \(g(S)=2\) and \(S\) has at
least two boundary components.  The natural homomorphism
\[
 \Mod^*(S)\longrightarrow\operatorname{Aut}\bigl(\mathcal C_s(S)\bigr)
\]
is an isomorphism.
\end{theorem}

Note that the diffeomorphism supplied by this theorem may reverse orientation.
Lemma~\ref{lem:orientation-correction} shows that its orientation can be
corrected without changing the homotopy class of the line field. The intrinsic recognition of separating-twist axes in the relevant
homological line-field kernel is proved in
Lemma~\ref{lem:twists-in-chillingworth-kernel}. Theorem~\ref{thm:kida-rigidity} does not include
\(\Sigma_{2,1}\) and in this case the two-torsion representation
distinguishes the line-field orbits by
Lemma~\ref{lem:S21-two-torsion}. These three lemmas are proved in
Section~\ref{appsubsec:line-field-stabilizers}.

\begin{lemma}\label{lem:stabilizer-determines-line-field}
Let \(S\) be a compact connected oriented surface of genus
at least two with non-empty boundary.
If two line fields \(\eta,\theta\) satisfy
\[
 \Stab_{\PMod(S)}(\eta)=\Stab_{\PMod(S)}(\theta),
\]
then \(\eta\simeq\theta\).
\end{lemma}

\begin{proof}
We compare the two stabilizers using weighted products of Dehn
twists about genus-one bounding pairs, to which
\eqref{eq:weighted-bounding-pair} applies.

Let \(a,b\) be an oriented bounding pair which bounds a genus-one
subsurface, and set
\[
 A=\omega_\eta(a),\quad B=\omega_\eta(b),
 \qquad
 A'=\omega_\theta(a),\quad B'=\omega_\theta(b).
\]
By \eqref{eq:weighted-bounding-pair} and the assumed
equality of the two pure stabilizers, we have
\begin{equation*}
 \{(u,v)\in\Z^2:uA-vB=0\}
 =\{(u,v)\in\Z^2:uA'-vB'=0\}.
\end{equation*}
Hence the integer vectors \((A,B)\) and \((A',B')\) are proportional over \(\mathbb Q\).  Neither is zero, because
\eqref{eq:bounding-pair-winding-difference} gives
\[
 A-B=-4,
 \qquad A'-B'=-4.
\]
The proportionality factor is consequently one, and
\begin{equation}\label{eq:winding-agrees-on-bp}
 \omega_\eta(a)=\omega_\theta(a),
 \qquad
 \omega_\eta(b)=\omega_\theta(b).
\end{equation}

Every nonseparating simple closed curve on a surface of genus at least two
belongs to a genus-one bounding pair.
Indeed, after cutting along the curve, the remaining surface
has genus at least one. One can choose a genus-one subsurface
with two boundary components adjacent to one of the new
boundary circles and disjoint from the original boundary.
Hence
\eqref{eq:winding-agrees-on-bp} holds for every nonseparating curve.  Write
\(\theta-\eta=\alpha\in H^1(S;\Z)\).  By
\eqref{eq:winding-difference}, \(\alpha\) vanishes on the homology class
of every nonseparating curve.  These classes span \(H_1(S;\Z)\).  Indeed,
a symplectic system gives the handle classes, while the class of each
boundary component is the difference between a nonseparating curve and
its band sum with that boundary component.  Thus \(\alpha=0\), and hence
\(\theta\simeq\eta\).
\end{proof}

\subsection{Reconstruction from the shift-pointed group}

We now combine the intrinsic group constructions with the rigidity results
above. We treat the one-holed genus-two surface first, since it lies
outside the range of Theorem~\ref{thm:kida-rigidity}.

\begin{proof}[Proof of Theorem~\ref{thm:derived-reconstruction} \textup{(ii)}]
Set
\[
 D=D(\mathbb S_A),\qquad D'=D(\mathbb S_{A'}),
 \qquad N=N(\mathbb S_A),\qquad N'=N(\mathbb S_{A'}).
\]
By Lemma~\ref{lem:characteristic-geometrization-kernel}, we have
\(\Phi(N)=N'\).  Thus \(\Phi\) induces an isomorphism of graded mapping
class groups.  Since \(\Phi([1])=[1]\), quotienting by the central grading
shift in \eqref{eq:grading-extension} gives an isomorphism
\(
 \phi:H_{\eta_A}\xrightarrow{\sim}H_{\eta_{A'}}\). Part~\textup{(i)} implies that \(S_A\) and \(S_{A'}\) have the same genus
and the same number of boundary components. We may therefore
choose an orientation-preserving diffeomorphism between them and
temporarily identify both with a fixed surface \(S=\Sigma_{g,b}\).
Denote by \(M,M'\) and \(\eta,\eta'\) the stop sets and line fields
transported to \(S\). 

Suppose first that \(S=\Sigma_{2,1}\).  By
\eqref{eq:ell-local-systems}, the characteristic group
\(
 V(\mathbb S)=N(\mathbb S)^{\mathrm{ab}}[2]
\)
is naturally \(H^1(S;\mathbb F_2)\).  Conjugation on it is the mod-two
homological representation of \(H_\eta\). After the stops are forgotten,
\(H_\eta\) maps onto \(\Stab_{\Mod(S)}(\eta)\), while its
boundary-rotation kernel acts trivially on cohomology. Its image is
therefore the group \(R_\eta\) of
Lemma~\ref{lem:S21-two-torsion}. The isomorphism \(\Phi\)
conjugates \(R_{\eta_A}\) to \(R_{\eta_{A'}}\), so the two
images have the same order.  Lemma~\ref{lem:S21-two-torsion} shows that the
two line fields have the same \(\sigma\)-invariant and, when defined, the
same Arf invariant.  Their boundary winding numbers are both \(-6\).
Theorem~\ref{thm:line-field-classification} now gives an
orientation-preserving diffeomorphism \(f:S_A\to S_{A'}\) with
\(f^*\eta_{A'}\simeq\eta_A\).  This proves the exceptional case.

We henceforth assume that either \(g\geq3\), or \(g=2\) and \(b\geq2\). Fix a prime \(\ell\), and set
\(
 L_\ell(\mathbb S)
 =\operatorname{Tor}_{\ell^\infty}
 \bigl(N(\mathbb S)^{\mathrm{ab}}\bigr)\).
Both \(N\) and \(L_\ell(\mathbb S)\) are characteristic in the data under
consideration.  Hence \(\phi\) identifies the kernels of the two
conjugation representations on \(L_\ell\).  By
\eqref{eq:ell-local-systems}, a mapping class acts trivially on
\(L_\ell\) precisely when it acts trivially on \(H_1(S;\Z)\).  An integral
matrix that is the identity on all \(\ell\)-power characters is congruent
to the identity modulo \(\ell^n\) for every \(n\), and hence is the
identity.  Thus \(\phi\) restricts to an isomorphism
\begin{equation}\label{eq:framed-torelli-isomorphism}
 J_\eta\xrightarrow{\sim}J_{\eta'},
 \qquad
 J_\eta
 =H_\eta\cap
 \ker\bigl(\Mod(S,M)\to\operatorname{Aut}H_1(S;\Z)\bigr).
\end{equation}

We next forget the stops in \eqref{eq:framed-torelli-isomorphism}.
Every element of \(J_\eta\) preserves each boundary component.
For \(b\geq2\), the subgroup of \(H_1(S;\Z)\)
generated by the oriented boundary classes is
\(\Z^b/\Z(1,\ldots,1)\), and the natural
permutation action on this lattice is faithful. Trivial action
on absolute homology excludes a non-trivial boundary
permutation. For \(b=1\), the assertion is automatic.  Recall that isotopies in \(\Mod^*(S)\) may move
points of \(\partial S\), and let
\(\PMod(S)\leq\Mod(S)\) be the subgroup preserving every boundary
component setwise.  The forgetful homomorphism
\[
 q_M:\Mod(S,M)\longrightarrow\Mod(S)\subset\Mod^*(S)
\]
forgets the stops.  Denote by \(\mathcal I(S)\) the Torelli group and by
\(\mathcal K(S)\) the Johnson kernel, generated by twists about separating
curves, in the conventions of \cite{Kida}.  We distinguish these from the
absolute homological kernel
\[
 \mathcal I_0(S)
 =\ker\bigl(\PMod(S)\to\operatorname{Aut}H_1(S;\Z)\bigr).
\]
If \(\theta-\eta=\alpha\in H^1(S;\Z)\) and
\(x\in\mathcal I_0(S)\), then
\begin{equation}\label{eq:torelli-displacement-independence}
 x^*\theta-\theta
 =(x^*\eta-\eta)+(x^*\alpha-\alpha)
 =x^*\eta-\eta.
\end{equation}
Every element of \(\mathcal I_0(S)\) fixes each boundary component and has a
representative preserving the finite stop set after an isotopy in boundary
collars.  Moreover, the kernel of \(q_M\)
preserves every line field.  Hence \(q_M(J_\eta)\) is exactly the subgroup
of \(\mathcal I_0(S)\) preserving \(\eta\).  Equation~\eqref{eq:torelli-displacement-independence} therefore shows that the
subgroup
\(
 \mathcal C_0(S)\coloneqq q_M(J_\eta)
\)
is independent of \(\eta\).  It is the kernel of the Chillingworth class
restricted to the absolute homological kernel \cite{Chillingworth}.
Separating twists preserve every line field, and the capping argument in
\cite[Section~2]{Kida} gives
\(\mathcal I_0(S)\subseteq\mathcal I(S)\). Consequently
\begin{equation*}
 \mathcal K(S)\subseteq\mathcal C_0(S)\subseteq\mathcal I(S).
\end{equation*}

The kernel of \(q_M:J_\eta\to\mathcal C_0(S)\) consists of boundary
rotations and is central in \(J_\eta\).  Conversely, if
\(z\in Z(J_\eta)\), then \(q_M(z)\) centralizes \(\mathcal C_0(S)\), hence
centralizes every separating twist.  Conjugating a twist shows that it
fixes every vertex of the separating-curve complex.  By
Theorem~\ref{thm:kida-rigidity}, the action on that complex is faithful in
the present range.  Hence \(q_M(z)=1\).  Thus
\begin{equation}\label{eq:center-pure-torelli}
 Z(J_\eta)=\ker(q_M|_{J_\eta}),
 \qquad J_\eta/Z(J_\eta)\simeq\mathcal C_0(S).
\end{equation}
The same statement holds for \(J_{\eta'}\).  Since \(\phi\) preserves
centres, it induces an isomorphism
\(
 \overline\phi:\mathcal C_0(S)\xrightarrow{\sim}\mathcal C_0(S)\).

Lemma~\ref{lem:twists-in-chillingworth-kernel} shows that
\(\overline\phi\) permutes the cyclic axes of separating twists.  Distinct
such axes commute exactly when their curves are disjoint, so
\(\overline\phi\) induces an automorphism of the separating-curve complex.
By Theorem~\ref{thm:kida-rigidity}, there is an element
\(f\in\Mod^*(S)\) inducing this automorphism. To show that
\(\overline\phi\) is conjugation by \(f\), take
\(x\in\mathcal C_0(S)\), a separating curve \(c\), and an integer
\(n\neq0\). In the identity
\[
 \overline\phi(xT_c^nx^{-1})
 =\overline\phi(x)\,\overline\phi(T_c^n)\,
  \overline\phi(x)^{-1},
\]
the left-hand side lies on the twist axis associated with
\(f(x(c))\), while the right-hand side lies on the twist axis
associated with \(\overline\phi(x)(f(c))\). Hence
\[
 f(x(c))=\overline\phi(x)(f(c))
\]
for every separating curve \(c\).  By
Theorem~\ref{thm:kida-rigidity}, the action of \(\Mod^*(S)\) on the
separating-curve complex is faithful in the present range.  It follows that
\(\overline\phi(x)=fxf^{-1}\) for all \(x\in\mathcal C_0(S)\).  Returning to
the central extensions \eqref{eq:center-pure-torelli}, we obtain
\begin{equation}\label{eq:kida-geometricity-on-J}
 q_{M'}\bigl(\phi(x)\bigr)=f q_M(x)f^{-1}
 \qquad\text{ for }x\in J_\eta.
\end{equation}
Undoing the preliminary identification of the two surfaces, \(f\) is a
diffeomorphism \(S_A\to S_{A'}\), not yet known to preserve orientation.

Equation~\eqref{eq:kida-geometricity-on-J} extends from \(J_\eta\) to the
whole line-field stabilizer.  Indeed, take \(y\in H_\eta\).  Since
\(J_\eta\triangleleft H_\eta\), comparison of
\eqref{eq:kida-geometricity-on-J} for \(x\) and \(yxy^{-1}\) shows that
\[
 q_{M'}\bigl(\phi(y)\bigr)^{-1}f q_M(y)f^{-1}
\]
 centralizes \(\mathcal C_0(S)\).  It is trivial because an element
centralizing \(\mathcal C_0(S)\) fixes every separating curve, and the same
faithfulness statement applies.  Therefore
\begin{equation}\label{eq:geometricity-on-stabilizer}
 q_{M'}\bigl(\phi(y)\bigr)=f q_M(y)f^{-1}
 \qquad\text{ for }y\in H_\eta.
\end{equation}

For a line field \(\lambda\) and any finite stop set \(P\) meeting every
boundary component, put
\(H_{\lambda,P}=\Stab_{\Mod(S,P)}(\lambda)\).  Then
\[
 q_P(H_{\lambda,P})\cap\PMod(S)=\Stab_{\PMod(S)}(\lambda).
\]
Indeed, every class in \(\PMod(S)\) has a representative which is the
identity on boundary collars and therefore preserves \(P\).  The kernel of
\(q_P\) preserves every line field.  Thus the displayed group is independent
of the chosen stop set.

Since \(\phi\) is surjective, equation
\eqref{eq:geometricity-on-stabilizer} gives
\[
 q_{M'}(H_{\eta'})=f q_M(H_\eta)f^{-1}.
\]
Intersecting with the pure mapping class groups of the surfaces without stops and putting
\(\theta=f^*\eta'\), we obtain
\begin{equation}\label{eq:equal-pure-coarse-stabilizers}
 \Stab_{\PMod(S)}(\eta)
 =\Stab_{\PMod(S)}(\theta).
\end{equation}

Lemma~\ref{lem:stabilizer-determines-line-field}, applied to
\eqref{eq:equal-pure-coarse-stabilizers}, gives
\(f^*\eta_{A'}\simeq\eta_A\).
It remains only to correct the orientation of \(f\).  If \(f\) preserves
orientation, there is nothing to prove.  Otherwise,
Lemma~\ref{lem:orientation-correction} gives an orientation-reversing
diffeomorphism \(\rho:S_A\to S_A\) satisfying
\(\rho^*\eta_A\simeq\eta_A\).  Replacing \(f\) by \(f\circ\rho\), we obtain
an orientation-preserving diffeomorphism and
\[
 (f\circ\rho)^*\eta_{A'}
 =\rho^*f^*\eta_{A'}
 \simeq\rho^*\eta_A
 \simeq\eta_A.
\]
This proves the theorem.
\end{proof}

\begin{corollary}\label{cor:geometric-reduction}
Let \(\mathbb S=(S,M,\eta)\) and
\(\mathbb S'=(S',M',\eta')\) satisfy the hypotheses of
Theorem~\ref{thm:derived-reconstruction}, and suppose that both
surfaces have at least two boundary components. A shift-preserving
isomorphism
\[
 \Phi:D(\mathbb S)\xrightarrow{\sim}D(\mathbb S')
\]
identifies the geometrisation kernels and induces an isomorphism
\(\phi:H_\eta\to H_{\eta'}\). It also identifies
\[
 J_\eta=H_\eta\cap
 \ker\bigl(\Mod(S,M)\to\operatorname{Aut}H_1(S;\Z)\bigr)
\]
with \(J_{\eta'}\), and carries \(C=Z(J_\eta)\) onto
\(C'=Z(J_{\eta'})\). There is a diffeomorphism
\(f:S\to S'\), possibly orientation reversing, such that, for every $y\in H_\eta$, we have
\[
 f^*\eta'\simeq\eta,
 \qquad
 q_{M'}\bigl(\phi(y)\bigr)=f q_M(y)f^{-1},
\]
where \(q_M\) and \(q_{M'}\) forget the stops.
\end{corollary}

\begin{proof}
By Lemma~\ref{lem:characteristic-geometrization-kernel}, the
geometrisation kernels are identified, and quotienting by the shift gives \(\phi\).
The conjugation action on the torsion subgroup of the abelianized kernel
yields an identification \(J_\eta\simeq J_{\eta'}\). The centre calculation
\eqref{eq:center-pure-torelli} then identifies \(C\) with \(C'\).
The rigidity argument leading to
\eqref{eq:geometricity-on-stabilizer} produces the diffeomorphism \(f\) and the displayed
conjugation formula. Finally,
Lemma~\ref{lem:stabilizer-determines-line-field} gives
\(f^*\eta'\simeq\eta\).
\end{proof}

Thus the shift-pointed derived Picard group determines the
underlying graded surface, without any additional restriction on
the boundary winding numbers or parity of the line field.
The next section shows that the stop data need not be determined,
even when the Serre element is also distinguished.

\section{Fibres of the pointed derived Picard invariant}
\label{sec:fibres}

In this section, we construct graded surfaces with different stop distributions whose
derived Picard groups are isomorphic by maps preserving both the shift and
the Serre functor. The construction gives a countably infinite fibre of
\(\mathfrak D\), together with arbitrarily large finite families in one
fibre having the same Grothendieck rank. We finish with examples in genera
zero and one for which even the underlying graded surfaces differ.
For a graded gentle algebra \(A\) with surface model $\mathbb S$, we also write
\[
 \mathfrak D(A)=(\DPic(A),[1],\mathsf S_A),
\]
where \(\mathsf S_A\) is the class of its Serre functor. Morita equivalence
identifies this triple with \(\mathfrak D(\mathbb S)\). The result is stated as follows. 

\begin{theorem}\label{thm:countable-fibres}
For every integer \(g\geq2\), there is a family
\(\{A_k\}_{k\in\Z_{\geq1}}\) of connected graded gentle algebras which are
finite-dimensional, homologically smooth, and proper. After the stops are
forgotten, their surface models have the same underlying graded surface
\((\Sigma_{g,3},\eta)\). Their stop distributions are
\[
 \bigl(k(2g-1),k(2g-1)+1,1\bigr).
\]
They satisfy
\[
 \mathfrak D(A_k)\simeq\mathfrak D(A_l)
 \qquad\text{for all }k,l\geq1,
\]
whereas the categories \(\Perf(A_k)\) are pairwise non-equivalent. The
line field \(\eta\) is induced by a vector framing, and the fibre containing
this family is countably infinite.

Moreover, for every \(R\geq1\), there is a fibre of \(\mathfrak D\)
containing more than \(R\) pairwise non-equivalent categories of genus
\(g\) with the same rank of the Grothendieck group.
\end{theorem}

\subsection{A family with three boundary components}
Fix \(g\geq2\), put \(q=2g+1\), and let
\(S=\Sigma_{g,3}\), with
\(\partial S=B_1\sqcup B_2\sqcup B_3\).
Orient the boundary with \(S\) on its left.
Since \(\chi(S)=-q\), the realization statement following
\eqref{eq:framing-coherence} allows us to choose a vector framing
\(\xi\) with
\[
 \bigl(\omega_\xi(B_1),\omega_\xi(B_2),\omega_\xi(B_3)\bigr)
 =(1,-1,-q).
\]
Set \(\eta=\mathbb R\xi\). By \eqref{eq:winding-doubling},
\begin{equation}\label{eq:chosen-line-windings}
 \bigl(\omega_\eta(B_1),\omega_\eta(B_2),\omega_\eta(B_3)\bigr)
 =(2,-2,-2q).
\end{equation}

Let \(T_i=T_{B_i}\) be the positive boundary twists, and write
\[
 F_0=
 \bigl\{
 f\in\Mod(S,\partial S)
 \,\bigm|\,
 f^*\xi\simeq\xi\text{ relative to }\partial S
 \bigr\},
 \qquad
 a=T_3T_1^{-q}T_2^q.
\]
In the definition of \(F_0\), mapping classes are represented by
diffeomorphisms equal to the identity on boundary collars.
Thus \(F_0\) preserves the relative homotopy class of the framing.
Lemma~\ref{lem:framed-stabilizer-presentation} shows that
\(a\in Z(F_0)\) and gives the stabilizer decompositions used below.
Both \(F_0\) and \(a\) are fixed before the stops are chosen.

For \(r,s\geq1\), let \(M_{r,s}\) have \(r,s,1\) stops on
\(B_1,B_2,B_3\), respectively, and put
\[
 \mathbb S_{r,s}=(S,M_{r,s},\eta).
\]
Let \(\rho_i\) denote the positive one-step rotation on \(B_i\),
supported in a boundary collar. These satisfy
\[
 \rho_1^r=T_1,\qquad
 \rho_2^s=T_2,\qquad
 \rho_3=T_3.
\]
The stabilizer and Serre calculations needed for the following
lemma are proved in Appendix~\ref{app:serre-family}.

\begin{lemma}\label{lem:common-pointed-group-family}
For every \(r,s\geq1\), the graded surface \(\mathbb S_{r,s}\)
admits a connected finite-dimensional, homologically smooth and
proper graded gentle algebra model \(A_{r,s}\).
Put \(N=H^1(S;\kk^\times)\), with the natural action of \(F_0\)
induced by mapping classes.
There are shift-preserving identifications
\[
 \DPic(A_{r,s})\simeq
 G=(N\rtimes F_0)\times
   \Z\langle\rho_1\rangle\times
   \Z\langle\rho_2\rangle\times
   \Z\langle[1]\rangle
\]
for all \(r,s\geq1\), where the two rotation factors are identified
with fixed copies of \(\Z\) as \(r,s\) vary.
We use the same symbols for \(a,\rho_1,\rho_2\) and their
corresponding elements in \(G\).
Let \(\mathsf S_{r,s}\in\DPic(A_{r,s})\) be the element represented
by the Serre functor of \(\Perf(A_{r,s})\). Under these identifications,
\begin{equation}\label{eq:normalized-serre-family}
 \mathsf S_{r,s}
 =a^{-1}\rho_1^{-(qr+1)}\rho_2^{qs-1}[1]^{e_g},
\end{equation}
where \(e_g\in\Z\) is independent of \(r,s\).
\end{lemma}

\begin{proof}
Put
\[
 H_{r,s}=\Stab_{\Mod(S,M_{r,s})}(\eta).
\]
Lemma~\ref{lem:framed-stabilizer-presentation}, whose subgroup
\(F_0\) is the relative framing stabilizer defined above, gives
the decompositions
\eqref{eq:F0-splitting}--\eqref{eq:Hrs-splitting}.

There is also a canonical splitting of the grading extension for this
family.  Let \(p\) be the unique stop on \(B_3\) and choose a lift of the
tangent line \(T_pB_3\) to the grading cover.  Every element of
\(H_{r,s}\) fixes this tangent line and has a unique graded lift fixing the
chosen point above it.  The product of two normalized lifts is again
normalized, so
\begin{equation}\label{eq:graded-extension-splitting-family}
 \MCGgr(\mathbb S_{r,s})
 \simeq H_{r,s}\times\Z\langle[1]\rangle.
\end{equation}

Choose a minimal full formal system of graded arcs and let \(A_{r,s}\)
be the basic total graded endomorphism algebra of their direct sum.  The existence and the smoothness and
properness criteria are given in \cite[Proposition~2.5(1),(3),(4)]{Opper2025}.
The ribbon graph dual to the chosen full system of arcs is connected because
\(S\) is connected and the complementary regions are discs.  Equivalently,
the underlying graph of the gentle quiver of \(A_{r,s}\) is connected.
Hence \(A_{r,s}\) is a connected, finite-dimensional, homologically smooth,
and proper graded gentle algebra and
\[
 \Perf(A_{r,s})\simeq\W(\mathbb S_{r,s}).
\]

Put \(N=H^1(S;\kk^\times)\).  Since all three windings in
\eqref{eq:chosen-line-windings} are non-zero, the additive factor in
Theorem~\ref{thm:opper-geometrization} is absent. Thus
\(N=N(\mathbb S_{r,s})\) for every \(r,s\). Equations~\eqref{eq:derived-picard-formula},
\eqref{eq:Hrs-splitting}, and
\eqref{eq:graded-extension-splitting-family} identify every derived
Picard group with
\begin{equation}\label{eq:common-derived-picard-group}
 G=(N\rtimes F_0)\times\Z^2\times\Z\langle[1]\rangle.
\end{equation}
Only \(F_0\) acts non-trivially on \(N\).  The collar rotations and the
shift act trivially and the former are isotopic to the identity after the
stops are forgotten.  Moreover, \(a=T_3T_1^{-q}T_2^q\) is a product of
boundary twists and acts trivially on \(H_1(S;\Z)\), hence on \(N\).
Since \(a\in Z(F_0)\), it is central in \(N\rtimes F_0\).  In particular,
the group in
\eqref{eq:common-derived-picard-group}, with its distinguished shift, is
independent of \((r,s)\).

It remains to identify the actual Serre element, since geometrisation
alone leaves a possible local-system factor and a shift. The argument
in Lemma~\ref{lem:framed-family-serre-normalization} tests that factor
on rank-one band objects supported on zero-winding curves. By
Lemma~\ref{lem:gradable-curves-span}, their classes generate
\(H_1(S;\Z)\), so triviality of these holonomies eliminates the
entire local-system factor. The same band calculation determines the
shift exponent from the normalized action over the fixed core. Hence
it is independent of \(r,s\). This gives
\eqref{eq:normalized-serre-family} and completes the proof.
\end{proof}

\subsection{Matching the Serre elements}

The common group presentation reduces the comparison of Serre elements
to a change of basis in the central rotation factor. For \(r,s\geq1\), set
\[
 v_{r,s}=(qr+1,1-qs)\in\Z^2.
\]
By \eqref{eq:normalized-serre-family}, the rotation coordinate of
\(\mathsf S_{r,s}\) is \(-v_{r,s}\). In particular, it is primitive
if and only if \(v_{r,s}\) is primitive.

\begin{lemma}\label{lem:primitive-serre-vectors}
If \(v_{r,s}\) and \(v_{r',s'}\) are primitive, then
\[
 \mathfrak D(\mathbb S_{r,s})\simeq
 \mathfrak D(\mathbb S_{r',s'}).
\]
\end{lemma}

\begin{proof}

Use the common presentation in
Lemma~\ref{lem:common-pointed-group-family}. For a primitive vector
\(v=(a,b)^T\), choose integers \(u,w\) with \(ua+wb=1\), and put
\[
 B(v)=\begin{pmatrix}u&w\\-b&a\end{pmatrix}\in\operatorname{SL}_2(\Z).
\]
Then \(B(v)v=(1,0)^T\). Consequently the explicit integral matrix
\[
 P=B(v_{r',s'})^{-1}B(v_{r,s})
\]
satisfies \(Pv_{r,s}=v_{r',s'}\), and hence
\(P(-v_{r,s})=-v_{r',s'}\). The automorphism of
\[
 (N\rtimes F_0)\times\Z^2\times\Z\langle[1]\rangle
\]
which is the identity on \(N\rtimes F_0\) and on the shift factor,
and is given by \(P\) on \(\Z^2\), preserves \([1]\). Formula~\eqref{eq:normalized-serre-family} shows that it sends
\(\mathsf S_{r,s}\) to \(\mathsf S_{r',s'}\).
\end{proof}

\begin{proof}[Proof of Theorem~\ref{thm:countable-fibres}]
Fix \(g\geq2\), put \(q=2g+1\), and use the line field and algebras
from Lemma~\ref{lem:common-pointed-group-family}. Set
\[
 d=q-2=2g-1,\qquad m_k=kd\quad\text{ for } k\geq1,
\]
and take \((r,s)=(m_k,m_k+1)\). Put
\(A_k=A_{m_k,m_k+1}\). Then
\[
 \gcd\bigl(qm_k+1,q(m_k+1)-1\bigr)
 =\gcd(qm_k+1,q-2)=1.
\]
Indeed, \(d\mid m_k\), and hence \(qm_k+1\equiv1\pmod d\).
Thus every vector \(v_{m_k,m_k+1}\) is primitive.
Lemma~\ref{lem:primitive-serre-vectors} identifies all the triples
\(\mathfrak D(\mathbb S_{m_k,m_k+1})\).

The total number of stops is \(2m_k+2\), so these marked surfaces
are pairwise non-isomorphic. Theorem~\ref{thm:complete-surface-invariant}
therefore gives
\[
 \Perf(A_{m_k,m_k+1})\not\simeq
 \Perf(A_{m_l,m_l+1})
 \qquad\text{for }k\neq l.
\]
The fibre is countably infinite. It contains the displayed infinite
family, while the set of isomorphism classes of graded surfaces with
stops is countable. To see the latter assertion, note that there are
countably many topological types of compact marked surfaces. On a fixed
surface the homotopy classes of line fields form a torsor over the
finitely generated group \(H^1(S;\Z)\).

For the fixed-rank assertion, let \(R\geq1\) and choose a prime
\(p>R+2\) which does not divide \(q\). Consider
\[
 (r,s,1)=(r,p-r,1),
 \qquad 1\leq r\leq p-1,
\]
and omit the unique residue satisfying \(qr+1\equiv0\pmod p\).
There remain \(p-2>R\) values. For each of them,
\[
 \gcd(qr+1,q(p-r)-1)
 =\gcd(qr+1,qp)=1.
\]
The corresponding vectors \(v_{r,p-r}\) are primitive, so all these
surfaces lie in one fibre of \(\mathfrak D\). Their total number of stops is
\(p+1\), and \eqref{eq:charge-rank} gives the common rank
\[
 \rk K(\mathbb S_{r,p-r})=2g+p+2.
\]
The boundary winding numbers \(2,-2,-2q\) are pairwise distinct.
Every graded diffeomorphism therefore preserves each boundary
component. It follows that two of these stopped graded surfaces are
isomorphic only when their values of \(r\) agree. The complete surface
invariant now shows that the corresponding categories are pairwise
non-equivalent.
\end{proof}
\begin{example}
For \(g=2\), the first three stop distributions in the infinite family
are \((3,4,1)\), \((6,7,1)\), and \((9,10,1)\). Their Serre elements
have rotation coordinates \((-16,19)\), \((-31,34)\), and
\((-46,49)\), respectively. Each is primitive, so
Lemma~\ref{lem:primitive-serre-vectors} identifies the three shift-
and Serre-pointed derived Picard groups. Their Grothendieck ranks
are \(13\), \(19\), and \(25\), and hence the categories are pairwise
non-equivalent.
\end{example}

\subsection{Examples in genera zero and one}

\begin{example}\label{ex:low-genus-failure}
For genus zero, consider a pair of pants with one stop on each boundary
component. Choose line fields \(\eta\) and \(\eta'\) with respective
boundary winding numbers \((2,4,-8)\) and \((2,6,-10)\). These two
triples of boundary winding numbers satisfy
\eqref{eq:line-field-boundary-coherence}, and hence are realized by line
fields. In each triple, the three entries are non-zero and pairwise
distinct. The corresponding line-field stabilizer therefore fixes every
boundary component. The pure mapping class group of a pair of pants with
one stop on each boundary component is freely generated by the three
boundary twists. Consequently,
\[
 H_\eta\simeq H_{\eta'}\simeq\Z^3.
\]
Choose one of the three stops and a lift of its tangent line to the grading
cover. The unique graded lift fixing this point splits the grading extension
\eqref{eq:grading-extension}. Since boundary twists act trivially on
absolute homology, Theorem~\ref{thm:opper-geometrization} gives, for both
line fields,
\begin{equation}\label{eq:genus-zero-common-group}
 D(\mathbb S)\simeq
 H^1(\Sigma_{0,3};\kk^\times)\times\Z^3
 \times\Z\langle[1]\rangle
 \simeq(\kk^\times)^2\times\Z^4.
\end{equation}
Let \(R_\partial=T_1T_2T_3\). The underlying ungraded mapping class
of the Serre functor is \(R_\partial^{-1}\). Thus, in the coordinates
of \eqref{eq:genus-zero-common-group}, the two Serre elements have
the form
\[
 \mathsf S=nR_\partial^{-1}[1]^e,
 \qquad
 \mathsf S'=n'R_\partial^{-1}[1]^{e'},
\]
where \(n,n'\in(\kk^\times)^2\) and \(e,e'\in\Z\). The vector
\((-1,-1,-1)\) is primitive in \(\Z^3\).
Choose homomorphisms \(\mu,\lambda\colon\Z^3\to\Z\) such that
\[
 \mu(R_\partial^{-1})=1,
 \qquad
 \lambda(R_\partial^{-1})=e'-e.
\] 
Writing \(c=n'n^{-1}\), the automorphism
\[
 (u,x,j)\longmapsto
 \bigl(u c^{\mu(x)},x,j+\lambda(x)\bigr)
\]
fixes \([1]\) and sends \(\mathsf S\) to \(\mathsf S'\). The two
shift- and Serre-pointed derived Picard groups are therefore isomorphic.

For genus one, realize \(\Sigma_{1,1}\) by removing an
open disc from the flat torus \(\mathbb R^2/\Z^2\),
and place one stop on its boundary. Restrict a constant
nowhere-vanishing vector field \(\xi_0\) from the torus
and put \(\eta_0=\mathbb R\xi_0\). Standard horizontal
and vertical nonseparating curves have winding number zero.
Their Dehn twists preserve \(\eta_0\) by
\eqref{eq:twist-line-field} and these twists generate the
mapping class group of the one-holed torus. Thus \(\eta_0\)
is invariant under this group. Since the group acts
transitively on nonseparating curves, every such curve has
zero winding. Boundary coherence gives
\(\omega_{\eta_0}(\partial\Sigma_{1,1})=-2\).

For a primitive class
\(\alpha\in H^1(\Sigma_{1,1};\Z)\), set
\(
 \eta_1=\eta_0+\alpha\), and
 \(\eta_2=\eta_0+2\alpha\).  It follows that, for
\(d\geq1\),
\[
 f^*(\eta_0+d\alpha)\simeq\eta_0+d\alpha
 \quad\Longleftrightarrow\quad
 f^*\alpha=\alpha.
\]
Since \(H^1(\Sigma_{1,1};\Z)\) is torsion free, the stabilizers of
\(\eta_1\) and \(\eta_2\) coincide. Denote their common value by
\(H_\alpha\).

The stabilizer of a primitive class in \(\operatorname{SL}_2(\Z)\) is
infinite cyclic. Its inverse image in the mapping class group of the
one-holed torus fits into a central extension
\[
 1\longrightarrow\Z\langle T_{\partial}\rangle
 \longrightarrow H_\alpha
 \longrightarrow\Z
 \longrightarrow1.
\]
This extension splits because \(H^2(\Z;\Z)=0\). Choosing a lift \(U\) of
the positive generator gives
\[
 H_\alpha\simeq
 \Z\langle U\rangle\times\Z\langle T_{\partial}\rangle.
\]
In particular, \(T_{\partial}\) is primitive in \(H_\alpha\). Normalizing
graded lifts at the unique stop again splits
\eqref{eq:grading-extension}. The two derived Picard groups are therefore
identified with the same semidirect product
\begin{equation}\label{eq:genus-one-common-group}
 \bigl(H^1(\Sigma_{1,1};\kk^\times)\rtimes H_\alpha\bigr)
 \times\Z\langle[1]\rangle.
\end{equation}
The underlying ungraded mapping class of each Serre functor is
\(T_\partial^{-1}\). Hence their Serre elements have the form
\[
 \mathsf S_d=n_dT_\partial^{-1}[1]^{e_d},
 \qquad d=1,2,
\]
where
\[
 n_d\in H^1(\Sigma_{1,1};\kk^\times)^{H_\alpha}.
\]
The invariance of \(n_d\) follows from the centrality of the Serre functor.
Since \(T_\partial^{-1}\) is primitive, the same shear construction
as in genus zero gives an automorphism of
\eqref{eq:genus-one-common-group} that fixes \([1]\) and sends
\(\mathsf S_1\) to \(\mathsf S_2\). Hence the two shift- and
Serre-pointed derived Picard groups are isomorphic.

The corresponding underlying graded surfaces are nevertheless different. In
genus zero, their unordered collections of boundary winding numbers are
distinct. In genus one, \eqref{eq:genus-one-line-field-invariant} gives
\(
 \widetilde{\mathcal A}(\eta_1)=1,
 \widetilde{\mathcal A}(\eta_2)=2
\). Theorem~\ref{thm:line-field-classification} shows that neither pair of
underlying graded surfaces is isomorphic. A minimal full formal system of graded arcs on each
surface produces a connected graded gentle algebra that is
finite-dimensional, homologically smooth, and proper, as in the proof of
Lemma~\ref{lem:common-pointed-group-family}.
Theorem~\ref{thm:complete-surface-invariant} also shows that the two categories in
each pair are non-equivalent.
\end{example}

% --------------------------------------------------------------------------
\section{Grothendieck-group representations}
\label{sec:grothendieck-actions}

In this section, we first study the Serre representation. For a line field induced by a
vector field, a unipotent power recovers the number of boundary components
and the characteristic polynomial recovers their stop numbers. We then
consider the representation of the whole derived Picard group. For arbitrary
line fields in genus at least two, its fibres have cardinality at most four.
We determine when this invariant gives complete reconstruction and give
examples for each exceptional number of boundary components.

The relation between the AG-invariant and the Coxeter polynomial for
ordinary gentle algebras is described in
\cite[Theorem~6.17]{JimenezGonzalezMroz}. See also
\cite[Section~3]{ChangElaginSchroll} for the Coxeter transformation and
Serre entropy. The argument below uses the nilpotent part of boundary
rotation in addition to its characteristic polynomial. The Serre
reconstruction uses only rational linear data.

\subsection{The Serre representation}
The main theorem of this subsection is as follows, and the proof will be separated into two steps.
\begin{theorem}\label{thm:serre-reconstruction}
Let \(\mathbb S=(S,M,\eta)\) and
\(\mathbb S'=(S',M',\eta')\) be compact graded surfaces with stops, with at
least one stop on every boundary component, and suppose that their
partially wrapped Fukaya categories are smooth and proper.  Assume that
\(\eta=\mathbb R\xi\) and \(\eta'=\mathbb R\xi'\) are induced by
nowhere-vanishing vector fields.  If a \(\mathbb Q\)-linear isomorphism
\[
 K(\mathbb S)\otimes_{\Z}\mathbb Q
 \xrightarrow{\sim}
 K(\mathbb S')\otimes_{\Z}\mathbb Q
\]
intertwines the actions of the two Serre functors, then there is an
orientation-preserving diffeomorphism of surfaces with stops
\[
 (S,M)\simeq(S',M').
\]
In particular, the conjugacy class over \(\mathbb Q\) of the Serre action determines the
genus, the number of boundary components, and the unordered collection of
numbers of stops on those components, in every genus.
\end{theorem}

\subsubsection{Recovering the number of boundary components}

\begin{lemma}\label{lem:serre-recovers-boundary-number}
Let \(\mathbb S=(S,M,\eta)\) satisfy the assumptions of
Theorem~\ref{thm:serre-reconstruction}, put
\(K=K(\mathbb S)\), and set
\[
 \mathsf C=(\mathsf S[-1])_*\in\operatorname{Aut}(K).
\]
The automorphism \(\mathsf C\) has a unipotent positive power.
Let \(n_{\mathsf C}\) be the least positive integer such that
\(\mathsf C^{n_{\mathsf C}}\) is unipotent over \(\mathbb Q\).
Then the number \(b\) of boundary components is
\[
 b=1+\rk_{\mathbb Q}
 \bigl(\mathsf C^{n_{\mathsf C}}-I\bigr).
\]
\end{lemma}

\begin{proof}
Write
\[
 \partial S=B_1\sqcup\cdots\sqcup B_b,
 \qquad m_i=|M\cap B_i|.
\]
Since \(\eta=\mathbb R\xi\),
Theorem~\ref{thm:charge-lattice} gives
\[
 K\simeq H_1(S,M^\circ;\Z).
\]
Let \(\mathsf S\) be the Serre functor and set
\[
 \mathsf C=(\mathsf S[-1])_*=-\mathsf S_*
 \in\operatorname{Aut}(K).
\]
By \eqref{eq:serre-rotation-action}, there is a sign
\(\varepsilon\in\{1,-1\}\) such that
\begin{equation}\label{eq:reduced-serre-versus-rotation}
 \mathsf C=\varepsilon R_*^{-1},
\end{equation}
where \(R\) is simultaneous positive one-step boundary rotation.
We first show that \(\mathsf C\) has a unipotent power. Set
\(L=\operatorname{lcm}(2,m_1,\ldots,m_b)\). Then
\(\varepsilon^L=1\), and \(\mathsf C^L\)
is induced by a product of powers of boundary twists. For
each boundary twist, its difference from the identity has
image in the span of the absolute boundary classes, on which
all boundary twists act trivially. Consequently
\((\mathsf C^L-I)^2=0\), so \(\mathsf C^L\) is unipotent.
We may therefore define
\begin{equation*}
 n_{\mathsf C}=
 \min\{n\geq1:\mathsf C^n
 \text{ is unipotent over }\mathbb Q\}.
\end{equation*}
We claim that
\begin{equation}\label{eq:boundary-rank-from-serre}
 b=1+\rk_{\mathbb Q}
 \bigl(\mathsf C^{n_{\mathsf C}}-I\bigr).
\end{equation}

Suppose first that \(S\) is not a disc.  Then
\(H_1(S;\mathbb Q)\neq0\), and \(R\) acts trivially on absolute homology.
Thus \(\mathsf C\) acts there by the scalar \(\varepsilon\), and
unipotence of \(\mathsf C^{n_{\mathsf C}}\) implies
\(\varepsilon^{n_{\mathsf C}}=1\). On
\(\widetilde H_0(M^\circ;\mathbb Q)\), the rotation is a product of cycles of
lengths \(m_1,\ldots,m_b\).  Every non-trivial eigenvalue of each cycle
survives in reduced homology, so \(m_i\mid n_{\mathsf C}\) for every \(i\). It
follows that
\begin{equation}\label{eq:power-as-boundary-twists}
 \mathsf C^{n_{\mathsf C}}=R_*^{-n_{\mathsf C}}
 =\prod_{i=1}^b(T_i)_*^{-n_{\mathsf C}/m_i},
\end{equation}
where \(T_i\) is the positive boundary twist about \(B_i\).

Orient each \(B_i\) by the boundary orientation and put
\(c_i=[B_i]\in H_1(S;\mathbb Q)\).  These classes span a vector space of
dimension \(b-1\), with the single relation
\[
 c_1+\cdots+c_b=0.
\]
For \(x\in H_1(S,M^\circ;\mathbb Q)\), write
\[
 \partial x=\sum_{p\in M^\circ}a_p[p],
 \qquad
 \lambda_i(x)=\sum_{p\in M^\circ\cap B_i}a_p.
\]
Then \(\sum_i\lambda_i(x)=0\).  Put
\[
 W_\partial=
 \{(u_1,\ldots,u_b)\in\mathbb Q^b:\textstyle\sum_i u_i=0\}.
\]
Moreover, the map
\begin{equation}\label{eq:boundary-charge-map}
 H_1(S,M^\circ;\mathbb Q)\longrightarrow W_\partial,
 \qquad x\longmapsto(\lambda_1(x),\ldots,\lambda_b(x))
\end{equation}
is surjective since relative arcs between points of \(M^\circ\) on distinct boundary
components realize the vectors \(e_i-e_j\).

With the boundary orientation fixed above, a collar representative
of \(T_i\) adds \(c_i\) to an oriented relative arc ending on \(B_i\)
and \(-c_i\) to one starting there. Hence
\begin{equation}\label{eq:boundary-twist-relative-homology}
 (T_i)_*x=x+\lambda_i(x)c_i.
\end{equation}
Combining \eqref{eq:power-as-boundary-twists} and
\eqref{eq:boundary-twist-relative-homology}, we obtain
\[
 \bigl(\mathsf C^{n_{\mathsf C}}-I\bigr)x
 =-\sum_{i=1}^b
 \frac{n_{\mathsf C}}{m_i}\lambda_i(x)c_i.
\]
Identify the span of the \(c_i\) with
\(\mathbb Q^b/\mathbb Q(1,\ldots,1)\). Since multiplication by
\(-1\) does not change the rank, it suffices to consider the map
\begin{equation*}
 W_\partial\longrightarrow\mathbb Q^b/\mathbb Q(1,\ldots,1),
 \qquad
 (u_i)\longmapsto
 \left(\frac{n_{\mathsf C}}{m_i}u_i\right).
\end{equation*}
This map is injective.  Indeed, if its value is zero, then
\((n_{\mathsf C}/m_i)u_i=c\) for one constant \(c\). Hence
\[
 0=\sum_i u_i=\frac{c}{n_{\mathsf C}}\sum_i m_i,
\]
which forces \(c=0\), and then \(u_i=0\) for every \(i\).  By the
surjectivity of \eqref{eq:boundary-charge-map}, the image of
\(\mathsf C^{n_{\mathsf C}}-I\) therefore has dimension \(b-1\),
proving \eqref{eq:boundary-rank-from-serre} when \(S\) is not a disc.

If \(S\) is a disc, then \(b=1\) and
\(H_1(S,M^\circ;\mathbb Q)=\widetilde H_0(M^\circ;\mathbb Q)\). The
operator \(\mathsf C\) is semisimple of finite order. Its least unipotent
power is therefore the identity, so
\(\rk(\mathsf C^{n_{\mathsf C}}-I)=0=b-1\). Thus
\eqref{eq:boundary-rank-from-serre} holds in every case.
\end{proof}

\subsubsection{Recovering the stop distribution}

\begin{lemma}\label{lem:serre-recovers-stops}
Under the assumptions of
Lemma~\ref{lem:serre-recovers-boundary-number}, the rational conjugacy class
of \(\mathsf C=(\mathsf S[-1])_*\) determines the unordered stop distribution
\(\{m_1,\ldots,m_b\}\).
\end{lemma}

\begin{proof}
Write
\[
 \partial S=B_1\sqcup\cdots\sqcup B_b,
 \qquad m_i=|M\cap B_i|,
\]
and put \(K=K(\mathbb S)\). By
Theorem~\ref{thm:charge-lattice}, we identify
\(K\simeq H_1(S,M^\circ;\Z)\). By
\eqref{eq:serre-rotation-action}, there is a sign
\(\varepsilon\in\{1,-1\}\) such that
\(\mathsf C=\varepsilon R_*^{-1}\), where \(R\) is simultaneous
positive one-step boundary rotation. Consider the exact sequence
\begin{equation}\label{eq:relative-homology-sequence}
 0\longrightarrow H_1(S;\Z)
 \longrightarrow H_1(S,M^\circ;\Z)
 \longrightarrow\widetilde H_0(M^\circ;\Z)
 \longrightarrow0,
\end{equation}
which is equivariant for \(R\). The rotation acts trivially on
\(H_1(S;\Z)\). On \(H_0(M^\circ;\Z)\), it is a product of disjoint
cycles of lengths \(m_1,\ldots,m_b\). Since
\(\rk H_1(S;\Z)=2g+b-1\), multiplicativity of characteristic
polynomials in \eqref{eq:relative-homology-sequence} gives, when
\(S\) is not a disc,
\begin{equation}\label{eq:rotation-polynomial}
 \det(tI-R_*)=(t-1)^{2g+b-2}\prod_{i=1}^b(t^{m_i}-1).
\end{equation}
For a disc, the right-hand side is understood as
\((t^{m_1}-1)/(t-1)\). The inverse rotation has the same cycle
lengths and is also the identity on absolute homology, so
\(R_*^{-1}\) has the same characteristic polynomial as \(R_*\).

For ordinary gentle algebras of finite global dimension, the related
Coxeter-polynomial formula in terms of boundary winding numbers is
given in \cite[Proposition~3.2]{ChangElaginSchroll}, following
\cite[Theorem~6.17]{JimenezGonzalezMroz}.

It remains to remove the sign \(\varepsilon\) in
\eqref{eq:reduced-serre-versus-rotation}. All eigenvalue multiplicities
in this proof are algebraic multiplicities.
For \(R_*^{-1}\), the multiplicities of the eigenvalues
\(1\) and \(-1\) are
\begin{equation}\label{eq:sign-multiplicities}
 p_+=2g+2b-2,
 \qquad
 p_- =\#\{i:2\mid m_i\}.
\end{equation}
Multiplying the operator by \(-1\) interchanges these multiplicities.
If \(g\geq1\), then \(p_+\geq2b>p_-\). If \(g=0\) and \(b\geq3\),
then \(p_+=2b-2>b\geq p_-\). Among \(\mathsf C\) and
\(-\mathsf C\), the condition that the multiplicity of \(1\) is
larger therefore selects \(R_*^{-1}\) uniquely in these cases.

The remaining low-complexity possibilities are read from the same linear
data.  Once \(b=1\) is known, a disc is characterized by
\[
 \max\{\operatorname{mult}_{\mathsf C}(1),
        \operatorname{mult}_{\mathsf C}(-1)\}\leq1.
\]
Indeed, a positive-genus surface with one boundary component has
\(p_+=2g\geq2\), whereas a disc has at most one of the eigenvalues
\(\pm1\) on reduced zero-homology.  For a disc the stop number is recovered
directly as
\(
 m_1=1+\rk K\). If \(b=2\), equality of the two multiplicities in
\eqref{eq:sign-multiplicities} can occur only for an annulus whose two
stop numbers are even.  In that case,
 $(t^{m_1}-1)(t^{m_2}-1)$ is invariant under \(t\mapsto-t\), so either sign gives the same
characteristic polynomial.  Thus, outside the already reconstructed disc,
the conjugacy class of \(\mathsf C\) determines the polynomial on the right of
\eqref{eq:rotation-polynomial}.

For \(d\geq2\), let \(e_d\) be the multiplicity of the cyclotomic
polynomial \(\Phi_d(t)\) in \eqref{eq:rotation-polynomial}, and set
\[
 a_n=\#\{i:m_i=n\}.
\]
Since \(t^n-1=\prod_{d\mid n}\Phi_d(t)\), we have
\begin{equation}\label{eq:cyclotomic-counts}
 e_d=\sum_{d\mid n}a_n
 \qquad\text{ for }d\geq2.
\end{equation}
Let \(D=\max\{d\geq2:e_d\neq0\}\), with \(D=1\) if the set is empty.
Descending from \(D\) to \(2\), equation
\eqref{eq:cyclotomic-counts} recovers
\begin{equation*}
 a_n=e_n-\sum_{\substack{k\geq2\\kn\leq D}}a_{kn}
 \qquad\text{ for }n\geq2,
\end{equation*}
and then
\[
 a_1=b-\sum_{n=2}^D a_n.
\]
Hence the unordered stop distribution \(\{m_1,\ldots,m_b\}\) is determined.
In particular, the total number of stops
\(|M|=\sum_n n a_n\) is determined.
\end{proof}

\begin{remark}\label{rem:rotation-polynomial-insufficient}
The nilpotent information is essential. Put one stop on each
boundary component of \(\Sigma_{1,2}\) and of \(\Sigma_{0,3}\).
For both surfaces, relative homology has rank four and the
boundary-rotation polynomial is \((t-1)^4\). However,
the ranks of \(R_*-I\) are one and two, respectively.
Thus the polynomial alone does not determine the number of
boundary components, whereas the rank used in
Lemma~\ref{lem:serre-recovers-boundary-number} distinguishes
these examples.
\end{remark}

\begin{proof}[Proof of Theorem~\ref{thm:serre-reconstruction}]
Let \(K=K(\mathbb S)\) and
\(\mathsf C=(\mathsf S[-1])_*\). Lemmas~\ref{lem:serre-recovers-boundary-number} and
\ref{lem:serre-recovers-stops} then recover \(b\), the unordered stop
distribution, and hence \(|M|\), from the rational conjugacy class of
\(\mathsf C\). The ranks in \eqref{eq:relative-homology-sequence} give
\begin{equation*}
 \rk K=2g+b+|M|-2,
 \qquad
 g=\frac{\rk K-b-|M|+2}{2}.
\end{equation*}
We have reconstructed the genus, the number of boundary components, and
the number of stops on every boundary component, up to permutation, which yields an orientation-preserving diffeomorphism
\((S,M)\simeq(S',M')\).  Every step used only the rational conjugacy class
of the Serre action, so an intertwining \(\mathbb Q\)-linear isomorphism gives the same
data on the two sides. 
\end{proof}

\subsection{The full derived Picard representation}

We now retain the action of the whole derived Picard group on the
integral Grothendieck group, but do not distinguish the Serre element.
For a graded surface with stops \(\mathbb S\), let
\[
 \varrho_{\mathbb S}\colon D(\mathbb S)\longrightarrow
 \operatorname{Aut}\bigl(K(\mathbb S)\bigr)
\]
be the natural representation and put
\[
 \mathfrak R(\mathbb S)=
 \bigl(D(\mathbb S),[1],K(\mathbb S),\varrho_{\mathbb S}\bigr).
\]
An isomorphism from \(\mathfrak R(\mathbb S)\) to
\(\mathfrak R(\mathbb S')\) is a pair \((\Phi,\Theta)\). Here
\(\Phi\) is a group isomorphism preserving the shift, while
\(\Theta\) is an integral isomorphism satisfying
\[
 \Theta\varrho_{\mathbb S}(x)
 =\varrho_{\mathbb S'}\bigl(\Phi(x)\bigr)\Theta
 \qquad\text{for every }x\in D(\mathbb S).
\]

\Needspace{14\baselineskip}
\begin{theorem}
\label{thm:general-boundary-reconstruction}
Let \(\mathbb S=(S,M,\eta)\) be a compact connected graded surface
of genus \(g\geq2\), with non-empty boundary and at least one stop on
each boundary component. Write \(b=|\pi_0(\partial S)|\), and suppose
that \(\W(\mathbb S)\) is smooth and proper. Let \(t\) be the number of
boundary components with odd winding. Then
\(\mathfrak R(\mathbb S)\) determines \((S,\eta)\) and \(|M|\). On
isomorphism classes of graded surfaces with stops subject to these
hypotheses, its fibre satisfies
\[
 \#\mathfrak R^{-1}([\mathfrak R(\mathbb S)])\leq
 \begin{cases}
 4,&b=4,\ t=4,\\
 2,&b=2,\ \text{or }(b=3,\ t=2),\\
 1,&\text{otherwise}.
 \end{cases}
\]
In particular, \(\mathfrak R\) determines the complete graded surface
with stops when \(b=1\) or \(b\geq5\).
\end{theorem}

The boundary-rotation lattice is the common thread through the
reconstruction results. Its central extension records the stops in
Section~\ref{sec:ungraded}, while changes of basis in the same lattice
produce the fibres in Section~\ref{sec:fibres}. The Serre representation
detects its diagonal element, whereas the full representation considered here
retains the action of every rotation on relative homology. Separating
twists provide a second link to
Section~\ref{sec:derived-reconstruction}. There, separating twists are
used to obtain coarse geometricity, while here their
Picard--Lefschetz actions identify the absolute homological
intertwiner.

The proof has three steps. First, we recover the saturated absolute
charge lattice and the unordered endpoint data from the
boundary-rotation lattice. Second, we use the line-field stabilizer to
attach these data to the geometric boundary components. Finally, we
use the integral gluing of the negative endpoint blocks to control the
remaining permutations. The local collar calculations and the required
integral lattice arguments are proved in
Sections~\ref{appsubsec:collar-twisted-homology}
and~\ref{appsubsec:scalar-integral-matching}.

For this section and Appendix~\ref{app:twisted-homology}, fix
\(\mathbb S=(S,M,\eta)\) satisfying the hypotheses of the theorem and
write \(\partial S=B_1\sqcup\cdots\sqcup B_b\). Set
\[
 L_{\Z}=\Z_\eta,\qquad L=L_{\Z}\otimes_{\Z}\mathbb Q,
 \qquad
 K=H_1(S,M^\circ;L_{\Z}),
\]
and
\[
 \nu_i=\operatorname{hol}_L(B_i)=(-1)^{\omega_\eta(B_i)}.
\]
We call \(B_i\) positive when \(\nu_i=1\) and negative when
\(\nu_i=-1\). Thus the negative boundaries are precisely those with
odd winding. The boundary relation implies that their number \(t\) is
even. In addition, we also set
\[
 A_{\Z}=H_1(S;L_{\Z}),\qquad
 A_{\mathbb Q}=A_{\Z}\otimes\mathbb Q,
 \qquad K_{\mathbb Q}=K\otimes\mathbb Q.
\]
For \(b\geq2\), Corollary~\ref{cor:geometric-reduction} identifies the
intrinsic boundary-rotation lattice \(C\cong\Z^b\). We write \(r_i\)
for its primitive direction at \(B_i\), and put
\(C_{\mathbb Q}=C\otimes\mathbb Q\).

\subsubsection{Intrinsic boundary data}

\begin{proposition}\label{prop:represented-boundary-data}
Let \(g\geq2\) and \(b\geq2\). The represented invariant determines a
normalized action \(R:C\to\operatorname{Aut}(K)\), the absolute rational homology
space and its saturated integral lattice
\[
 A_{\mathbb Q}=K_{\mathbb Q}^{C}=H_1(S;L),
 \qquad
 A_{\Z}=K\cap A_{\mathbb Q}=H_1(S;L_{\Z}),
\]
and the unordered collection
\[
 \{(m_i,\nu_i):1\leq i\leq b\},
 \qquad m_i=|M\cap B_i|.
\]
If \(L\) is non-trivial, it also determines the rational endpoint
blocks, up to permutation, and the boundary-pointwise stabilizer
\[
 F=\Stab_{\Mod(S,\partial S)}(\eta).
\]
For a specified primitive boundary-rotation direction \(r_i\), its
normalized action \(R_i\) satisfies
\[
 m_i=\rk_{\mathbb Q}(R_i-I).
\]
\end{proposition}

\begin{proof}
By Theorem~\ref{thm:derived-reconstruction}, the represented invariant determines whether \(L\)
is trivial, and by Corollary~\ref{cor:geometric-reduction}, it determines
\(C\). By Lemma~\ref{lem:geometrization-charge-action}, its lifts act through
twisted relative homology and differ only by an overall sign. 

Suppose first that \(L\)
is trivial. Then the ordinary action of \(x\in C\) satisfies
\[
 \operatorname{tr}R(x)=2g+b-2+
 \#\operatorname{Fix}(x|_{M^\circ})>0.
\]
Positive trace therefore selects the ordinary action, and the selected
actions are multiplicative.

In this case, let \(x\in K_{\mathbb Q}\) be fixed by every
boundary rotation. On each boundary component, the endpoint
coefficients of \(x\) are constant. Invariance under the full twist
\(T_i\), together with \eqref{eq:boundary-twist-relative-homology},
gives \(\lambda_i(x)c_i=0\). Since \(b\geq2\), every boundary
class \(c_i\) is non-zero, so \(\lambda_i(x)=0\). Over
\(\mathbb Q\), the constant endpoint coefficients must therefore
all vanish. The relative homology sequence now gives
\(K_{\mathbb Q}^{C}=H_1(S;\mathbb Q)\). Conversely, boundary
rotations act trivially on absolute homology. The same exact sequence
over \(\Z\) shows that \(H_1(S;\Z)\) is saturated in \(K\).

We next construct the normalized action when \(L\) is non-trivial.
Let \(\widetilde C\) be the
full inverse image of \(C\) under \(D(\mathbb S)\to H_\eta\).
Its charge image is abelian: the collar-normalized rotations commute,
and all other lifts differ in their charge action only by an overall
sign. In the complex semisimplification of this representation of
\(\widetilde C\), let \(\chi_0\) denote the character of the action over
the core. Its multiplicity is \(\dim A_{\mathbb Q}+p\), where \(p\) is the number
of positive boundaries. Every other joint character occurs once, and its
ratio with \(\chi_0\) is supported on a single boundary coordinate. Thus
\(\chi_0\) is the unique joint character of multiplicity greater than
one. It takes values in \(\{1,-1\}\), is trivial on the geometrisation
kernel, and sends the shift to \(-1\).

For \(c\in C\), choose a lift
\(\widetilde c\in\widetilde C\) and set
\[
 R(c)=\chi_0(\widetilde c)^{-1}
       \varrho_{\mathbb S}(\widetilde c).
\]
This is independent of the lift and defines an intrinsic integral action
of \(C\) which is the identity on absolute homology. These assertions
follow from the signed-cycle calculation in
Lemma~\ref{lem:normalized-twisted-collar} and the normal forms in
Lemma~\ref{lem:twisted-homology-normal-forms}. The fixed space of the normalized action is \(A_{\mathbb Q}\). Indeed, a
negative endpoint block has no fixed vector. On a positive block, a
fixed endpoint vector has constant coefficients, and the full twist
forces its total to vanish. The exact sequence in
Lemma~\ref{lem:geometrization-charge-action} then identifies the
integral intersection with \(A_{\Z}\) and proves saturation.

We next record the endpoint data. For
\(x=\sum_i x_ir_i\in C_{\mathbb Q}\), choose \(d>0\) such that
\(dx_i\in m_i\Z\) on positive boundaries and
\(dx_i\in2m_i\Z\) on negative boundaries. For a positive boundary
\(B_i\), choose a local generator and put \(c_i=[B_i]_L\). Let
\(\lambda_i\) be the total endpoint coefficient on \(B_i\), and let
\(\varepsilon_\partial\in\{1,-1\}\) be the common collar sign. Then
\[
 \mathcal L(x)=d^{-1}\bigl(R(dx)-I\bigr)
 =\varepsilon_\partial
   \sum_{\nu_i=1}\frac{x_i}{m_i}c_i\otimes\lambda_i.
\]
The formula is independent of the sufficiently divisible choice of
\(d\) and defines a \(\mathbb Q\)-linear map on \(C_{\mathbb Q}\).
The positive boundary classes \(c_i\) are independent by
Lemma~\ref{lem:twisted-homology-normal-forms}. The endpoint totals
\(\lambda_i\) are independent because \(H_0(S;L)=0\), so the endpoint
map onto \(H_0(M^\circ;L)\) is surjective. Factoring through their coordinate
spaces therefore gives
\[
 \operatorname{rk}\mathcal L\!\left(\sum_i x_i r_i\right)
 =\#\{i:\nu_i=1,\ x_i\neq0\}.
\]
Thus \(C_-:=\ker\mathcal L\) is the rational span of the negative
axes, and \(C\cap C_-\) is their primitive integral span. The
non-zero rank-one locus of \(\mathcal L\) on
\(C_{\mathbb Q}/C_-\) is precisely the union of the punctured
positive quotient axes.

Let \(\mathcal S\) be the set of non-trivial joint characters of \(R\).
On a positive block with \(m_i>1\), they form the non-trivial part of
a cyclic group \(X_i\) of order \(m_i\). On a negative block, they are the odd powers
of a character generating a group \(X_i\) of order \(2m_i\). The groups
\(X_i\) are exactly the maximal members of
\[
 \{\langle\chi\rangle:\chi\in\mathcal S\}.
\]
Indeed, a generator of each \(X_i\) occurs in \(\mathcal S\), every
occurring character is supported on a single boundary coordinate, and
the groups attached to different coordinates have trivial intersection.
Such a group comes from a negative boundary exactly when its restriction
to \(C\cap C_-\) is non-trivial. Its order therefore recovers \(2m_i\)
in the negative case and \(m_i\) in the positive case. The number of
boundaries supplies the positive one-point blocks which do not occur in
\(\mathcal S\). This recovers the pairs \((m_i,\nu_i)\). When \(L\) is
trivial, the same maximal-character argument recovers all \(m_i>1\),
and \(b\) supplies the remaining entries equal to one.

For non-trivial \(L\), put \(E=K_{\mathbb Q}/A_{\mathbb Q}\).
The action on \(E\) has finite image and is semisimple over
\(\mathbb Q\). A negative endpoint block is the sum of the character
spaces belonging to the odd powers in \(X_i\). This sum is
Galois-stable and hence is defined over \(\mathbb Q\).

Each positive quotient axis gives a rank-one component of
\(\mathcal L\), whose kernel determines the corresponding endpoint
total \(\lambda_i\) up to a non-zero scalar. The map
\[
 E^C\longrightarrow\mathbb Q^p,
 \qquad v\longmapsto(\lambda_i(v))_{\nu_i=1},
\]
is an isomorphism: the fixed endpoint vectors have constant
coefficients on each positive boundary, and their totals may be
prescribed independently. The constant line associated with the
\(i\)-th positive quotient axis is therefore cut out by
\(\lambda_j=0\) for all other positive axes.

Each positive character group \(X_i\) is trivial on \(C\cap C_-\)
and hence descends to \(C/(C\cap C_-)\). Its restriction is
non-trivial on the primitive integral lattice of exactly one positive
quotient axis. This attaches its non-trivial character spaces to the
constant line associated with that axis. A positive axis on which no
such character occurs gives a one-point block. Thus every endpoint
block is recovered as a rational subspace, up to permutation.

Finally, the subgroup preserving every boundary component is
the centralizer of \(C\) in \(H_\eta\). An element of this subgroup lies
in \(F\) exactly when its projective action on every endpoint block is
trivial. Indeed, on the endpoint block at \(B_i\), the quotient action
is a scalar times a cyclic or signed cyclic permutation of the stops.
It is scalar precisely when the rotation exponent is divisible by
\(m_i\). This is the condition of fixing every stop, equivalently of
admitting a boundary-pointwise representative. It characterizes
\(F\) intrinsically, including the case \(m_i=1\). The last formula in
the statement is the rank calculation in
Lemma~\ref{lem:normalized-twisted-collar}.
\end{proof}

\begin{remark}
The last formula applies to a specified primitive rotation. It does not
identify the primitive axes from the rotation representation alone. On
a vector-framed surface with two boundary components and one stop on
each, the two twists act as \(I+U\), where \(U\neq0\) and \(U^2=0\).
Thus
\[
 (x,y)\longmapsto I+(x+y)U
\]
is preserved by changes of basis which do not permute the two axes.
The action of the line-field stabilizer and the integral charge
lattice impose further restrictions on this ambiguity. These are
precisely the restrictions used in
Theorem~\ref{thm:general-boundary-reconstruction}.
\end{remark}

\subsubsection{Matching the boundary blocks}
\begin{proposition}\label{prop:geometric-boundary-matching}
Assume that \(b\geq2\), and let \(\mathbb S'=(S',M',\eta')\) satisfy
the same hypotheses as \(\mathbb S\). Let
\((\Phi,\Theta):\mathfrak R(\mathbb S)\to\mathfrak R(\mathbb S')\) be
an isomorphism, and let \(f\) and \(\phi\) be supplied by
Corollary~\ref{cor:geometric-reduction}. After
choosing an integral identification of the orientation local systems,
\[
 \Theta|_{H_1(S;L_{\Z})}=\varepsilon f_*,
 \qquad \varepsilon\in\{1,-1\}.
\]
If \(L\) is trivial and \(b\geq3\), then \(f\) matches every boundary
component with a boundary having the same number of stops. If \(L\) is
non-trivial, the same conclusion holds for every positive boundary.
\end{proposition}

\begin{proof}
The scalar assertion follows from Lemma~\ref{lem:absolute-intertwiner-scalar}.
Suppose first that \(L\) is trivial. For a sufficiently divisible
positive integer \(d\), the normalized action of \(dx\), with
\(x\in C_{\mathbb Q}\), defines
\[
 \mathcal L(x)=d^{-1}(R(dx)-I).
\]
For a common collar sign
\(\varepsilon_\partial\in\{1,-1\}\), the collar formula gives
\[
 \mathcal L\!\left(\sum_i x_ir_i\right)
 =\varepsilon_\partial
  \sum_i\frac{x_i}{m_i}c_i\otimes\lambda_i,
\]
where \(c_i=[B_i]\) and \(\lambda_i\) is the endpoint total on
\(B_i\). The only relations in these two families are
\(\sum_i c_i=0\) and \(\sum_i\lambda_i=0\). If \(b\geq3\), these
relations show that
\begin{equation}\label{eq:boundary-axis-from-absolute-class}
 \{x\in C_{\mathbb Q}:\operatorname{im}\mathcal L(x)
       \subseteq\mathbb Qc_i\}=\mathbb Qr_i.
\end{equation}
To verify this identity, write
\(\mathcal L(x)=\sum_j a_jc_j\otimes\lambda_j\) and reduce modulo
\(\mathbb Qc_i\). The functionals \(\lambda_j\) with \(j\neq i\)
are independent coordinates, since the only relation among all
endpoint totals is their sum. Each \(c_j\), for \(j\neq i\),
remains non-zero modulo \(\mathbb Qc_i\) when \(b\geq3\). Thus
\(a_j=0\) for every \(j\neq i\) and the converse is immediate from the
collar formula.

The absolute intertwiner carries \(\mathbb Qc_i\) to
\(\mathbb Qf_*c_i\). Equation~\eqref{eq:boundary-axis-from-absolute-class}
and equivariance of \(\mathcal L\) therefore show that \(\phi\) carries
\(\mathbb Qr_i\) to the rotation axis at \(f(B_i)\).
Proposition~\ref{prop:represented-boundary-data} then recovers \(m_i\)
from the primitive integral generator.

Now suppose that \(L\) is non-trivial. For each positive boundary, put
\(c_i=[B_i]_L\), let \(\lambda_i\) be its endpoint-total covector, and
choose the common collar sign
\(\varepsilon_\partial\in\{1,-1\}\). The same logarithm is
\[
 \mathcal L\!\left(\sum_i x_ir_i\right)
 =\varepsilon_\partial
  \sum_{\nu_i=1}\frac{x_i}{m_i}c_i\otimes\lambda_i.
\]
The positive boundary classes and the endpoint totals are independent.
Put
\[
 C_-=\operatorname{span}_{\mathbb Q}\{r_i:\nu_i=-1\}.
\]
It follows that the kernel is \(C_-\) and
that the non-zero rank-one locus in \(C_{\mathbb Q}/C_-\) is the union of
the punctured positive axes. The scalar assertion matches each positive boundary
line with its geometric counterpart. Equivariance of \(\mathcal L\)
therefore matches the corresponding positive axis. The cyclic
characters in Proposition~\ref{prop:represented-boundary-data} recover
its stop number, including the case of one stop.

Once the relevant stop numbers have been matched, an orientation
correction from Lemma~\ref{lem:orientation-correction} may be applied
before the final boundary-collar isotopies. It preserves every boundary
component and the line-field homotopy class.
\end{proof}

Rational rotation data alone do not determine the negative boundary
matching. After labelling the target boundaries by \(f\), let \(\pi\)
be the permutation of the negative endpoint blocks. The integral
charge lattice carries pair labels which imply
\[
 [B_i]+[B_j]=[B_{\pi(i)}]+[B_{\pi(j)}]
 \quad\text{in }H_1(S;\mathbb F_2)
\]
for every pair of distinct negative boundaries. Together with the
single relation \(\sum_i[B_i]=0\), these identities fix all negative
blocks when there are at least four of them, except when all four
boundaries are negative. In that case a pair can be replaced by its
complement, and the permitted permutations form the Klein four group
of double transpositions. When there are exactly two negative
boundaries, a separate integral argument excludes their exchange if
at least two positive boundaries are present. These assertions are
proved in Lemma~\ref{lem:integral-negative-boundary-matching}, and the
second uses the integral gluing, not just the rational rotation
representation.

\subsubsection{Completion of the proof}

The preceding results contain all the geometric input. The proof now
reduces to a case distinction according to the orientation local system
and the number of negative boundary components.

\begin{proof}[Proof of Theorem~\ref{thm:general-boundary-reconstruction}]
The shift-pointed part of \(\mathfrak R\) and
Theorem~\ref{thm:derived-reconstruction} recover \((S,\eta)\), as well
as \(g\) and \(b\). The rank formula gives
\[
 |M|=\rk K(\mathbb S)-2g-b+2.
\]
For \(b=1\), this is the only stop number. For \(b=2\),
Proposition~\ref{prop:represented-boundary-data} recovers the unordered
pair of stop numbers, leaving at most two assignments.

Suppose \(b\geq3\). If \(L\) is trivial,
Proposition~\ref{prop:geometric-boundary-matching} matches every
boundary. If \(L\) is non-trivial, the same proposition matches all
positive boundaries. When \(t=0\), this completes the reconstruction.
When \(t=2\) and \(b\geq4\), there are at least two positive
boundaries, and Lemma~\ref{lem:integral-negative-boundary-matching} \textup{(ii)}
excludes the exchange of the two negative blocks. For \(b=3\), only
that exchange remains, giving at most two assignments.

If \(t\geq4\), the mod-two pair labels in
Lemma~\ref{lem:integral-negative-boundary-matching} \textup{(i)} fix every negative
block unless \(b=t=4\). In the latter case the remaining permutations
belong to the Klein four group and give at most four assignments.
Repeated stop counts and graded automorphisms can only reduce these
bounds. Whenever all counts match, an orientation correction followed
by collar isotopies identifies the graded surfaces with stops.
\end{proof}

\subsubsection{Exceptional fibres and low-genus limitations}

The following examples show that non-singleton fibres occur for
\(b=2,3,4\) in every genus \(g\geq2\). The integral calculations and the extension to the full derived Picard
group are given in
Lemmas~\ref{lem:normalized-representation-extension} and
\ref{lem:exceptional-representation-calculations} in
Section~\ref{appsubsec:exceptional-isomorphisms}.

\begin{example}\label{ex:two-boundary-represented-failure}
For every \(g\geq2\), let \(S=\Sigma_{g,2}\), and put \(m=2g+1\) and
\(q=4g+1\). Consider a line field \(\eta\) on \(S\) with boundary
windings \((1,-(4g+1))\). Such a line field can be constructed by
projectivizing a vector framing with boundary windings \((0,-2g)\) and
then adding an integral cohomology class with boundary evaluations
\((1,-1)\) and zero handle evaluations. Since the two boundary windings
are distinct, every element of the line-field stabilizer \(H_\eta\)
preserves each boundary component. Hence \(H_\eta=H_\eta^0\).

We now consider the two assignments of stop counts to the ordered boundary
components \((B_1,B_2)\): $(1,m)$ and $(m,1)$. The two resulting graded surfaces with stops have isomorphic represented
invariants, but their partially wrapped Fukaya categories are not
equivalent.
Let \(F\) be the boundary-pointwise stabilizer. By
Lemma~\ref{lem:relative-displacement}, relative displacement
gives a character
\[
 \beta:F\longrightarrow\Z,
 \qquad \beta(T_1)=1,\quad\beta(T_2)=q.
\]
The central element \(k=T_1^{-q}T_2\) lies in \(\ker\beta\), and
\(\alpha(y)=yk^{\beta(y)}\) is an automorphism of \(F\). By
Lemma~\ref{lem:stopped-stabilizer-root-presentation}, the stopped
stabilizers are obtained by adjoining \(\rho_2^m=T_2\) and
\((\rho'_1)^m=T_1\). The formula
\[
 \phi|_F=\alpha,
 \qquad
 \phi(\rho_2)=(\rho'_1)^{-q^2}T_2^2
\]
extends to an isomorphism and admits an integral intertwiner of the
charge representations by
Lemma~\ref{lem:exceptional-representation-calculations} \textup{(i)}. Hence the
two invariants \(\mathfrak R\) are isomorphic.

On the other hand, the boundary windings are distinct, so no graded
diffeomorphism can exchange the two boundary components.
Theorem~\ref{thm:complete-surface-invariant} therefore implies that the
two partially wrapped Fukaya categories are not equivalent. Their common
Grothendieck rank is \(4g+2\).
\end{example}

\begin{example}\label{ex:three-boundary-represented-failure}
Let \(S=\Sigma_{g,3}\), where \(g\geq2\), and let \(m>1\). There is a
line field \(\eta\) on \(S\) with boundary windings
\(
 (1,-1,-(4g+2))
\)
for which the ordered stop counts \((1,m,1)\) and \((m,1,1)\)
on \((B_1,B_2,B_3)\) give graded surfaces with stops having isomorphic
represented invariants and non-equivalent categories.

Begin with a vector-framed line field of boundary windings
\((0,0,-(4g+2))\), and add a class with boundary evaluations
\((1,-1,0)\) and zero handle evaluations. The three resulting boundary
windings are distinct, so \(H_\eta=H_\eta^0\). Let \(F\) be the
boundary-pointwise stabilizer, and set
\[
 L_\partial=
 \ker\bigl(H^1(S,\partial S;\Z)\longrightarrow H^1(S;\Z)\bigr).
\]
Let \(\beta_\eta:F\to L_\partial\) be the relative-displacement map of
Lemma~\ref{lem:relative-displacement}, and let \(e_1-e_2\) denote the
coordinate-difference homomorphism on \(L_\partial\). With the common
sign \(\varepsilon_0\) from that lemma, put
\[
 h=\varepsilon_0(e_1-e_2)\circ\beta_\eta:F\longrightarrow\Z.
\]
Then
\[
 h(T_1)=h(T_2)=1,
 \qquad h(T_3)=0.
\]
Set \(k=T_1^{-1}T_2\) and \(q=2m-1\). The element \(k\) is central and
lies in \(\ker h\). The map \(\alpha(y)=yk^{qh(y)}\) is an automorphism
of \(F\). The additional
roots are \(\rho_2^m=T_2\) and \((\rho'_1)^m=T_1\), and
\[
 \phi|_F=\alpha,
 \qquad
 \phi(\rho_2)=(\rho'_1)^{-q}T_2^2.
\]
Lemma~\ref{lem:exceptional-representation-calculations} \textup{(ii)} gives the
integral intertwiner and hence the isomorphism of \(\mathfrak R\).
The distinct windings exclude a graded boundary exchange, so the
categories are not equivalent by
Theorem~\ref{thm:complete-surface-invariant}. Their Grothendieck rank
is \(2g+m+3\). In particular, \(g=3\) and \(m=2\) give rank \(11\).
\end{example}

\begin{example}\label{ex:four-boundary-represented-failure}
Let \(S=\Sigma_{g,4}\), where \(g\geq2\), and choose \(m,n>1\) with
\(m+n=2g+2\). There is a line field \(\eta\) on \(S\) with boundary windings
\[
 (1,-(2m-1),-1,-(2n+1))
\]
for which \((1,m,1,n)\) and \((m,1,n,1)\) have isomorphic represented
invariants but non-equivalent categories.

The four winding numbers are pairwise distinct. Indeed, the only
possible remaining equality would force \(m=n+1\), making \(m+n\)
odd and contradicting \(m+n=2g+2\). Hence \(H_\eta=H_\eta^0\).

Start with a line field induced by a vector framing and having boundary
winding numbers \((0,-2m,0,-2n)\). Add a class with boundary evaluations
\((1,1,-1,-1)\) and zero handle evaluations. Put \(q_m=2m-1\) and
\(q_n=2n+1\). For the boundary-pointwise stabilizer \(F\),
represent the relative displacement of \(y\in F\) by
\[
 (e_1(y),e_2(y),e_3(y),e_4(y))
 \in\Z^4/\Z(1,1,1,1).
\]
Let \(\varepsilon_0\) be the common sign in
Lemma~\ref{lem:relative-displacement}, and set
\[
 h_{12}=\varepsilon_0(e_1-e_2),
 \qquad h_{34}=\varepsilon_0(e_3-e_4).
\]
These characters have values
\[
 \begin{array}{c|rrrr}
  &T_1&T_2&T_3&T_4\\\hline
  h_{12}&1&q_m&0&0\\
  h_{34}&0&0&-1&q_n.
 \end{array}
\]
Set
\[
 k_{12}=T_1^{-q_m}T_2,
 \qquad k_{34}=T_3^{-q_n}T_4^{-1},
\]
and
\[
 \alpha(y)=y k_{12}^{h_{12}(y)}k_{34}^{h_{34}(y)}.
\]
The elements \(k_{12}\) and \(k_{34}\) are central, and both characters
vanish on them. Hence \(\alpha\) is an automorphism of \(F\). Let
\(\rho_2,\rho_4\) be the source rotations and
\(\rho'_1,\rho'_3\) the target rotations. Their root relations are
\[
 \rho_2^m=T_2,\qquad \rho_4^n=T_4,
 \qquad (\rho'_1)^m=T_1,\qquad (\rho'_3)^n=T_3.
\]
The stopped stabilizers are related by
\[
 \phi|_F=\alpha,
 \qquad
 \phi(\rho_2)=(\rho'_1)^{-q_m^2}T_2^2,
 \qquad
 \phi(\rho_4)=(\rho'_3)^{-q_n^2}T_4^{-2}.
\]
Lemma~\ref{lem:exceptional-representation-calculations} \textup{(iii)} proves that
this isomorphism admits an integral charge intertwiner. Thus the two
represented invariants agree. The distinct winding numbers prevent a
graded boundary exchange. The categories are therefore not equivalent
by Theorem~\ref{thm:complete-surface-invariant} and their common
Grothendieck rank is \(4g+6\).
\end{example}

\begin{example}\label{ex:represented-low-genus-failure}
The genus hypothesis in
Theorem~\ref{thm:general-boundary-reconstruction} cannot be removed.
For genus zero, use the pair of pants in
Example~\ref{ex:low-genus-failure}. Its two line fields have even
boundary winding numbers and hence are induced by vector fields.
Normalize graded lifts at the same stop in both models. In the
trivializations supplied by these vector fields, the induced
local-system isomorphism has sign \(+1\) at that stop and therefore
everywhere. Thus corresponding boundary twists act identically on
\(H_1(\Sigma_{0,3},M^\circ;\Z)\). In the common presentation
\eqref{eq:genus-zero-common-group}, take the identity on the
local-system torus, the boundary-twist lattice, and the shift factor.
The torus acts trivially on the charge lattice and the shift acts by
\(-I\). This group identification, together with the identity on
ordinary relative homology, intertwines the full integral
representations.

For genus one, use the notation of the same example and compare
\(\eta_0+2\alpha\) with \(\eta_0+4\alpha\). Both line fields are induced
by vector fields, and their stabilizer is the same group \(H_\alpha\).
Normalize the graded lifts at the common unique stop. In the
trivializations supplied by the vector fields, their induced maps on
the orientation local systems have sign \(+1\) at this stop and
therefore everywhere. Thus the normalized lifts in both models act
on ordinary relative homology by the same underlying diffeomorphisms.
The shift acts by \(-I\) and the local-system torus acts trivially.
\enlargethispage*{2\baselineskip}
The common group identification consequently intertwines the two
integral representations. Thus their invariants \(\mathfrak R\) are
isomorphic. Nevertheless,
\[
 \widetilde{\mathcal A}(\eta_0+2\alpha)=2,
 \qquad
 \widetilde{\mathcal A}(\eta_0+4\alpha)=4.
\]
The line-field classification distinguishes both pairs of graded surfaces.
\end{example}

\clearpage
\appendix
\section{Mapping class groups and relative gradings}
\label{app:mapping-class-gradings}

This appendix establishes the mapping-class-group and grading input
for the reconstruction arguments.
Section~\ref{appsubsec:capping-rigidity} contains the group-theoretic
and cohomological calculations used in the capping argument of
Section~\ref{sec:ungraded} and in the virtual-centre computation of
Section~\ref{sec:derived-reconstruction}. Section~\ref{appsubsec:line-field-stabilizers}
contains the mapping-class-group calculations needed to
recognize line-field stabilizers and to treat the exceptional surface
\(\Sigma_{2,1}\). We retain the conventions of Section~\ref{sec:preliminaries} and introduce additional notation where it is first needed.

\subsection{Relative gradings and the charge lattice}
\label{appsubsec:relative-gradings-charge}
Fix a grading $\eta$ on $S$. Let \(T_i=T_{B_i}\) be the positive boundary twist and put
\[
 L_\partial=
 \ker\bigl(H^1(S,\partial S;\Z)\longrightarrow H^1(S;\Z)\bigr),
\]
and let \(u_i\in L_\partial\) be the image of the \(i\)-th standard
basis vector under the boundary connecting map
\(H^0(\partial S;\Z)\to H^1(S,\partial S;\Z)\).
Represent elements of \(\Stab_{\Mod(S,\partial S)}(\eta)\) by diffeomorphisms that are the identity on
boundary collars. For \(y\in \Stab_{\Mod(S,\partial S)}(\eta)\), set
\[
 \beta_\eta(y)=[y^*\eta]-[\eta]\in L_\partial,
\]
where the difference is taken relative to the boundary. For a graded
lift of \(y\), let
\(\delta_i(y)\in\{1,-1\}\) be its sign on the orientation local system
over \(B_i\). At the level of the charge lattice, the \(B_i\)-normalized
lift is the one for which \(\delta_i(y)=1\).

\begin{lemma}\label{lem:relative-displacement}
The map \(\beta_\eta\) is a homomorphism
\[
 \beta_\eta:\Stab_{\Mod(S,\partial S)}(\eta)\longrightarrow L_\partial
 \simeq\Z^b/\Z(1,\ldots,1).
\]
There is a sign \(\varepsilon_0\in\{1,-1\}\), independent of \(i\), such
that every vector framing \(\xi\), viewed relative to its boundary
restriction, satisfies
\begin{equation}\label{eq:relative-boundary-twist-framing}
 [T_i^*\xi]-[\xi]
 =\varepsilon_0\omega_\xi(B_i)u_i.
\end{equation}
Moreover,
\[
 \beta_\eta(T_i)=\varepsilon_0\omega_\eta(B_i)u_i.
\]
If \(\beta_\eta(y)\) is represented by \((d_1,\ldots,d_b)\), then
\[
 \delta_i(y)\delta_j(y)=(-1)^{d_i-d_j}.
\]
Thus the \(B_i\)- and \(B_j\)-normalized charge actions differ by
\((-1)^{d_i-d_j}\).
\end{lemma}

\begin{proof}
Relative homotopy classes of framings with prescribed boundary value
form a torsor over \(H^1(S,\partial S;\Z)\). Since a boundary twist is
isotopic to the identity after the boundary condition is forgotten,
the difference \([T_i^*\xi]-[\xi]\) lies in \(L_\partial\).

Choose basepoints on the boundary and transverse arcs from \(B_1\)
to the other boundary components, with fixed tangent directions at
their endpoints. Together with an absolute homology basis, these arcs
detect a relative framing \cite[Lemma~2.2]{CalderonSalter}. Changing
the endpoint conventions adds constants that cancel in the difference
before and after a boundary twist. The twist-linearity formula
\cite[Lemma~2.4]{CalderonSalter} gives
\[
 \omega_\xi(T_i(a))-\omega_\xi(a)
 =\langle a,B_i\rangle\omega_\xi(B_i)
\]
for every such oriented arc \(a\). The functional
\(a\mapsto\langle a,B_i\rangle\) is, up to the orientation convention
common to all boundary components, the evaluation of \(u_i\). This proves
\eqref{eq:relative-boundary-twist-framing}.

Relative homotopy classes of line fields with prescribed boundary
value form the analogous torsor. Since an element of \(\Stab_{\Mod(S,\partial S)}(\eta)\) preserves
the absolute homotopy class of \(\eta\), its relative displacement lies
in \(L_\partial\). The cohomology sequence of the pair identifies this
group with \(\Z^b/\Z(1,\ldots,1)\). The boundary-pointwise mapping class
group acts trivially on this quotient, so the displacement crossed
homomorphism is a homomorphism. The preceding arc calculation applies
to tangent lines and proves the stated twist formula.

At a fixed point of \(B_i\), an absolute homotopy from \(y^*\eta\) to
\(\eta\) is a loop in the projective tangent line. Differences of its
degrees on two boundary components represent the coordinate differences
of \(\beta_\eta(y)\). A loop of degree one in
\(\mathbb P^1(\mathbb R)\) reverses its oriented-line lift. This proves
the sign formula and the final assertion.
\end{proof}

For the next lemma, let \(A\) and its surface model
\(\mathbb S=(S,M,\eta)\) be as in
Theorem~\ref{thm:opper-geometrization}. Let \(M^\circ\) be the auxiliary
endpoint set fixed in Section~\ref{sec:preliminaries}.

\begin{lemma}\label{lem:geometrization-charge-action}
The geometrisation kernel \(N(\mathbb S)\) acts trivially on
\(K(\mathbb S)\). Under the identification of
Theorem~\ref{thm:charge-lattice}, a graded mapping class acts through
the induced map on twisted relative homology. Changing its graded lift
by \([n]\) multiplies this action by \((-1)^n\).

There is also a natural exact sequence
\[
 0\longrightarrow H_1(S;\Z_\eta)
 \longrightarrow K(\mathbb S)
 \longrightarrow
 \ker\bigl(H_0(M^\circ;\Z_\eta)\longrightarrow
                 H_0(S;\Z_\eta)\bigr)
 \longrightarrow0.
\]
In particular, \(H_1(S;\Z_\eta)\) is saturated in
\(K(\mathbb S)\).
\end{lemma}

\begin{proof}
Put \(N=N(\mathbb S)\) and \(K=K(\mathbb S)\). By
Theorem~\ref{thm:opper-geometrization}, we may write
\[
 N=U\rtimes T,\qquad
 U\simeq(\kk,+)^{\phi_A(1,1)},\qquad
 T=H^1(S;\kk^\times).
\]
Both \(U\) and \(T\) are divisible abelian groups. This uses
characteristic zero for \(U\), and algebraic closedness together with
the freeness of \(H_1(S;\Z)\) for \(T\). A homomorphism from a
divisible group to a finite group \(G\) is trivial, since every image
element is a \(|G|\)-th power in \(G\). Fix a prime \(\ell\).
For every \(n\geq1\), the induced homomorphism
\[
 N\longrightarrow\operatorname{Aut}(K/\ell^nK)
\]
has finite target. Its restrictions to \(U\) and \(T\) are therefore
trivial. These subgroups generate \(N\), so the entire homomorphism is
trivial. Thus, for \(z\in N\) and \(x\in K\),
\[
 (\varrho_{\mathbb S}(z)-I)x
 \in\bigcap_{n\geq1}\ell^nK=0,
\]
because \(K\) is free abelian of finite rank. This proves that \(N\)
acts trivially on \(K\).

The geometric section of the geometrisation homomorphism
\cite[Theorem~7.1]{Opper2025}, together with the naturality of
Theorem~\ref{thm:charge-lattice}, identifies the resulting quotient
action with the induced action on twisted relative homology. A
generator of the deck group represents the shift and acts by \(-I\)
on the Grothendieck group. Changing a graded lift by \([n]\) therefore
multiplies its action by \((-1)^n\).

The exact sequence in the statement is the homology sequence of the
pair \((S,M^\circ)\). Its quotient is a subgroup of the free group
\(H_0(M^\circ;\Z_\eta)\), and is therefore free. This also proves the
saturation assertion.
\end{proof}

\subsection{Capping and mapping class groups}
\label{appsubsec:capping-rigidity}
\begin{lemma}\label{lem:no-normal-abelian-subgroup}
Let \(X\) be a closed oriented surface of genus at least two and let
\(Q\neq\varnothing\) be finite.  A finite-index subgroup of
\(\PMod(X,Q)\) has no non-trivial normal abelian subgroup.  Moreover, the
centralizer in \(\Mod(X,Q)\) of a finite-index subgroup of
\(\PMod(X,Q)\) is trivial.
\end{lemma}

\begin{proof}
Let \(\Gamma\leq\PMod(X,Q)\) have finite index and let
\(A\triangleleft\Gamma\) be abelian. Choose a finite-index normal
subgroup \(\Gamma_0\triangleleft\Gamma\) contained in a torsion-free
pure finite-index subgroup of the mapping class group, and put
\(B=A\cap\Gamma_0\).

Suppose that \(B\neq1\). Its canonical reduction system is preserved
by \(\Gamma\), since \(B\) is normal in \(\Gamma\). This system
cannot be non-empty. Indeed, a power of a Dehn twist about a curve
intersecting it belongs to \(\Gamma\) and does not preserve it. A non-trivial pure
abelian subgroup with empty reduction system contains a pseudo-Anosov
element. In this case \(\Gamma\) preserves the pair of its invariant
measured foliations up to scale. The stabilizer of this pair is virtually
cyclic, which is impossible for a finite-index subgroup of
\(\PMod(X,Q)\). See \cite{BirmanLubotzkyMcCarthy,FarbMargalit}.
Thus \(B=1\), and \(A\) is finite.

The kernel \(\Gamma_1\) of the conjugation action of \(\Gamma\) on \(A\)
has finite index in \(\PMod(X,Q)\).  It contains a non-zero power of the
twist about every essential curve.  Every element of \(A\) centralizes
\(\Gamma_1\), and hence fixes every vertex of the curve complex.  The
action on that complex is faithful in the present punctured range
\cite{Ivanov1997,Luo}.  Thus \(A=1\), proving the first assertion.  The
same curve-complex argument applies directly to any element centralizing a
finite-index subgroup of \(\PMod(X,Q)\), and proves the second assertion.
\end{proof}

Let \(X\) and \(X'\) be closed oriented surfaces of genus \(g\geq2\).
Let their non-empty labelled point sets be
\[
 Q=\{q_1,\ldots,q_b\}\subset X,
 \qquad
 Q'=\{q'_1,\ldots,q'_b\}\subset X'
\]
and denote the corresponding cotangent-line classes by
\[
 \psi_1,\ldots,\psi_b\in H^2(\PMod(X,Q);\mathbb Q)
\]
and \(\psi'_1,\ldots,\psi'_b\), respectively. If an extended
diffeomorphism \(f:(X,Q)\to(X',Q')\) satisfies
\(f(q_i)=q'_{\pi(i)}\), set
\[
 c_f:\PMod(X,Q)\longrightarrow\PMod(X',Q'),
 \qquad c_f(h)=fhf^{-1}.
\]

\begin{lemma}\label{lem:psi-classes}
The classes \(\psi_1,\ldots,\psi_b\) are linearly independent. Moreover,
\[
 c_f^*\psi'_{\pi(i)}=\varepsilon\psi_i,
 \qquad \varepsilon\in\{1,-1\},
\]
where \(\varepsilon=1\) if \(f\) preserves orientation and
\(\varepsilon=-1\) otherwise.
\end{lemma}

\begin{proof}
Let \(\mathcal M_{g,b}\subset\overline{\mathcal M}_{g,b}\) be the moduli
stack of smooth genus-\(g\) curves with \(b\) ordered marked points and
its stable compactification. The degree-two localization sequence contains
the exact segment
\[
 \bigoplus_{\Delta\subset
  \overline{\mathcal M}_{g,b}\setminus\mathcal M_{g,b}}
 \mathbb Q[\Delta]
 \longrightarrow H^2(\overline{\mathcal M}_{g,b};\mathbb Q)
 \longrightarrow H^2(\mathcal M_{g,b};\mathbb Q).
\]
Thus the kernel of restriction from the compactification is spanned by
boundary-divisor classes.
Hence a relation among the \(\psi_i\) on \(\mathcal M_{g,b}\) would give
a relation among those classes and the boundary divisors on
\(\overline{\mathcal M}_{g,b}\).

For \(g>2\), \cite[Theorem~2.2]{ArbarelloCornalba} gives no such
relation, apart from the symmetry identifying two names for the same
boundary divisor. In genus two, the additional relation in
\cite[Theorem~2.2]{ArbarelloCornalba} has non-zero coefficient of the
first Mumford class. It therefore gives no non-zero
relation involving only the \(\psi_i\) and boundary-divisor classes.
This proves independence in every genus \(g\geq2\). Indeed, Teichm\"uller
space \(\mathcal T_{g,b}\) is contractible and
\(\mathcal M_{g,b}^{\mathrm{an}}=[\mathcal T_{g,b}/\PMod(X,Q)]\).
Thus the analytic stack is a classifying orbifold for
\(\PMod(X,Q)\).

The final assertion follows from the naturality of the cotangent-line
classes. If \(L_i\) and \(L'_{\pi(i)}\) denote the corresponding complex
cotangent lines, then an orientation-preserving diffeomorphism identifies
\(f^*L'_{\pi(i)}\) with \(L_i\). An orientation-reversing diffeomorphism
identifies it with \(\overline{L_i}\). Since
\(c_1(\overline{L_i})=-c_1(L_i)\), all cotangent-line classes acquire the
same sign.
\end{proof}

\begin{lemma}\label{lem:separating-fixer}
Let \(S=\Sigma_{g,b}\) with \(g\geq2\), and let \(M\) meet every boundary
component. If \(\mathcal K_{\mathrm{sep}}\leq\Mod(S,M)\) is the subgroup fixing the
isotopy class of every separating essential curve, then the
boundary-rotation lattice \(A_M\) has finite index in
\(\mathcal K_{\mathrm{sep}}\).
\end{lemma}

\begin{proof}
Use the capping map \(\kappa\) from
\eqref{eq:coloured-capping-map}.  Its kernel is \(A_M\), and
\(\kappa(\mathcal K_{\mathrm{sep}})\) fixes every vertex of the separating-curve
complex of the capped surface \((\widehat S,Q)\).  If \(g\geq3\), or if
\(g=2\) and \(b\geq2\), the natural extended mapping-class-group action
on this complex is faithful by \cite[Theorem~1.2]{Kida}. Thus the image is
trivial in these cases.

It remains to consider \((g,b)=(2,1)\). Choose a genus-one
separating curve \(c\) on the once-punctured genus-two
surface and a pseudo-Anosov mapping class \(f\). For
sufficiently large \(n\), the curves \(c\) and \(f^n(c)\)
fill the surface because their distance in the curve graph eventually
exceeds two. Both are genus-one separating curves.
By the Alexander method, the subgroup fixing the isotopy
class of every curve in this filling collection is finite (see \cite{FarbMargalit}).
Since
\(\kappa(\mathcal K_{\mathrm{sep}})\) fixes every separating curve, it is contained in
this finite stabilizer.  Hence
\(\mathcal K_{\mathrm{sep}}/A_M
\simeq\kappa(\mathcal K_{\mathrm{sep}})\) is finite in all cases.
\end{proof}

\subsection{Line-field stabilizers}
\label{appsubsec:line-field-stabilizers}

\begin{lemma}\label{lem:orientation-correction}
Let \(S\) be a compact connected oriented surface with non-empty boundary
and \(g(S)\geq2\). Every line field \(\eta\) on \(S\) is preserved up to homotopy
by an orientation-reversing diffeomorphism that preserves every boundary
component setwise.
\end{lemma}

\begin{proof}
By the classification of compact surfaces, there is an
orientation-reversing diffeomorphism \(r:S\to S\) which preserves every
boundary component setwise.  Put
\(\theta=r^*\eta\).  The orientation of \(S\) orients the fibres of
\(\mathbb P(TS)\).  Since \(r\) reverses this fibre orientation, for every
oriented immersed closed curve \(c\) one has
\begin{equation}\label{eq:orientation-reversal-winding}
 \omega_\theta(c)=-\omega_\eta(r(c)).
\end{equation}
If \(B\subset\partial S\) is given its boundary orientation, then
\(r(B)=-B\) as oriented curves.  Consequently
\[
 \omega_\theta(B)=\omega_\eta(B).
\]

We compare the remaining invariants in the classification of mapping-class
orbits of line fields.  Reduction of
\eqref{eq:orientation-reversal-winding} modulo \(2\) shows that the mod-two
winding class of \(\theta\) is obtained from that of \(\eta\) by the
isomorphism induced by \(r\).  It is therefore zero if and only if the
mod-two winding class of \(\eta\) is zero.  In the case in which the Arf
invariant is defined, let \(q_\eta\) and \(q_\theta\) be the associated
\(\Z/4\)-valued quadratic functions.  Formula~\eqref{eq:orientation-reversal-winding} gives
\[
 q_\theta(x)=-q_\eta(r_*x)
 \qquad \text{ for }x\in H_1(S;\Z/4).
\]
Both functions then take values in \(2\Z/4\), on which multiplication
by \(-1\) is the identity.  After division by \(2\), they descend to
quadratic forms on the homology of the capped surface.  The map \(r_*\)
preserves the mod-two intersection form, because the distinction between
symplectic and anti-symplectic disappears over \(\Z/2\).  The two
quadratic forms therefore have the same Arf invariant.

Thus \(\eta\) and \(\theta\) have the same boundary winding numbers, the same
mod-two invariant, and, when required, the same Arf invariant.  By
Theorem~\ref{thm:line-field-classification}, whose mapping class group fixes
the boundary pointwise, there is an orientation-preserving diffeomorphism
\(h:S\to S\), equal to the identity on \(\partial S\), such that
\(h^*\theta\simeq\eta\).  The composition
\(
 \rho=r\circ h
\)
is orientation reversing, preserves every boundary component setwise, and
satisfies
\[
 \rho^*\eta=h^*r^*\eta=h^*\theta\simeq\eta.
\]
\end{proof}

For a group \(E\) and an element \(x\in E\), write \(C_E(x)\)
for the centralizer of \(x\) and \(Z(C_E(x))\) for its centre.
Let \(S=\Sigma_{g,b}\), let \(\eta\) be a line field, and put
\[
 \mathcal I_0(S)
 =\ker\bigl(\PMod(S)\to\operatorname{Aut}H_1(S;\Z)\bigr),
 \qquad
 \mathcal C_0(S)
 =\{x\in\mathcal I_0(S):x^*\eta\simeq\eta\}.
\]
We write \(\mathcal I(S)\) and \(\mathcal K(S)\) for the Torelli group
and the Johnson kernel, as in \cite{Kida}. Here a bounding pair is
understood in the sense of \cite{Kida}: it consists of two disjoint,
non-isotopic, nonseparating curves whose union separates \(S\).
In contrast to the oriented bounding pairs used in
\eqref{eq:weighted-bounding-pair}, these curves need not be homologous
in \(H_1(S;\Z)\).

\begin{lemma}\label{lem:twists-in-chillingworth-kernel}
Assume that \(g\geq3\), or that \(g=2\) and \(b\geq2\), and put
\(r=2g+b-3\). For \(1\neq x\in\mathcal C_0(S)\), the following are
equivalent.
\begin{enumerate}[label=\textup{(\alph*)}]
\item \(Z(C_{\mathcal C_0(S)}(x))\simeq\Z\), and
\(x\in A\leq\mathcal C_0(S)\) for some \(A\simeq\Z^r\).
\item \(x\) is a non-zero power of a Dehn twist about a separating curve.
\end{enumerate}
Hence every isomorphism \(\mathcal C_0(S)\to\mathcal C_0(S')\) in this
range preserves the cyclic axes \(\langle T_c\rangle\) of separating
twists.
\end{lemma}

\begin{proof}
For any other line field \(\theta\) on \(S\), write
\(\theta-\eta=\alpha\in H^1(S;\Z)\). If
\(x\in\mathcal I_0(S)\), then
\[
 x^*\theta-\theta=(x^*\eta-\eta)+(x^*\alpha-\alpha)
 =x^*\eta-\eta.
\]
This proves the independence of \(\eta\). Separating twists act
trivially on absolute homology and preserve every line field, so
\(\mathcal K(S)\subseteq\mathcal C_0(S)\). To prove the other
inclusion, cap one boundary component and denote the resulting surface
by \(R\), with capping homomorphism
\(q:\PMod(S)\to\PMod(R)\). By the Birman exact sequence,
\(\ker q\) is generated by point-pushing maps. Pushing the capped
point around a handle generator gives a bounding-pair map, while
pushing it around another boundary component gives a product of
separating twists. Thus \(\ker q\subseteq\mathcal I(S)\);
see \cite[Section~4.2]{FarbMargalit} and the conventions in
\cite[Section~2]{Kida}. The separating twists and bounding-pair
maps generating \(\mathcal I(R)\) admit lifts of the same type,
and capping sends every generator of \(\mathcal I(S)\) into
\(\mathcal I(R)\). Hence
\[
 q(\mathcal I(S))=\mathcal I(R),
 \qquad
 \mathcal I(S)=q^{-1}(\mathcal I(R)).
\]
Iterating and using the homological description of the Torelli group
of the closed surface shows that \(\mathcal I_0(S)\subseteq
\mathcal I(S)\). Consequently,
\[
 \mathcal K(S)\subseteq\mathcal C_0(S)\subseteq\mathcal I(S).
\]

By \cite[Proposition~6.3]{Kida}, the maximal rank of an abelian subgroup of both
\(\mathcal K(S)\) and \(\mathcal I(S)\) is \(r\).  The inclusions above therefore show that the
same is true for \(\mathcal C_0(S)\).

Suppose first that \(x=T_c^n\), where \(c\) is separating
and \(n\neq0\). Apply \cite[Lemma~6.5]{Kida} with
\(\Gamma=\mathcal K(S)\). It follows that
\(
 Z(C_{\mathcal K(S)}(x))\simeq\Z
\)
and that \(x\) lies in a free abelian subgroup of
\(\mathcal K(S)\) of rank \(r\). The latter subgroup is
also contained in \(\mathcal C_0(S)\).

The groups \(\mathcal K(S)\) and \(\mathcal C_0(S)\)
have the same maximal abelian rank \(r\), and every abelian
subgroup of \(\mathcal C_0(S)\) is finitely generated
\cite{BirmanLubotzkyMcCarthy}.
The rank-\(r\) subgroup just constructed contains \(x\), so the
remaining hypothesis of \cite[Lemma~6.8]{Kida} is also satisfied.
Applying that lemma to
\(\mathcal K(S)\hookrightarrow\mathcal C_0(S)\), we obtain
\[
 \rk Z(C_{\mathcal C_0(S)}(x))
 \leq \rk Z(C_{\mathcal K(S)}(x))=1.
\]
The left-hand centre contains the infinite-order element \(x\).
It is finitely generated by \cite{BirmanLubotzkyMcCarthy} and
is torsion free by \cite[Theorem~A.1]{Kida}. Hence it is isomorphic
to \(\Z\). This proves
\textup{(b)}\(\Rightarrow\)\textup{(a)}.

Conversely, assume \textup{(a)}.  Apply \cite[Lemma~6.8]{Kida} to the
inclusion
\(
 \mathcal C_0(S)\hookrightarrow\mathcal I(S)\). The domain and codomain have the same finite maximal abelian rank \(r\),
and abelian subgroups of the codomain are finitely generated
\cite{BirmanLubotzkyMcCarthy}.
Condition~\textup{(a)} supplies a rank-\(r\) free abelian subgroup
of the domain containing \(x\), as required by the lemma. It gives
\[
 \rk Z(C_{\mathcal I(S)}(x))
 \leq \rk Z(C_{\mathcal C_0(S)}(x))=1.
\]
The left-hand centre contains \(x\), and \(\mathcal I(S)\) is torsion free
by \cite[Theorem~A.1]{Kida}.
Thus it is infinite cyclic.  The rank-\(r\) subgroup of
\(\mathcal C_0(S)\) containing \(x\) is also a rank-\(r\) subgroup of
\(\mathcal I(S)\). \cite[Lemma~6.6]{Kida} now implies that \(x\) is a non-zero power either
of a separating twist or of a bounding-pair twist.

The second possibility is excluded by the two defining conditions of
\(\mathcal C_0(S)\).  Suppose that
\(x=(T_aT_b^{-1})^n\in\mathcal C_0(S)\), where \(n\neq0\).  The homological
action of a bounding-pair map is integral and unipotent.  If its non-zero
power acts trivially, then the map itself acts trivially.  Thus
\(T_aT_b^{-1}\in\mathcal I_0(S)\), and \(a,b\) may be oriented so that
\([a]=[b]\).  The relation obtained
from a component of \(S\setminus(a\cup b)\) then shows that one of the two
components contains no component of \(\partial S\).  Indeed, among the
boundary classes, the only relation is the sum of all of them.  Choose that
component and denote its genus by \(h\).  Since \(a\) and \(b\) are not
isotopic, \(h\geq1\).  Formulae
\eqref{eq:twist-line-field} and
\eqref{eq:bounding-pair-winding-difference} give
\[
 \bigl((T_aT_b^{-1})^n\bigr)^*\eta-\eta
 =n\bigl(\omega_\eta(a)-\omega_\eta(b)\bigr)a^\vee
 =-4hn\,a^\vee\neq0
\]
for \(n\neq0\).  Hence no non-zero power of a bounding-pair twist belongs
to \(\mathcal C_0(S)\), and \textup{(a)} implies \textup{(b)}.
\end{proof}

For a line field \(\eta\) on \(S=\Sigma_{2,1}\), put
\[
 R_\eta=\operatorname{im}\bigl(
 \Stab_{\Mod(S)}(\eta)
 \longrightarrow\operatorname{Aut}H^1(S;\mathbb F_2)
 \bigr).
\]

\begin{lemma}\label{lem:S21-two-torsion}
\begin{equation}\label{eq:S21-image-orders}
 |R_\eta|=
 \begin{cases}
 48,&\sigma(\eta)=1,\\
 72,&\sigma(\eta)=0\text{ and }\operatorname{Arf}(\eta)=0,\\
 120,&\sigma(\eta)=0\text{ and }\operatorname{Arf}(\eta)=1.
 \end{cases}
\end{equation}
Consequently, \(|R_\eta|\) determines the mapping-class orbit of \(\eta\).
\end{lemma}

\begin{proof}
The Poincar\'e--Hopf formula gives \(\omega_\eta(\partial S)=-6\). The mod-two
intersection pairing identifies \(H_1(S;\mathbb F_2)\) equivariantly
with its dual. We may therefore compute the representation on
\(V=H_1(S;\mathbb F_2)\), whose intersection form is symplectic. Moreover,
\[
 |\operatorname{Sp}(V)|
 =|\operatorname{Sp}_4(\mathbb F_2)|=720.
\]
If \(\sigma(\eta)=1\), reduction of the winding function gives a non-zero
linear form on \(V\).  There are fifteen such forms, and
\(\operatorname{Sp}_4(\mathbb F_2)\) acts transitively on them.  The
stabilizer therefore has order \(720/15=48\).

If \(\sigma(\eta)=0\), the \(\Z/4\)-valued quadratic function of
the line field is even and, after division by two, gives a quadratic form
on \(V\).  There are ten even and six odd quadratic forms on a symplectic
four-space.  Their stabilizers therefore have orders \(720/10=72\) and
\(720/6=120\), respectively.

We must still prove that \(R_\eta\) is the full stabilizer of the indicated
datum.  The homological representation
\(\Mod(\Sigma_{2,1})\to\operatorname{Sp}_4(\Z)\) is surjective, and so is
its reduction modulo two \cite[Chapter~6]{FarbMargalit}.  Let
\(s\in\operatorname{Sp}_4(\mathbb F_2)\) stabilize the datum and choose a
mapping class \(f\) inducing \(s\).  Put \(\theta=f^*\eta\).
If \(\sigma(\eta)=1\), then the quadratic functions of \(\theta\) and
\(\eta\) have the same reduction modulo two. Their difference is an
even linear form, so there is \(\bar a\in V\) such that
\[
 q_\theta(x)-q_\eta(x)=2(\bar a\cdot\bar x)
 \qquad\text{for }x\in H_1(S;\Z/4),
\]
where \(\bar x\) denotes reduction modulo two. If \(\bar a=0\),
no correction is needed. Otherwise choose a nonseparating simple
curve \(a\) representing a primitive integral lift of \(\bar a\).
Formula~\textup{(1.7)} in the proof of
\cite[Theorem~1.8]{LekiliPolishchuk} gives
\[
 q_\eta(T_a^2x)-q_\eta(x)
 =2q_\eta([a])([a]\cdot x)\pmod4.
\]
If \(q_\eta([a])\) is odd, \(T_a^2\) corrects the difference.
If it is even, choose a nonseparating simple curve \(c\) with
\(q_\eta([c])\) odd. Then \(\bar a+\bar c\neq0\). Choose a
nonseparating simple curve \(d\) representing a primitive integral
lift of this class. Both \(q_\eta([c])\) and \(q_\eta([d])\) are
odd, and the same formula shows that \(T_d^2T_c^{-2}\) corrects
the difference. These corrections act trivially on
\(H_1(S;\mathbb F_2)\). If \(\sigma(\eta)=0\), stabilization of the
quadratic form \(q_\eta/2\) already gives \(q_\theta=q_\eta\), so no first
correction is needed.

After this correction, \cite[Lemma~1.7]{LekiliPolishchuk} leaves a
difference of the form \(4u\), with \(u\in H^1(S;\Z)\).  This
second correction may also be chosen in the level-two subgroup.  Indeed,
write the Poincar\'e dual of \(u\) as a sum of primitive nonseparating
classes. By the change-of-coordinates principle
\cite{FarbMargalit}, each such class is represented on
\(\Sigma_{2,1}\) by a curve \(a\) belonging to a genus-one bounding pair
\((a,b)\), and
\[
 (T_aT_b^{-1})^*\eta-\eta=-4a^\vee.
\]
Products of these bounding-pair maps realize every required class \(4u\),
and act trivially on integral homology.  We have therefore constructed a
mapping class \(h\), acting trivially on mod-two homology, such that
\(h^*f^*\eta\simeq\eta\). Thus
\(f\circ h\in\Stab_{\Mod(S)}(\eta)\) realizes \(s\).
The image \(R_\eta\) is the full stabilizer, and the three orders in
\eqref{eq:S21-image-orders} follow.

The orders are distinct.  Hence the conjugacy class of the representation
determines \(\sigma(\eta)\) and, when it is defined, the Arf invariant.
Together with the fixed boundary winding numbers, these are precisely the
invariants in Theorem~\ref{thm:line-field-classification}, and the final assertion
follows.
\end{proof}

\section{The three-boundary family and the Serre element}
\label{app:serre-family}

This appendix supplies the calculations used in
Section~\ref{sec:fibres}. Section~\ref{appsubsec:framed-stabilizer}
constructs the common stabilizer presentation.
Section~\ref{appsubsec:zero-winding-generators} gives the homological
generators needed to detect a rank-one local system.
Section~\ref{appsubsec:serre-normalization} then identifies the actual
Serre element, rather than only its image under geometrisation.

\subsection{The stabilizer presentation}
\label{appsubsec:framed-stabilizer}

Fix \(g\geq2\), put \(q=2g+1\), and let
\(S=\Sigma_{g,3}\) carry a vector framing \(\xi\) with boundary winding
numbers \((1,-1,-q)\).  Put \(\eta=\mathbb R\xi\), and let \(M_{r,s}\)
have \(r,s,1\) stops on the three boundary components in the
indicated order, where \(r,s\geq1\). Write
\[
 F=\Stab_{\Mod(S,\partial S)}(\eta),
 \qquad H_{r,s}=\Stab_{\Mod(S,M_{r,s})}(\eta).
\]
Let \(T_i=T_{B_i}\) be the positive boundary twist, and let
\(\rho_i\) be the positive one-step rotation on \(B_i\). Thus
\[
 \rho_1^r=T_1,
 \qquad \rho_2^s=T_2,
 \qquad \rho_3=T_3.
\]

\begin{lemma}\label{lem:framed-stabilizer-presentation}
There are a subgroup \(F_0\leq F\) and an element \(a\in Z(F_0)\) such
that
\begin{align}
 F
 &\simeq F_0\times\Z\langle T_1\rangle
                  \times\Z\langle T_2\rangle,
 \label{eq:F0-splitting}\\
 T_3&=aT_1^qT_2^{-q},
 \label{eq:central-element-a}\\
 H_{r,s}
 &\simeq F_0\times\Z\langle\rho_1\rangle
                  \times\Z\langle\rho_2\rangle,
 \label{eq:Hrs-splitting}\\
 \rho_3&=a\rho_1^{qr}\rho_2^{-qs}.
 \notag
\end{align}
\end{lemma}

\begin{proof}
The group \(F\) is also the stabilizer of the absolute homotopy class
of \(\xi\).
Indeed, homotopy classes of vector framings and of line fields are affine
spaces over \(H^1(S;\Z)\), and projectivization multiplies difference
classes by two.  Since \(H^1(S;\Z)\) is torsion free, preservation of
\(\mathbb R\xi\) is equivalent to preservation of \(\xi\).

Fix the value of \(\xi\) on the boundary.  Relative framings inducing the
same absolute framing form a torsor over
\begin{equation*}
 L_\partial=
 \ker\bigl(H^1(S,\partial S;\Z)\longrightarrow H^1(S;\Z)\bigr)
 \simeq\Z^3/\Z(1,1,1).
\end{equation*}
Let \(u_i\) be the image of the \(i\)-th standard basis vector, so that
\(u_1+u_2+u_3=0\).  For \(f\in F\), choose a collar-fixed representative
and define
\[
 \beta(f)=[f^*\xi]-[\xi]\in L_\partial.
\]
A boundary-pointwise mapping class acts trivially on \(L_\partial\), by naturality
of the cohomology sequence of the pair.  The cocycle identity therefore
reduces to \(\beta(fh)=\beta(f)+\beta(h)\), so \(\beta:F\to L_\partial\) is a
homomorphism.

By Lemma~\ref{lem:relative-displacement}, one common sign
\(\varepsilon_0\in\{1,-1\}\) satisfies
\[
 \beta(T_i)=\varepsilon_0\omega_\xi(B_i)u_i.
\]
Write \(\beta(f)=xu_1+yu_2\) and set
\[
 \chi(f)=\varepsilon_0(x,-y)\in\Z^2.
\]
Using \(u_3=-u_1-u_2\) and the prescribed boundary windings gives
\begin{equation*}
 \chi(T_1)=(1,0),\qquad
 \chi(T_2)=(0,1),\qquad
 \chi(T_3)=(q,-q).
\end{equation*}
Put \(F_0=\ker\chi=\ker\beta\).
This is precisely the subgroup preserving the relative homotopy
class of \(\xi\), as defined in Section~\ref{sec:fibres}.  The twists \(T_1,T_2\) are central in \(F\), and
their images form the standard basis.  This proves
\eqref{eq:F0-splitting}.  The third identity above gives the central
element
\[
 a=T_3T_1^{-q}T_2^q\in F_0,
\]
which proves \eqref{eq:central-element-a}.

The three boundary components are fixed because the line-field windings
\((2,-2,-2q)\) are pairwise distinct.  Restriction to the boundary gives
\begin{equation}\label{eq:boundary-restriction-sequence}
 1\longrightarrow F\longrightarrow H_{r,s}
 \longrightarrow\Z/r\Z\times\Z/s\Z\longrightarrow1.
\end{equation}
The rotations \(\rho_1,\rho_2\) give surjectivity.  An element in the
kernel fixes every stop on each boundary circle.  Isotopy extension makes
it the identity on all boundary collars, so it lies in \(F\).

Collar-fixed representatives of \(F\) commute with the \(\rho_i\), and
\begin{equation*}
 \rho_1^r=T_1,\qquad \rho_2^s=T_2,\qquad \rho_3=T_3.
\end{equation*}
These are the only new relations.  Indeed, any relation first projects to
\eqref{eq:boundary-restriction-sequence}, so the exponents of
\(\rho_1,\rho_2\) are divisible by \(r,s\), respectively, after which it
is a relation in \(F\).  Substitution into
\eqref{eq:F0-splitting} and \eqref{eq:central-element-a} proves
\eqref{eq:Hrs-splitting} and the formula for \(\rho_3\).
\end{proof}

\subsection{Zero-winding generators}
\label{appsubsec:zero-winding-generators}

To eliminate the local-system factor in the Serre element, we shall test
its holonomy on zero-winding curves. The following spanning statement
ensures that these tests determine the entire local system.

\begin{lemma}\label{lem:gradable-curves-span}
Let \(S\) have genus at least two, and let \(\xi\) be a vector framing
such that \(\omega_\xi(B)\) is odd on every boundary component.
Then \(H_1(S;\Z)\) is generated by classes of nonseparating simple
closed curves \(c\) with \(\omega_\xi(c)=0\). These curves may be chosen
outside any prescribed boundary collars.
\end{lemma}

\begin{proof}
Let \(\eta\) be the line field induced by \(\xi\), so that
\(\omega_\eta=2\omega_\xi\). The quadratic form
\[
 Q_\xi([c])=\omega_\xi(c)+1\pmod2
\]
on \(H_1(S;\Z/2)\) vanishes on the boundary radical and descends to
the capped surface. See Theorem~\ref{thm:line-field-classification}.
We first show that every integral class \(x\) admitting a nonseparating
simple representative and satisfying \(Q_\xi(\bar x)=1\) has a
zero-winding simple representative. Choose such a representative
\(c\). Its winding is even. Choose a
nonseparating curve \(d\) with \([d]\cdot[c]=1\) and complete it to a
genus-one bounding pair \((d,e)\) containing no original boundary
component. Such a pair exists in genus at least two, as in the proof
of Lemma~\ref{lem:stabilizer-determines-line-field}. Its bounding-pair
map acts trivially on integral homology. Orient \(d\) and \(e\)
homologously. Since they bound a genus-one subsurface containing no
original boundary component,
\[
 \omega_\eta(d)-\omega_\eta(e)=\pm4.
\]
The formulas
\eqref{eq:weighted-bounding-pair},
\eqref{eq:winding-doubling}, and
\eqref{eq:bounding-pair-winding-difference} give
\[
 \Delta\omega_\eta(c)
 =\bigl(\omega_\eta(d)-\omega_\eta(e)\bigr)([d]\cdot[c])=\pm4.
\]
Thus the framing winding changes by \(2\) or \(-2\). Iterating this map
therefore gives a simple representative of \(x\) with winding zero.

A nondegenerate quadratic space of dimension at least four over
\(\Z/2\) has a symplectic pair on which the quadratic form is one
on both vectors. Lifting a symplectic basis with this first pair,
choose integral handle classes \(a_j,b_j\) with
\(Q_\xi(\bar a_1)=Q_\xi(\bar b_1)=1\). If \(j>1\) and
\(Q_\xi(\bar a_j)=0\), replace \(a_j\), for the purpose of generating
homology, by \(a_j+a_1\). Do the same with \(b_j\) when necessary.
All the resulting primitive nonseparating classes have quadratic
value one and, together with \(a_1,b_1\), generate the handle classes.
For each boundary class \(c_i=[B_i]\), a band sum represents
\(a_1+c_i\) by a nonseparating simple curve, and
\(Q_\xi(\bar a_1+\bar c_i)=1\). Thus \(c_i\) is the difference of
two classes of quadratic value one. These classes generate all of
\(H_1(S;\Z)\), and the first paragraph gives the required
zero-winding representatives. All curves may be chosen outside fixed
boundary collars.
\end{proof}

\subsection{Normalization of the Serre element}
\label{appsubsec:serre-normalization}

We now combine the stabilizer presentation with the zero-winding
generators. Return to the notation of
Lemma~\ref{lem:framed-stabilizer-presentation}. For \(r,s\geq1\), put
\(\mathbb S_{r,s}=(S,M_{r,s},\eta)\) and
\(N=H^1(S;\kk^\times)\). Normalize the grading extension at the unique
stop \(p\) on \(B_3\), whose stop count remains fixed throughout the
family.

\begin{lemma}\label{lem:framed-family-serre-normalization}
The resulting identifications may be chosen as
\[
 D(\mathbb S_{r,s})
 \simeq (N\rtimes F_0)\times
        \Z\langle\rho_1\rangle\times
        \Z\langle\rho_2\rangle\times
        \Z\langle[1]\rangle
\]
so that
\[
 \mathsf S_{r,s}
 =a^{-1}\rho_1^{-(qr+1)}\rho_2^{qs-1}[1]^{e_g},
\]
where \(e_g\in\Z\) is independent of \(r,s\).
\end{lemma}

\begin{proof}
The boundary winding numbers of \(\eta\) are
\((2,-2,-2q)\), so the additive factor in
Theorem~\ref{thm:opper-geometrization} is absent. Normalization at \(p\)
splits the grading extension. The displayed group identification now
follows from that theorem and
Lemma~\ref{lem:framed-stabilizer-presentation}. The collar rotations
act trivially on \(N\), since they become isotopic to the identity after
the stops are forgotten. Moreover, \(a\) is a product of boundary
twists. It acts trivially on \(N\) and remains central in
\(N\rtimes F_0\).

Use the geometric section constructed in
\cite[Theorem~7.1]{Opper2025}, together with the graded lifts
normalized at \(p\). By
Proposition~\ref{prop:geometric-serre-functor}, the underlying
ungraded mapping class of the Serre functor is
\((\rho_1\rho_2\rho_3)^{-1}\). Since
\[
 (\rho_1\rho_2\rho_3)^{-1}
 =a^{-1}\rho_1^{-(qr+1)}\rho_2^{qs-1},
\]
the split geometrisation sequence gives
\[
 \mathsf S_{r,s}
 =n_{r,s}a^{-1}\rho_1^{-(qr+1)}
  \rho_2^{qs-1}[1]^{e_{r,s}},
\]
where the local-system factor \(n_{r,s}\in N\) and the shift
exponent \(e_{r,s}\in\Z\) remain to be determined.

Let \(S^\circ\) be the complement of the interiors of small boundary
collars. The normalized lifts of \(\rho_1,\rho_2\) are the identity
on the grading cover over \(S^\circ\). Their supports avoid both
\(p\) and this connected core. The normalized lift of \(a\) is the
identity as an ungraded map on the core, and hence acts there by a
fixed deck translation \(\kappa_g\). It depends only on \(\eta\), \(a\),
and the normalization at \(p\), and is therefore independent of \(r,s\). Consequently, the normalized lift of \(a^{-1}\) acts on the
grading cover over the core by the deck translation \(-\kappa_g\).
For a nonseparating curve \(c\subset S^\circ\) with
\(\omega_\xi(c)=0\), and hence \(\omega_\eta(c)=0\), let
\(X_{c,\lambda}\) be its rank-one band object with monodromy
\(\lambda\in\kk^\times\). Thus \(X_{c,1}\) has trivial local system.
For \(n\in N\) and an oriented closed curve \(c\), write
\[
 n(c)=\langle n,[c]\rangle\in\kk^\times.
\]
The local-system factor \(n\) sends \(X_{c,\lambda}\) to
\(X_{c,n(c)\lambda}\). This follows from the description in
\cite[Proposition~10.2]{Opper2025}. For the fixed graded
nonseparating simple closed curve \(c\), the classification of graded
band objects in \cite[Theorem~3.5]{Opper2025} gives
\[
 X_{c,\lambda}[d]\simeq X_{c,1}
 \quad\Longrightarrow\quad \lambda=1\text{ and }d=0.
\]
Indeed, an overall grading shift does not change the rank-one
monodromy parameter and reversing the orientation replaces it by its
inverse, which does not affect comparison with \(1\). Thus the
classification first gives \(\lambda=1\). A non-zero shift cannot
then fix this non-zero object since iterating such a periodicity would give
non-zero self-morphisms in infinitely many degrees, contrary to
properness.

Every band object is fixed by Auslander--Reiten translation
\cite[Proposition~D.1]{OpperPlamondonSchroll}. Therefore
\[
 X_{c,1}\simeq\mathsf S[-1](X_{c,1})
 \simeq X_{c,n_{r,s}(c)}[-\kappa_g+e_{r,s}-1],
\]
and the preceding implication yields
\[
 n_{r,s}(c)=1,
 \qquad -\kappa_g+e_{r,s}-1=0.
\]
The boundary windings of \(\xi\) are all odd, so
Lemma~\ref{lem:gradable-curves-span} shows that these curves generate
\(H_1(S;\Z)\). Hence
\[
 n_{r,s}=1,\qquad e_{r,s}=1+\kappa_g=:e_g
\]
for every \(r,s\).
\end{proof}

\section{Twisted homology and integral boundary matching}
\label{app:twisted-homology}

This appendix provides the homological calculations for
Section~\ref{sec:grothendieck-actions}.
Section~\ref{appsubsec:collar-twisted-homology} establishes the local
rotation formulas and their integral gluing.
Section~\ref{appsubsec:scalar-integral-matching} proves scalar rigidity
on absolute homology and the restrictions on negative boundary
permutations. Section~\ref{appsubsec:exceptional-isomorphisms} verifies
the represented isomorphisms in the boundary counterexamples.

We now return to the general setting of
Section~\ref{sec:grothendieck-actions} and retain the notation fixed
there.

\subsection{Boundary rotations and twisted homology}
\label{appsubsec:collar-twisted-homology}

Positive and negative boundary holonomies give different rotation
actions and integral lattices. For each \(i\), let \(\rho_i\) be the positive
one-step rotation at \(B_i\). The graded lift which is the identity
away from its rotation collar induces an operator \(R_i\) on
\(K_{\mathbb Q}\). If \(\nu_i=1\), choose a local generator along
\(B_i\), write \(c_i=[B_i]_L\), and let \(\lambda_i\) be the total
endpoint coefficient on \(M^\circ\cap B_i\). If \(\nu_i=-1\), choose
a primitive local generator at every \(x\in M^\circ\cap B_i\).
The negative half of the twisted boundary loop based at \(x\) defines
a rational collar class \(q_x\). Set
\[
 Q_i=\bigoplus_{x\in M^\circ\cap B_i}\mathbb Qq_x.
\]

\begin{lemma}\label{lem:normalized-twisted-collar}
The operator \(R_i\) satisfies \(R_i|_{A_{\mathbb Q}}=I\).
Its induced endpoint action \(\overline R_i\) on
\(H_0(M^\circ\cap B_i;L)\) is a signed cyclic permutation with
\[
 \det(XI-\overline R_i)=X^{m_i}-\nu_i.
\]
Moreover, \(R_i^{m_i}\) is the normalized full boundary-twist action.

If \(\nu_i=1\), then
\[
 (R_i^{m_i}-I)x
 =\varepsilon_\partial c_i\lambda_i(x),
 \qquad (R_i^{m_i}-I)^2=0
\]
for \(x\in K_{\mathbb Q}\),
where \(\varepsilon_\partial\in\{1,-1\}\) is independent of \(i\).

If \(\nu_i=-1\), then \(Q_i\) maps isomorphically to the endpoint
block and
\[
 R_i^{m_i}|_{Q_i}=-I,
 \qquad R_i^{2m_i}=I.
\]
It is the identity on the complementary collar-fixed summand.

In either case, if \(g\geq2\) and \(b\geq2\), the algebraic
multiplicity of \(1\) for \(R_i\) is greater than that of \(-1\).
This distinguishes \(R_i\) from its negative, and
\[
 \rk_{\mathbb Q}(R_i-I)
 =\rk_{\mathbb Q}(R_i^{-1}-I)=m_i.
\]
\end{lemma}

\begin{proof}
The collar rotation is joined to the identity by an isotopy supported
in the collar, during which the stops need not be preserved. The
induced graded lift is the identity on the complement. Any two graded
lifts differ by a global deck transformation, so this normalization is
unique. Following an endpoint once around \(B_i\) multiplies its local
coefficient by \(\nu_i\). This gives the signed cycle and the assertion
about the full power. Absolute cycles may be moved away from the
interior of the collar.

For positive holonomy, a full turn adds the boundary cycle once for
each endpoint coefficient. The added class has no endpoint coefficient,
which proves the square-zero formula and the independence of the collar
sign.

For negative holonomy, choose a core vertex \(v\), a based boundary
loop \(c\), and paths \(u_x\) to the endpoints. Choose local generators
so that \(\partial c=-2v\) and \(\partial u_x=x-v\) in the absolute
chain complex. After quotienting by the endpoint chains, these become
\[
 \partial c=-2v,
 \qquad \partial u_x=-v,
 \qquad q_x=u_x-\frac{c}{2}.
\]
Thus \(q_x\) is a rational relative cycle whose image under the
boundary map of the pair is the chosen generator at \(x\).
The normalized full twist fixes \(c\) and sends \(u_x\) to
\(c-u_x\). It therefore sends \(q_x\) to \(-q_x\). Over
\(\mathbb Q\), the negative-holonomy boundary circle has zero homology.
Mayer--Vietoris splits off its relative collar homology, and the twist
is the identity on the complementary summand. Its square is the
identity over \(\mathbb Q\), and hence over \(\Z\) because the charge
lattice is torsion free.

We now prove the rank formula. The action on the core has a
\(1\)-eigenspace of dimension at least \(2g+b-2\), while a single
endpoint block contributes at most one \(-1\)-eigenvector. This
selects the stated sign. If \(\nu_i=-1\), the action has finite
order, and its signed cyclic block has no invariant vector. It therefore
contributes rank \(m_i\).

Suppose that \(\nu_i=1\). We claim that the class \(c_i\) is non-zero. For
trivial \(L\) this follows from \(b\geq2\), while for non-trivial
\(L\) it follows from \(H_2(S,\partial S;L)=0\) and the exact
sequence of the pair. The endpoint total \(\lambda_i\) is also
non-zero. For non-trivial \(L\), the map to \(H_0(M^\circ;L)\) is
surjective because \(H_0(S;L)=0\). For trivial \(L\), an arc to
another boundary component has non-zero total at \(B_i\). Hence
\(R_i^{m_i}-I\) has rank one. The non-trivial eigenvalues of the
cyclic block contribute \(m_i-1\) to \(\operatorname{rk}(R_i-I)\).
On the generalized \(1\)-eigenspace, taking a positive power does
not change the rank of the difference from the identity. Indeed, in the
factorization
\[
 R_i^{m_i}-I=(R_i-I)(I+R_i+\cdots+R_i^{m_i-1}),
\]
the second factor is invertible on that subspace. The remaining
contribution is therefore one, so the rank is \(m_i\) in both cases.
Finally, \(R_i^{-1}-I=-R_i^{-1}(R_i-I)\), and inversion does not
change the rank.
\end{proof}

We normalize the pairing \(L_{\Z}\otimes L_{\Z}\to\Z\) by sending
the square of a primitive local generator to \(1\), and write
\(\langle-,-\rangle_L\) for the induced twisted intersection pairing.
In the next lemma,
\(x\) lies in \(H_1(S;L)\) or \(H_1(S,M^\circ;L)\). In the latter case,
the intersection covector is induced by the natural map to
\(H_1(S,\partial S;L)\).

In the local-system calculations below, a lift means an isomorphism \(f^*L_{\Z}\simeq L_{\Z}\). For a nonseparating curve \(c\), the condition
	\(\operatorname{hol}_L(c)=1\) ensures that \(T_c\) preserves the
	local system, whereas preservation of the line-field homotopy class
	requires \(\omega_\eta(c)=0\).
    
\begin{lemma}\label{lem:twisted-picard-lefschetz}
The following statements hold.
\begin{enumerate}[label=\textup{(\roman*)}]
\item If \(c\) is a simple closed curve with
\(\operatorname{hol}_L(c)=1\), choose a primitive local generator along
\(c\). The lift of \(T_c\) fixed away from its twisting annulus satisfies
\[
 (T_c-I)x
 =\varepsilon_{\mathrm{tw}}
   \langle x,[c]_L\rangle_L[c]_L,
 \qquad \varepsilon_{\mathrm{tw}}\in\{1,-1\}.
\]
\item Suppose that \(a,b\) are simple closed curves meeting once, with
\(\operatorname{hol}_L(a)=-1\) and
\(\operatorname{hol}_L(b)=1\). If \(\gamma\) bounds their
one-holed-torus neighbourhood, choose a primitive generator along \(b\)
and the lift of \(T_\gamma\) fixed outside its twisting annulus. Then
\[
 [\gamma]_L=\pm2[b]_L,
 \qquad
 T_\gamma-I=\pm4[b]_L\langle-,[b]_L\rangle_L.
\]

\item If \(c\) is separating and \(\operatorname{hol}_L(c)=-1\), cutting
along \(c\) gives a rational Mayer--Vietoris decomposition. The lift of
\(T_c\) normalized to be the identity on one side acts by \(I\) there
and by \(-I\) on the other, on both absolute and relative twisted homology.
\end{enumerate}
\end{lemma}

\begin{proof}
Let \(\widetilde S\to S\) be the double cover defined by \(L\).
Over \(\mathbb Q\), twisted homology identifies with the anti-invariant
part of the homology of \(\widetilde S\). We normalize its intersection
form as one half of the intersection form on the cover. A
positive-holonomy curve has two disjoint lifts, and its twist lifts to
the product of their commuting twists. Restricting the ordinary
Picard--Lefschetz formula to the anti-invariant part proves the first
assertion over \(\mathbb Q\). The same argument applies after lifting
the relative endpoint set.

These formulas are compatible with the integral lattices. Indeed, the
transported-coefficient intersection pairing is integral and a primitive local generator gives an integral
class \([c]_{L_{\Z}}\). Both sides of the formula therefore preserve the
integral twisted homology lattice. Since that lattice is torsion free,
the rational equality implies the integral equality.

Orient \(a,b\) so that the boundary word of their regular
neighbourhood is \(aba^{-1}b^{-1}\). Write \(a,b\) also for the
corresponding based twisted \(1\)-chains. The Fox derivatives, evaluated
in the sign character of \(L\), give the chain identity
\[
 \gamma=\bigl(1-\operatorname{hol}_L(b)\bigr)a
 +\bigl(\operatorname{hol}_L(a)-1\bigr)b=-2b.
\]
Since \(b\) has positive holonomy, it is a twisted cycle. Passing to
homology gives \([\gamma]_L=-2[b]_L\), and substitution into the first
formula gives the second assertion.

For negative holonomy, the homology of the cutting circle vanishes over
\(\mathbb Q\), so Mayer--Vietoris gives the asserted direct sum.
Crossing the twisting annulus transports a local coefficient once
around that circle. Relative to the normalization on the first side,
the coefficient on the other side is multiplied by \(-1\).
\end{proof}

\begin{lemma}\label{lem:twisted-homology-normal-forms}
Assume that \(g\geq2\) and that \(L_{\Z}\) is non-trivial. The positive
boundary classes form a basis of the radical of the twisted intersection
form on \(A_{\mathbb Q}\).

If every boundary has positive holonomy, there is a geometric handle
basis on which the character of \(L\) is non-trivial only on the first
handle. The remaining handles span a non-degenerate complement of
dimension \(2g-2\).

If some boundary has negative holonomy, there is a geometric handle basis on which \(L\) is
trivial. Its handle core has connected planar complement. The handle
homology is non-degenerate of dimension \(2g\), and the remaining
rational homology is generated by the positive boundary curves and by
planar curves enclosing pairs of negative boundaries.
\end{lemma}

\begin{proof}
Poincar\'e--Lefschetz duality identifies the radical with the image of
boundary homology. Since \(L\) is non-trivial, we have
\(H_2(S,\partial S;L)=0\). Hence, the radical has dimension equal to
the number of positive boundary components.

When every boundary holonomy is positive, the character extends over
the capped surface. A geometric symplectic basis concentrates its
non-zero mod-two character on the first handle. The remaining handle
homology is non-degenerate, and dimensions give the asserted rational
decomposition.

Suppose now that \(B\) has negative holonomy. Pushing \(B\) along a
handle curve \(a\) acts on ordinary homology, up to the push convention,
by
\[
 x\longmapsto x+\langle x,a\rangle[B].
\]
The mod-two character evaluates non-trivially on \([B]\). Therefore pushing along
\(a\) changes the value of the character on the dual handle generator. Since boundary classes lie in the intersection radical, successive pushes
eliminate all handle evaluations while leaving the boundary evaluations
unchanged. In
the resulting planar complement, choose a graph spine with core vertex
\(v\). Positive boundary loops have twisted differential zero, while
negative boundary loops have differential \(-2v\). The kernel is generated by the positive loops and differences
of negative loops. Curves enclosing pairs of negative boundaries realize
these differences modulo the positive boundaries. Gluing the handle
core kills the outer-boundary class and gives the asserted spanning set
in \(H_1(S;L)\).
\end{proof}

We record the integral forms used below. Whenever a boundary is
negative, we use the handle basis from the preceding lemma on which
\(L_{\Z}\) is trivial. Suppose first that exactly two boundary
components are negative, numbered \(B_1\) and \(B_2\). Write
\(b_+\) for the number of positive boundary components. Choose a graph
spine with core vertex \(v\), based negative boundary chains \(c_1,c_2\),
and positive boundary cycles \(d_1,\ldots,d_{b_+}\). If \(H_{\Z}\) is the
handle lattice, local generators may be chosen so that
\begin{equation}\label{eq:two-negative-absolute-normal-form}
 \partial c_1=\partial c_2=-2v,
 \qquad c_2-c_1=d_1+\cdots+d_{b_+},
 \qquad
 A_{\Z}=H_{\Z}\oplus\bigoplus_{i=1}^{b_+}\Z d_i.
\end{equation}
Indeed, the twisted surface relation is
\(c_1-c_2+\sum_i d_i=0\), and the planar chain complex has no further
relation.

For the other normal form used below, assume that \(b=4\) and every
boundary is negative. Choose a graph spine with core vertex \(v\) and
based boundary chains \(c_1,\ldots,c_4\) with
\(\partial c_i=-2v\). The surface relation gives
\[
 c_4=c_1-c_2+c_3.
\]
The kernel of the planar boundary map is then
\[
 \left\{\sum_{i=1}^3x_ic_i\ \middle|\
 x_i\in\Z,\ \sum_{i=1}^3x_i=0\right\}.
\]

For the next lemma, assume that there is a negative boundary and that
\(L_{\Z}\) is trivial on a chosen handle basis. Fix one such boundary,
denoted \(B_{i_0}\).
Choose a graph spine with core vertex \(v\), based negative boundary
chains \(c_i\) with \(\partial c_i=-2v\), and paths \(u_x\) from
\(v\) to the endpoints. For \(x\in M^\circ\cap B_i\), set
\[
 q_x=
  \begin{cases}
  u_x-c_i/2,&\nu_i=-1,\\
  u_x-c_{i_0}/2,&\nu_i=1.
  \end{cases}
\]
For a family of coefficients \((n_x)_{x\in M^\circ}\), write
\(n_i=\sum_{x\in M^\circ\cap B_i}n_x\).

\begin{lemma}\label{lem:integral-collar-gluing}
The collar classes give a rational decomposition
\[
 K_{\mathbb Q}=A_{\mathbb Q}\oplus
 \bigoplus_{x\in M^\circ}\mathbb Qq_x.
\]
In these coordinates,
\[
 K=\left\{a+\sum_xn_xq_x\ \middle|\
 \begin{array}{c}
  a\in A_{\mathbb Q},\ n_x\in\Z,\ \sum_xn_x\in2\Z,\\[2pt]
  a-\frac12\sum_{\nu_i=-1}n_i(c_i-c_{i_0})\in A_{\Z}
 \end{array}\right\}.
\]
\end{lemma}

\begin{proof}
The endpoint map first forces every \(n_x\) to be integral. Since
\(H_0(S;L_{\Z})=\Z/2\), its image consists exactly of the endpoint
vectors of even total. Thus \(\sum_xn_x\in2\Z\). Expanding the
collar coordinates gives
\[
 a+\sum_xn_xq_x
 =a-\frac12\sum_{\nu_i=-1}n_i(c_i-c_{i_0})
  +\sum_xn_xu_x-\frac12\left(\sum_xn_x\right)c_{i_0}.
\]
The last two terms are now integral. The whole expression is therefore
integral if and only if the first term belongs to \(A_{\Z}\).
Conversely, these conditions produce an integral relative cycle, proving
the formula.
\end{proof}

\subsection{Scalar rigidity and boundary matching}
\label{appsubsec:scalar-integral-matching}

We now pass from the recovered absolute charge lattice to the geometric
assignment of endpoint blocks.

\begin{lemma}\label{lem:separating-twisted-scalar-commutant}
Assume that \(S\) has genus at least two and that \(L\) is non-trivial.
For every curve
\(\gamma\) bounding a one-holed torus in the interior of \(S\), give
\(T_\gamma\) the lift fixed away from its twisting annulus. The
commutant of these actions in
\(\operatorname{End}_{\mathbb Q}(A_{\mathbb Q})\) is
\(\mathbb Q I_{A_{\mathbb Q}}\).
\end{lemma}

\begin{proof}
First consider a positive-holonomy nonseparating simple curve \(b\).
It has a negative-holonomy simple partner \(a\) meeting it once.
Indeed, choose any once-intersecting partner. If both curves have
positive holonomy, their one-holed torus has trivial local system, so
its connected complement has non-trivial local system. This complement
has a negative-holonomy simple closed curve. Simple handle and boundary
curves generate its first homology, and the character is
non-trivial on at least one such generator. Join that curve to the
chosen partner by an embedded band whose interior avoids \(b\).
The resulting simple curve still meets \(b\) once, while its holonomy
is the product of the two original holonomies and is therefore negative.
The band and the curves may all be chosen in the interior of \(S\).

The boundary \(\gamma\) of a regular neighbourhood of \(a\cup b\)
belongs to the family in the statement. By
Lemma~\ref{lem:twisted-picard-lefschetz},
\[
 [\gamma]_L=\pm2[b]_L,\qquad
 T_\gamma-I=\pm4N_b,\qquad
 N_b(x)=\langle x,[b]_L\rangle_L[b]_L.
\]
Thus an endomorphism commuting with the stated twists commutes with
\(N_b\) for every positive-holonomy nonseparating curve \(b\).
Whenever \(N_b\neq0\), it preserves
\(\operatorname{im}N_b=\mathbb Q[b]_L\).

Choose the normal form in
Lemma~\ref{lem:twisted-homology-normal-forms}. It gives a
non-degenerate handle subspace \(W\) of dimension \(2g-2\) when all
boundary holonomies are positive, and of dimension \(2g\) otherwise.
Put \(V_{\mathrm{comp}}=W^\perp\), so that
\(
 A_{\mathbb Q}=W\oplus V_{\mathrm{comp}}\).
The complementary subspace is spanned by positive boundary classes and,
when needed, by positive-holonomy curves enclosing pairs of negative
boundaries. These curves may be chosen disjoint from the handle core.

For a twisted homology class \(v\), we also write
\(N_v(x)=\langle x,v\rangle_Lv\). Let \(U\) commute with the stated
twists, and choose a geometric symplectic basis \(e_1,f_1,\ldots,e_h,f_h\) of \(W\), where
\(h=\tfrac12\dim W\geq1\). The associated handle curves have positive
holonomy. Commutation with \(N_{e_i}\) and \(N_{f_i}\) makes the
lines \(\mathbb Qe_i\) and \(\mathbb Qf_i\) invariant. Applying
\(UN_{e_i}=N_{e_i}U\) to \(f_i\) shows that the two eigenvalues agree.
Write them as \(a_i\). Moreover,
\[
 V_{\mathrm{comp}}
 =\bigcap_{i=1}^h\bigl(\ker N_{e_i}\cap\ker N_{f_i}\bigr),
\]
so this subspace is also \(U\)-invariant. For \(i>1\), a band sum of
handle curves represents \(e_1+e_i\), after compatible choices of local
generators. Its transvection is non-zero, and invariance of
\(\mathbb Q(e_1+e_i)\) gives \(a_i=a_1\). Hence
\(U|_W=\lambda I_W\), with \(\lambda=a_1\). This also covers
\(h=1\), when no comparison between different handles is required.

For each generator \(u\) of \(V_{\mathrm{comp}}\) described above,
band-sum its curve with the handle curve representing \(e_1\). The
result is positive-holonomy and nonseparating, with twisted class
\(e_1\pm u\). Since its pairing with \(f_1\) is non-zero, its
transvection is non-zero. Thus \(U\) preserves
\(\mathbb Q(e_1\pm u)\). Comparing the \(W\)-coordinate with
\(Ue_1=\lambda e_1\), and using the invariance of
\(V_{\mathrm{comp}}\), gives \(Uu=\lambda u\). These generators span
\(V_{\mathrm{comp}}\), so \(U=\lambda I_{A_{\mathbb Q}}\). Conversely,
scalar endomorphisms commute with every action, proving the assertion.
\end{proof}

\begin{lemma}\label{lem:absolute-intertwiner-scalar}
Let \(\mathbb S=(S,M,\eta)\) and \(\mathbb S'=(S',M',\eta')\) have genus
at least two and at least two boundary components. Let
\((\Phi,\Theta):\mathfrak R(\mathbb S)\to\mathfrak R(\mathbb S')\) be
an isomorphism, and let \(f\) and \(\phi\) be supplied by
Corollary~\ref{cor:geometric-reduction}. After
choosing an integral identification of the orientation local systems,
\[
 \Theta|_{H_1(S;L_{\Z})}=\varepsilon f_*,
 \qquad \varepsilon\in\{1,-1\}.
\]
\end{lemma}

\begin{proof}
By Proposition~\ref{prop:represented-boundary-data}, \(\Theta\)
carries the saturated absolute lattice on the source to the one on the
target. Transport the target by \(f\) and write \(V\) for the restriction
of the resulting intertwiner to \(A_{\mathbb Q}\).

Suppose first that \(L\) is trivial. Graded equivariance gives a character
\(e:H_\eta\to\{1,-1\}\) such that
\(
 Vy_*=e(y)y_*V\). The kernel of \(e\) has finite index. Choose a vector framing
inducing \(\eta\) and fix its boundary values. The relative framing
stabilizer is contained in the stabilizer of the absolute framing, hence in \(H_\eta\). By
\cite[Theorem~C]{CalderonSalterRelative}, for \(g\geq2\), its
action on \(H_1(S,\partial S;\Z)\) has finite-index image in the pure
homological automorphism group. This image is the kernel of a crossed
homomorphism with finite target. Poincar\'e--Lefschetz duality
transfers this statement to absolute homology. Let \(R_\partial\)
be the boundary radical and choose a symplectic complement \(W\).
In the splitting \(A_{\mathbb Q}=R_\partial\oplus W\), the resulting
Zariski closure is the pure parabolic group
\[
 \left\{\begin{pmatrix}I&B\\0&S\end{pmatrix}:
 S\in\operatorname{Sp}(W),\quad
 B\in\operatorname{Hom}(W,R_\partial)\right\}.
\]
This algebraic group is connected. Therefore, intersecting the relative framing
stabilizer with \(\ker e\) does not change its Zariski closure, so \(V\) commutes with the displayed group.
The commutant of its symplectic subgroup is block diagonal and scalar
on \(W\). Commutation with every \(B\) makes it scalar on the boundary
radical with the same scalar. Hence \(V=aI\). Integrality gives
\(a=\pm1\). The equivariance equation then forces \(e(y)=1\) for every
\(y\).

Now suppose that \(L\) is non-trivial. For every one-holed-torus
separating curve \(\gamma\), let \(G_\gamma\) be the local-system action
normalized outside its twisting annulus. Corollary~\ref{cor:geometric-reduction}
and Lemma~\ref{lem:geometrization-charge-action} give
\[
 VG_\gamma V^{-1}=e_\gamma G_\gamma,
 \qquad e_\gamma\in\{1,-1\}.
\]
The operator \(G_\gamma\) is unipotent, while \(-G_\gamma\) has only
the eigenvalue \(-1\). Thus \(e_\gamma=1\). Lemma~\ref{lem:separating-twisted-scalar-commutant}
shows that \(V\) is scalar. Preservation of the integral lattice again
gives \(V=\pm I\). Restoring \(f_*\) proves the lemma.
\end{proof}

Fix an isomorphism
\[
 (\Phi,\Theta):\mathfrak R(\mathbb S)
 \xrightarrow{\sim}\mathfrak R(\mathbb S')
\]
between represented invariants of surfaces of genus at least two with
at least two boundary components. Transport the target by the
diffeomorphism \(f\) supplied by
Corollary~\ref{cor:geometric-reduction}, assume that \(L\) is non-trivial,
and let \(\pi\) be the induced permutation of the negative endpoint
blocks.

\begin{lemma}\label{lem:integral-negative-boundary-matching}
The following statements hold.
\begin{enumerate}[label=\textup{(\roman*)}]
\item For distinct negative boundaries \(B_i,B_j\),
\[
 [B_i]+[B_j]=[B_{\pi(i)}]+[B_{\pi(j)}]
 \quad\text{in }H_1(S;\mathbb F_2).
\]
If there are at least four negative boundaries, every negative block
is fixed unless \(b=4\) and every boundary is negative. In that case
\(\pi\) belongs to the Klein four group.
\item If there are exactly two negative boundaries and at least two
positive boundaries, then both negative blocks are fixed.
\end{enumerate}
\end{lemma}

\begin{proof}
For a negative boundary, let \(Q_i\subset K_{\mathbb Q}\) be its
canonical sum of non-trivial primary spaces. A local generator at an
endpoint \(x\) gives a collar class \(q_x\), the negative half of the
twisted boundary loop based at \(x\), with endpoint coefficient \(x\).
Thus \(Q_i\) maps isomorphically to its endpoint block.
Projection of \(K\) to \(Q_i\) gives a lattice
\(\Lambda_i\cong\Z^{m_i}\). Its intersection
\(\Lambda_i^{\mathrm{ev}}=Q_i\cap K\) consists of the vectors with even
total coefficient. Hence
\(\Lambda_i/\Lambda_i^{\mathrm{ev}}\cong\mathbb F_2\).
The assertion follows directly from the relative chain complex of
the negative-holonomy circle. Boundary loops have endpoint coefficient
\(\pm2\), while paths between endpoints give signed differences.
Projection realizes an odd coordinate because another boundary can
supply the second odd endpoint. Both lattices are intrinsic.

For two negative blocks choose \(v_i\in\Lambda_i\) and
\(v_j\in\Lambda_j\) with odd parity. There is an integral
class \(x\in K\) with these endpoint coordinates and no others.
The image of the endpoint map is the even-total lattice, as
\(H_0(S;L_{\Z})=\Z/2\). Write
\[
 a=2x-2v_i-2v_j\in A_{\Z}.
\]
Its class \(\delta_{ij}\in A_{\Z}/2A_{\Z}\) is
independent of all choices. For another choice
\(x',v_i',v_j'\), the differences \(v_i'-v_i\) and
\(v_j'-v_j\) lie in \(\Lambda_i^{\mathrm{ev}}\) and
\(\Lambda_j^{\mathrm{ev}}\). They are integral, and
cancellation of the endpoint coordinates gives
\[
 x'-x-(v_i'-v_i)-(v_j'-v_j)\in A_{\Z}.
\]
The two resulting absolute classes therefore differ by an element of
\(2A_{\Z}\). An integral intertwiner carries \(\delta_{ij}\) to the
corresponding class for the permuted blocks.

Reduction of the twisted chain complex modulo two gives the ordinary
chain complex. Since \(H_0(S;L_{\Z})\simeq\Z/2\), the universal
coefficient sequence is
\[
 0\longrightarrow A_{\Z}/2A_{\Z}
 \longrightarrow H_1(S;\mathbb F_2)
 \longrightarrow
 \operatorname{Tor}_1^{\Z}(H_0(S;L_{\Z}),\mathbb F_2)
 \longrightarrow0.
\]
The class of \(2x\) reduces to zero, while \(2v_i\) and \(2v_j\) reduce
to the two ordinary boundary loops.
Consequently \(\delta_{ij}\)
maps to \([B_i]+[B_j]\). The injection in the displayed universal
coefficient sequence is essential here. In particular, one must not
identify \(A_{\Z}/2A_{\Z}\) with all of \(H_1(S;\mathbb F_2)\), since the
final Tor term is non-zero. By
Lemma~\ref{lem:absolute-intertwiner-scalar}, the absolute intertwiner
is \(\pm f_*\). Naturality of the displayed injection therefore
identifies the two boundary sums. The common sign disappears modulo
two, giving the asserted identity.

Over \(\mathbb F_2\), the boundary classes have the single relation
\(\sum_i[B_i]=0\). If \(b\ne4\) and there
are at least four negative boundaries, equality of two sums of
distinct pairs forces equality of the unordered pairs. A permutation
of at least three elements fixing every unordered pair is the
identity.

For \(b=4\), a pair may instead be replaced by its
complement. The permutations preserving the three partitions into
complementary pairs form the Klein four group. This proves
\textup{(i)}.

By Proposition~\ref{prop:geometric-boundary-matching}, the positive
boundaries have matching stop numbers. Identify their stop sets by
isotopies in boundary collars. Multiplying the intertwiner by a common
sign, we may assume
\(
 \Theta|_{A_{\mathbb Q}}=I\). Write \(b_+\) for the number of positive boundary components.
Write \(B^-_1,B^-_2\) for the negative boundaries and
\(B^+_1,\ldots,B^+_{b_+}\) for the positive boundaries. Denote the
positive boundary classes by \(d_1,\ldots,d_{b_+}\), and set
\[
 R=\operatorname{span}_{\mathbb Q}\{d_1,\ldots,d_{b_+}\}.
\]
The classes \(d_i\) are independent by
Lemma~\ref{lem:twisted-homology-normal-forms}.

Set \(F=\Stab_{\Mod(S,\partial S)}(\eta)\).
Proposition~\ref{prop:represented-boundary-data} identifies the
boundary-pointwise stabilizers. After the geometric transport, the
induced isomorphism on their common underlying group \(F\) has the form
\[
 y\longmapsto yz(y),
 \qquad z(y)=\prod_B T_B^{n_B(y)}.
\]
Here \(B\) ranges over the boundary components and
\(n_B(y)\in\Z\).
Indeed, the two maps have the same image after the stops are forgotten,
while their difference on \(F\) lies in the central lattice generated
by the full boundary twists.
The full twists, with the complement-fixed normalization of
Lemma~\ref{lem:normalized-twisted-collar}, act trivially on absolute
twisted homology. Since \(\Theta|_{A_{\mathbb Q}}=I\), the equivariance equation has no
additional common sign when the same local-system lift of \(y\) is
used on both sides.

We now prove that
\begin{equation}\label{eq:two-negative-fixed-totals}
 \lambda_i\Theta=\lambda_i
 \qquad\text{ for }1\leq i\leq b_+,
\end{equation}
where \(\lambda_i\) is the endpoint total at the \(i\)-th positive
boundary \(B_i^+\). Let \(\beta\) be a
positive-holonomy nonseparating curve with nonradical twisted class.
Choose a negative-holonomy partner meeting it once, as in the proof of
Lemma~\ref{lem:separating-twisted-scalar-commutant}. The boundary
\(\gamma\) of their regular neighbourhood is separating. By
Lemma~\ref{lem:twisted-picard-lefschetz}, its twist acts on relative
homology as
\[
 I+N_\beta,
 \qquad N_\beta=a_\beta[\beta]\otimes\lambda_\beta,
 \qquad a_\beta\neq0,
\]
where \(\lambda_\beta\) is the twisted intersection covector. This action
is the identity on the endpoint quotient.

The image of \(T_\gamma\) under the represented isomorphism must also
act identically on that quotient. The negative-boundary exponents in
\(z(T_\gamma)\) are therefore even, and their relative actions are
trivial by Lemma~\ref{lem:normalized-twisted-collar}. The remaining
central correction has the form
\[
 I+D,
 \qquad
 D=\varepsilon_\partial\sum_i n_i d_i\otimes\lambda_i,\qquad
 n_i\in\Z.
\]
Since \(N_\beta D=DN_\beta=0\) and \(\Theta N_\beta=N_\beta\), expansion
of the equivariance equation gives
\[
 N_\beta-N_\beta\Theta=D\Theta.
\]
The left-hand side has image in \(\mathbb Q[\beta]\), while the right-hand
side has image in \(R\). These spaces intersect trivially because
\([\beta]\) is nonradical. Hence, we have
\(\lambda_\beta\Theta=\lambda_\beta\).

Band-sum \(\beta\) with \(B_i^+\). The resulting
curve is positive-holonomy and nonseparating, with class
\([\beta]\pm d_i\). The same argument applies to it. With compatible choices of signs, the
twisted intersection pairing gives
\[
 \lambda_{\beta\pm d_i}=\lambda_\beta\pm\lambda_i.
\] Subtracting the invariant
covectors proves \eqref{eq:two-negative-fixed-totals}.

Let \(M_+^\circ\) be the set of endpoints on positive boundaries and put
\[
 C_-=\operatorname{span}_{\mathbb Q}\{r_i:\nu_i=-1\}
 \subset C_{\mathbb Q}.
\]
The collar
decomposition and Lemma~\ref{lem:normalized-twisted-collar} identify
\[
 K_{+,\mathbb Q}
 =K_{\mathbb Q}^{C\cap C_-}
 =H_1(S,M_+^\circ;L).
\]
The action on this fixed subspace factors through
\(C/(C\cap C_-)\). Let \(V_0\) be the simultaneous generalized
trivial-character subspace for this quotient action. This definition
is intrinsic and does not require a choice of lifts of the positive
quotient axes to \(C\). In particular, the represented intertwiner
carries \(V_0\) to its target counterpart. Collar rotations commute
with the \(F\)-action, so \(V_0\) is an \(F\)-submodule. It fits into
\[
 0\longrightarrow A_{\mathbb Q}
 \longrightarrow V_0
 \xrightarrow{(\lambda_1,\ldots,\lambda_{b_+})}
 \mathbb Q^{b_+}\longrightarrow0.
\]
The positive endpoint sets have already been identified. The inclusion
of pairs therefore identifies both copies of \(K_{+,\mathbb Q}\) with
the same \(H_1(S,M_+^\circ;L)\), including their uncorrected \(F\)-actions
and the subspace \(V_0\). The actual actions differ only by the central
boundary corrections above. Together with
\(\Theta|_{A_{\mathbb Q}}=I\), equation
\eqref{eq:two-negative-fixed-totals} now gives
\[
 \Theta|_{V_0}=I+\Delta,
 \qquad \Delta|_{A_{\mathbb Q}}=0,
 \qquad \Delta(V_0)\subseteq A_{\mathbb Q}.
\]

The normal form in
Lemma~\ref{lem:twisted-homology-normal-forms} gives a
non-degenerate \(2g\)-dimensional handle complement to \(R\). The
restrictions to \(A_{\mathbb Q}\) of the intersection covectors
of its positive-holonomy handle curves have common annihilator \(R\).
Their invariance under \(\Theta\) therefore
implies \(\Delta(V_0)\subseteq R\). Write
\[
 \Delta=\sum_{i,j}\Delta_{ij}d_i\otimes\lambda_j.
\]

For \(y\in F\), let \(\mu_i(y)\in\{1,-1\}\) be its coefficient sign
at the \(i\)-th positive boundary. Naturality gives
\[
 y_*d_i=\mu_i(y)d_i,
 \qquad \lambda_i y_*=\mu_i(y)\lambda_i.
\]
On \(V_0\), a central correction has the form \(I+D_y\), with
\(D_y\) in the span of the diagonal operators
\(d_i\otimes\lambda_i\). Since both \(D_y\) and \(\Delta\) vanish on
absolute homology, the equivariance equation gives
\[
 \Delta y_*-y_*\Delta=y_*D_y.
\]
Its off-diagonal coefficients satisfy
\[
 \bigl(\mu_j(y)-\mu_i(y)\bigr)\Delta_{ij}=0
 \qquad\text{ for }i\neq j.
\]

For each \(i\), choose a separating curve enclosing \(B^-_1\) and
\(B_i^+\), and no other boundary component. Its
holonomy is negative. Normalize the local-system lift of its twist
to be the identity on the side containing the other positive boundaries.
By Lemma~\ref{lem:twisted-picard-lefschetz}, this gives
\(\mu_i=-1\) and \(\mu_j=1\) for \(j\neq i\). Changing the lift
changes all signs at once and does not affect the displayed equation.

For a fixed ordered pair \(i\ne j\), this choice gives
\(\mu_j-\mu_i=2\), up to a common sign. Over \(\mathbb Q\), the coefficient equation therefore forces \(\Delta_{ij}=0\).
Varying \(i\) eliminates every off-diagonal coefficient. Consequently,
\begin{equation}\label{eq:two-negative-diagonal-correction}
 \Theta|_{V_0}=I+\sum_i a_i d_i\otimes\lambda_i.
\end{equation}

We now retain the integral information. Use the two-negative normal
form in \eqref{eq:two-negative-absolute-normal-form} and write
\[
 c_2-c_1=d_1+\cdots+d_{b_+}=:d,
 \qquad
 A_{\Z}=H_{\Z}\oplus\bigoplus_{i=1}^{b_+}\Z d_i.
\]
Choose the collar coordinates \(q_x\) from
Lemma~\ref{lem:integral-collar-gluing} with
\(B^-_1\) as the base boundary. 
An expression with absolute part
\(a\in A_{\mathbb Q}\) and endpoint coefficients \(n_x\) is integral
if and only if
\begin{equation}\label{eq:two-negative-integral-lattice}
 n_x\in\Z,
 \qquad \sum_xn_x\in2\Z,
 \qquad a-\frac{n^-_2}{2}d\in A_{\Z},
 \qquad n^-_2=\sum_{x\in M^\circ\cap B^-_2}n_x.
\end{equation}

Let \(Q_k\) be the canonical negative block at \(B^-_k\). Projection of
\(K\) to \(Q_k\) gives
\[
 \Lambda_k=\bigoplus_{x\in M^\circ\cap B^-_k}\Z q_x,
 \qquad
 Q_k\cap K=
 \left\{\sum_{x\in M^\circ\cap B^-_k}n_xq_x\ \middle|\
             \sum_{x\in M^\circ\cap B^-_k}n_x\in2\Z\right\}.
\]
These are the intrinsic lattices from the first part of the proof. If \(\Theta\) exchanges
\(Q_1\) and \(Q_2\), the image of any \(q_x\) with
\(x\in M^\circ\cap B^-_1\)
has integral coordinates and odd total coefficient in \(Q_2\).

It remains to control the extra endpoints on a positive boundary. Let
\(W_j\) be the sum of the non-trivial primary spaces of its rotation.
For two endpoints \(y,y'\) on \(B_j^+\), set
\[
 U_j=\operatorname{span}_{\mathbb Q}
 \{d_j,\ q_y-q_{y'}\mid y,y'\in M^\circ\cap B_j^+\}.
\]
The full twist is trivial on \(U_j\), so the one-step rotation
has finite order there and is semisimple. Its endpoint quotient
is the zero-total cyclic module, which has no invariant vector.
Hence its invariant subspace is exactly \(\mathbb Qd_j\).
Consequently,
\(
 U_j=\mathbb Qd_j\oplus W_j\). In particular, \(W_j\) has no absolute boundary coordinate along
\(d_i\) for \(i\neq j\). Decompose
\(q_y=v_y+w_y\), with \(v_y\in V_0\) and \(w_y\in W_j\). The vector
\(v_y\) has endpoint-total vector equal to the \(j\)-th standard
basis vector. The positive block at \(B_j^+\) has already been matched,
so \(\Theta(W_j)=W_j\). Equation
\eqref{eq:two-negative-diagonal-correction} now gives
\(
 \Theta q_y-q_y\in U_j\). Thus \(\Theta q_y\) has total endpoint coefficient one on \(B_j^+\), no
endpoint coefficients on the other positive boundaries, and no
absolute coordinate along \(d_i\) for \(i\neq j\).

Choose \(x\in M^\circ\cap B^-_1\) and an endpoint \(y\) on \(B_j^+\).
The vector
\(q_x+q_y\) is integral by
\eqref{eq:two-negative-integral-lattice}. If the negative blocks were
exchanged, its image would have odd total coefficient on \(B^-_2\).
The same integral lattice condition would force its absolute
coefficient along every \(d_i\) to be \(1/2\) modulo \(\Z\). For
\(i\neq j\), the preceding arguments show that this coefficient is
zero. Such an \(i\) exists because \(b_+\geq2\). This contradiction
proves that the negative blocks cannot be exchanged, completing
\textup{(ii)}.
\end{proof}

\subsection{Exceptional represented isomorphisms}
\label{appsubsec:exceptional-isomorphisms}

We first record the extension step used in all three examples. Let
\(\mathbb S=(S,M,\eta)\) and \(\mathbb S'=(S',M',\eta')\) satisfy the hypotheses of
Theorem~\ref{thm:opper-geometrization}. Assume that all boundary
windings are non-zero and that both line-field stabilizers fix a
chosen stop. Normalizing graded lifts at these stops gives
\[
 D(\mathbb S)\cong
 \bigl(H^1(S;\kk^\times)\rtimes H_\eta\bigr)
 \times\Z\langle[1]\rangle,
\]
and likewise for \(\mathbb S'\).

\begin{lemma}\label{lem:normalized-representation-extension}
Suppose that \(\phi:H_\eta\to H_{\eta'}\) is an isomorphism and that
\(\Theta:K(\mathbb S)\to K(\mathbb S')\) intertwines the normalized
stabilizer representations. If there is an isomorphism
\(h:H_1(S;\Z)\to H_1(S';\Z)\) satisfying
\(hy_*=\phi(y)_*h\) for every \(y\in H_\eta\), then \((\phi,\Theta)\)
extends to a shift-preserving isomorphism of represented invariants.
\end{lemma}

\begin{proof}
The non-zero windings eliminate the additive factors in the
geometrisation kernel of Theorem~\ref{thm:opper-geometrization}.
Normalization at the fixed stop splits the grading extension.
Choose a splitting of the geometrisation sequence and restrict it
to the normalized copy of \(H_\eta\) in the graded mapping class
group. Together with the actual central shift \([1]\), this gives
the displayed decomposition of \(D(\mathbb S)\). We also make the same
choices for \(\mathbb S'\).
Define
\[
 h^\vee:H^1(S;\kk^\times)\longrightarrow H^1(S';\kk^\times),
 \qquad h^\vee(\chi)=\chi\circ h^{-1}.
\]
The assumed ordinary homological equivariance makes \(h^\vee\)
equivariant for \(\phi\). It therefore extends \(\phi\) to the
semidirect products in the displayed decomposition. By
Lemma~\ref{lem:geometrization-charge-action}, the local-system torus
acts trivially on the charge lattices, while the shift acts by \(-I\). Sending
\((\chi,y,n)\) to \((h^\vee(\chi),\phi(y),n)\) and using \(\Theta\)
on the charge lattice gives the required represented isomorphism.
\end{proof}

\begin{lemma}\label{lem:exceptional-representation-calculations}
In each of the following cases, the formula for \(\phi\) defines an
isomorphism of the stopped line-field stabilizers and admits an integral
intertwiner of the normalized charge representations:
\begin{enumerate}[label=\textup{(\roman*)}]
\item the two-boundary construction of
Example~\ref{ex:two-boundary-represented-failure};
\item the three-boundary construction of
Example~\ref{ex:three-boundary-represented-failure};
\item the four-boundary construction of
Example~\ref{ex:four-boundary-represented-failure}.
\end{enumerate}
Consequently, the two represented invariants in each example are
isomorphic.
\end{lemma}

\begin{proof}
We retain the notation of the respective examples. In each case the
source \(\mathbb S\) and target \(\mathbb S'\) have the same underlying
surface and orientation local system, and we use common notation for
absolute homology. With the line fields, stop counts, characters, and
explicit formulas specified there, we verify the root relations and
integral equivariance separately in each case.

In all three cases, the central correction on \(F\) is invertible.
Its defining characters vanish on the correcting central elements,
so negating the correction exponents gives the inverse. We also check
the root relations and the induced maps on the finite quotients by \(F\).

\medskip
\noindent\emph{Proof of \textup{(i):} two boundary components.}
The root relation is preserved because
\[
 \phi(\rho_2)^m=T_1^{-q^2}T_2^{2m}=\alpha(T_2).
\]
On the quotient by \(F\), the map is multiplication by
\(-q^2\equiv-1\pmod m\), so it is an isomorphism. It induces the
identity after the boundary twists are forgotten and hence on ordinary
homology.

The orientation local system is positive on the handles and negative
on both boundaries. Equation~\eqref{eq:two-negative-absolute-normal-form}
then shows that the two negative boundary chains agree.
By Lemma~\ref{lem:integral-collar-gluing}, after exchanging the source
and target boundary blocks, both charge lattices have the model
\[
 K=A_{\Z}\oplus
 \left\{\sum_x n_xq_x:n_x\in\Z,\quad
                         \sum_x n_x\equiv0\pmod2\right\}.
\]
Normalize source lifts at the unique stop on \(B_1\) and target lifts
at the unique stop on \(B_2\). For \(y\in F\), relative displacement gives the signs
\[
 \begin{array}{c|cc}
  &B_1&B_2\\\hline
  \text{source}&1&(-1)^{\beta(y)}\\
  \text{target}&(-1)^{\beta(y)}&1.
 \end{array}
\]
The lift of \(k\) which is the identity off its boundary collars changes
both endpoint signs and fixes absolute homology. After normalization it
acts by \(-I\) on absolute homology and trivially on the endpoint
blocks. By
Lemma~\ref{lem:relative-displacement}, changing the normalization of \(y\) from \(B_1\) to \(B_2\)
multiplies its absolute action by \((-1)^{\beta(y)}\). The factor \(k^{\beta(y)}\)
contributes the same sign and cancels this discrepancy. Thus the target
action of \(\alpha(y)\) agrees with the source action of \(y\) on
absolute homology, while the two endpoint sign characters agree after
exchanging the blocks. Choose cyclic bases
\(e_0,\ldots,e_{m-1}\) and \(e'_0,\ldots,e'_{m-1}\) for the non-trivial
source and target root blocks, and denote the corresponding signed cycles
by \(R\) and \(R'\). Then \(R^m=(R')^m=-I\). Since
\(q^2\equiv1\pmod{2m}\), its target image acts as \((R')^{-1}\).
The signed reversal
\[
 e_0\longmapsto e'_0,
 \qquad e_j\longmapsto-e'_{m-j}\quad\text{ for }1\leq j<m
\]
intertwines these cycles. Exchanging the two endpoint blocks by this
map, and taking the identity on \(A_{\Z}\), preserves the even-total
lattice. Lemma~\ref{lem:normalized-representation-extension} gives the
required isomorphism of represented invariants.

\medskip
\noindent\emph{Proof of \textup{(ii):} three boundary components.}
Here
\[
 \phi(\rho_2)^m=T_1^{-q}T_2^{2m}=\alpha(T_2),
 \qquad -q\equiv1\pmod m.
\]
The map is therefore an isomorphism. Its matrix on the first two
rotation directions has determinant
\(2m(1-q)+q^2=1\), and it fixes the third direction.

Let \(c_1,c_2\) be the two negative boundary chains and put
\(d=c_2-c_1=[B_3]_L\), up to the common collar convention. Write
\[
 A_{\Z}=H_{\Z}\oplus\Z d,
 \qquad
 K_{\mathbb Q}=H_{\mathbb Q}\oplus\mathbb Qd
       \oplus\mathbb Qz\oplus E_1\oplus E_2,
\]
for the source, where \(H_{\mathbb Q}=H_{\Z}\otimes\mathbb Q\). The
target has the analogous decomposition with \(E_1\) and \(E_2\)
exchanged. Choose the source collar coordinates
\(q_x\) from Lemma~\ref{lem:integral-collar-gluing} with \(B_1\) as
the base boundary. Here
\[
 E_i=\operatorname{span}_{\mathbb Q}
 \{q_x:x\in M^\circ\cap B_i\}
 \quad\text{ for }i=1,2,
\]
and \(x_3\) is the unique endpoint on \(B_3\), with collar coordinate
\(z=q_{x_3}\).
An element of \(K_{\mathbb Q}\) has a unique expression
\[
 x_H+ad+n_3z+
 \sum_{x\in (M^\circ\cap B_1)\sqcup(M^\circ\cap B_2)}n_xq_x,
 \qquad
 x_H\in H_{\mathbb Q},\quad a,n_3,n_x\in\mathbb Q.
\]
Lemma~\ref{lem:integral-collar-gluing} shows that this element is
integral exactly when
\begin{equation}\label{eq:three-boundary-integral-conditions}
 \begin{gathered}
 x_H\in H_{\Z},\qquad n_3,n_x\in\Z,\qquad
 n_1+n_2+n_3\equiv0\pmod2,\\
 a-\frac{n_2}{2}\in\Z,
 \qquad n_i=\sum_{x\in M^\circ\cap B_i}n_x\quad\text{ for }i=1,2.
 \end{gathered}
\end{equation}
Normalize both sides at the endpoint on \(B_3\). The lift of
\(k=T_1^{-1}T_2\) is \(-I\) on \(E_1\oplus E_2\) and the identity on
\(
 U=H_{\mathbb Q}\oplus\mathbb Qd\oplus\mathbb Qz\), which is the image of
\(H_1(S,\{x_3\};L)\). The two inclusions are \(F\)-equivariant for
the common normalization, so the source and target actions agree on
\(U\). For \(y\in F\), let \(\delta_i(y)\) be the scalar by which this
normalized lift acts on \(E_i\). Relative displacement gives
\(\delta_1(y)\delta_2(y)=(-1)^{h(y)}\). Since \(q\) is odd, the factor
\(k^{qh(y)}\) exchanges the two endpoint sign characters. The full
twist at \(B_3\) has the form \(I+N_3\), where
\[
 N_3(z)=\varepsilon_\partial d,
 \qquad N_3(H_{\mathbb Q}\oplus\mathbb Qd)=0.
\]
It commutes with the stabilizer action. Since
\(\varepsilon_\partial^2=1\), the operator
\(I+\varepsilon_\partial N_3/2\) fixes
\(H_{\mathbb Q}\oplus\mathbb Qd\) and sends \(z\) to \(z+d/2\).
Composing it with the exchange of \(E_1\) and \(E_2\) therefore
intertwines \(F\) with \(\alpha(F)\). On the coordinates in
\eqref{eq:three-boundary-integral-conditions}, it sends
\[
 (n_1,n_2,n_3,a)\longmapsto(n_2,n_1,n_3,a+n_3/2).
\]
The even-total condition shows that this map and its inverse are
integral. Let \(E'_1\) be the target endpoint block at \(B_1\). The root
actions on \(E_2\) and \(E'_1\) are signed cycles \(R,R'\) with
\(R^m=(R')^m=-I\), and
\[
 \varrho_{\mathbb S'}\bigl(\phi(\rho_2)\bigr)|_{E'_1}
 =(R')^{-q}=R'.
\]
Thus the root actions are also intertwined. The coarse homological
action is unchanged, so
Lemma~\ref{lem:normalized-representation-extension} applies.

\medskip
\noindent\emph{Proof of \textup{(iii):} four boundary components.}
The root relations follow from \(1+q_m=2m\) and \(1-q_n=-2n\).
The induced maps on the cyclic quotients are invertible because
\[
 q_m^2\equiv1\pmod{2m},
 \qquad q_n^2\equiv1\pmod{2n}.
\]
Thus \(\phi\) is an isomorphism and induces the identity on the
common image of the two stabilizers in \(\Mod(S)\).

Choose a core vertex \(v\) and negative boundary chains with
\[
 c_4=c_1-c_2+c_3,
 \qquad \partial c_i=-2v.
\]
For
\(x\in M^\circ\cap B_i\), write \(i(x)=i\) and choose the collar
coordinate \(q_x\) from Lemma~\ref{lem:integral-collar-gluing}
with \(B_1\) as the base boundary. For endpoint coordinates
\(\mathbf n=(n_x)\), put
\[
 \Lambda_{\mathrm{ev}}=
 \left\{\mathbf n\in\Z^{M^\circ}:\sum_xn_x\equiv0\pmod2\right\},
 \qquad
 \Gamma(\mathbf n)=\frac12\sum_xn_x(c_{i(x)}-c_1).
\]
Lemma~\ref{lem:integral-collar-gluing} identifies the charge lattice as
\begin{equation}\label{eq:four-negative-integral-gluing}
 K=\left\{a+\sum_xn_xq_x:\begin{array}{l}
 a\in A_{\mathbb Q},\quad \mathbf n\in\Lambda_{\mathrm{ev}},\\
 a-\Gamma(\mathbf n)\in A_{\Z}
 \end{array}\right\}.
\end{equation}
Let \(\Gamma'\) be the corresponding correction term for the target,
again based at \(B_1\). Its charge lattice has the same form with
\(\Gamma\) replaced by \(\Gamma'\).
For \(\pi=(12)(34)\), let \(\bar c_i\) denote the class represented by
the reduction of \(c_i\) in \(H_1(S;\mathbb F_2)\). Then
\[
 \bar c_{\pi(i)}-\bar c_i=\bar c_1+\bar c_2
 \qquad\text{for every }i.
\]
Let \(\Pi\) be any signed source-to-target endpoint permutation inducing
\(\pi\). Since
vectors in \(\Lambda_{\mathrm{ev}}\) have even total, we have
\[
 2\bigl(\Gamma'(\Pi\mathbf n)-\Gamma(\mathbf n)\bigr)
 \equiv
 \left(\sum_x n_x\right)(\bar c_1+\bar c_2)=0\pmod2.
\]
The integral chain on the left is a cycle which vanishes in
\(H_1(S;\mathbb F_2)\). The universal coefficient injection
\[
 A_{\Z}/2A_{\Z}\longrightarrow H_1(S;\mathbb F_2)
\]
therefore shows that
\(2(\Gamma'(\Pi\mathbf n)-\Gamma(\mathbf n))\in2A_{\Z}\). Thus
\(\Gamma'(\Pi\mathbf n)-\Gamma(\mathbf n)\in A_{\Z}\), and \(\Pi\),
extended by the identity on \(A_{\mathbb Q}\), carries the source lattice
in \eqref{eq:four-negative-integral-gluing} onto the target lattice.

Since all four boundaries are negative, there is an \(F\)-stable
decomposition
\[
 K_{\mathbb Q}=A_{\mathbb Q}\oplus Q_1\oplus Q_2\oplus Q_3\oplus Q_4.
\]
For \(y\in F\), choose one of its two lifts to \(L\) and use it in both
representations. Let \(s_i(y)\in\{1,-1\}\) be its boundary sign on
\(Q_i\). The complement-fixed lifts of \(k_{12}\) and
\(k_{34}\) have sign vectors
\[
 (-1,-1,1,1),
 \qquad (1,1,-1,-1).
\]
Lemma~\ref{lem:relative-displacement} then gives
\[
 (-1)^{h_{12}(y)}=s_1(y)s_2(y),
 \qquad
 (-1)^{h_{34}(y)}=s_3(y)s_4(y).
\]
Normalize the source at the unique stop on \(B_1\) and the target at
the unique stop on \(B_2\). Writing \(s_i=s_i(y)\), the source signs are
\(
 (1,s_1s_2,s_1s_3,s_1s_4)\),
while the signs of the target action of \(\alpha(y)\) are
\(
 (s_1s_2,1,s_1s_4,s_1s_3)\). They are intertwined by \(\pi=(12)(34)\). Before normalization, the
central corrections change the boundary-sign vector to
\((s_2,s_1,s_4,s_3)\) and act trivially on \(A_{\mathbb Q}\). The sign
at the target normalization point on \(B_2\) is therefore \(s_1\),
the same sign used to normalize the source at \(B_1\). The two
normalizations multiply the common absolute action by the same sign,
so the normalized absolute actions also agree.

On the non-trivial root blocks, the formulas for \(\phi\) act as the
inverses of the target signed cycles. Signed reversal conjugates each
cycle to its inverse. Together with \(\pi\), these reversals give an
integral intertwiner of the full stabilizer representations. The
ordinary homological action is unchanged, and
Lemma~\ref{lem:normalized-representation-extension} completes the
extension to \(\mathfrak R\).
\end{proof}

\section*{Acknowledgments}

The first author was supported by the Fundamental Research Funds for the
Central Universities (grant no.~GK202403003) and by the National Natural
Science Foundation of China (grant nos.~12271321 and 12671047). The second author would like to thank Shunsuke Kano for answering questions about mapping class groups and for helpful suggestions concerning references.

\section*{AI usage statement}
The authors thank ChatGPT for its assistance during the preparation of this article. After completing a full
draft, the authors used ChatGPT to proofread the text and audit
the arguments. This audit exposed a gap in the original form of
Theorem~\ref{thm:general-boundary-reconstruction} and suggested the
counterexample recorded as
Example~\ref{ex:two-boundary-represented-failure}. The authors subsequently
verified and rewrote the retained AI-assisted material. Both authors
have carefully checked this material and take full responsibility for
the correctness of the paper.

\end{document}